\documentclass[12pt]{amsart}
\usepackage{amsmath,amssymb,mathabx,amsthm,mathtools}
\usepackage{mathabx,overpic,enumitem}
\usepackage[margin=1.5in]{geometry}
\newtheorem*{thm*}{Theorem}
\newtheorem{thm}{Theorem}[section]
\newtheorem{prop}[thm]{Proposition}
\newtheorem{lem}[thm]{Lemma}
\newtheorem{cor}[thm]{Corollary}
\numberwithin{equation}{section}
\def\:{\colon}
\newcommand{\C}{\mathbb{C}}
\newcommand{\R}{\mathbb{R}}
\newcommand{\D}{\mathbb{D}}
\newcommand{\N}{\mathbb{N}}
\newcommand{\Z}{\mathbb{Z}}
\newcommand{\Q}{\mathbb{Q}}
\def\H{\mathbb{H}}
\newcommand{\bC}{\widehat{\mathbb{C}}}

\newcommand{\ga}{\gamma} 
\newcommand{\Ga}{\Gamma}

\newcommand{\la}{\lambda} 
\newcommand{\ra}{\rightarrow} 
\newcommand{\sub}{\subseteq} 
\newcommand{\im}{\operatorname{Im}}
\newcommand{\re}{\operatorname{Re}}
\newcommand{\inte}{\operatorname{inte}}

\newcommand{\PSL}{\operatorname{PSL}}
\newcommand{\wC}{  \widehat {\mathbb{C}}}
\newcommand{\al}{\alpha}
\newcommand{\be}{\beta}
\newcommand{\de}{\delta}
\newcommand{\sig}{\sigma}
\begin{document}
\title{Conformal maps and critical points of Eisenstein series}
\author{Mario Bonk}
\thanks{The author was partially supported by  NSF grant DMS-2054987.}
\address {Department of Mathematics, University of California, Los Angeles,
CA 90095}
\email{mbonk@math.ucla.edu}
%
\date{June 26, 2025}
\begin{abstract} 
We investigate the critical points of the basic (quasi-)modular forms $E_2$, 
$E_4$, and $E_6$. They occur where  some associated
polymorphic functions have poles. By an explicit  description of these  polymorphic functions as conformal maps, one can give an accurate qualitative analysis of the locations 
of the critical points of $E_2$, 
$E_4$, and $E_6$. \end{abstract}

\maketitle

\section{Introduction} The critical points of  Eisenstein series have gained some recent interest (see \cite{CL19, GO22, CL24, Ch25} for  results and some  motivations for these investigations).  It is known that the location of these critical points is the same as the location of the poles of some associated polymorphic 
functions (see \cite{SS12}). The main point of this paper is 
to show  that these  polymorphic functions allow some explicit description as conformal maps 
 in the case of the lowest-weight Eisenstein series $E_2$, 
$E_4$, and $E_6$. This in turn can be used to locate their  critical points at least on a qualitative level.

To  each of  the  Eisenstein series $E_2$, 
$E_4$, $E_6$ we associate a certain polymorphic function $s^\pm_2$, $s_4$, $s_6$, respectively. As we will see, each of them is a  conformal map
of  a circular arc triangle arising from the   modular triangulation of the upper half-plane to some associated triangle. For example, $s_4$ maps the basic modular triangle with angles $0$ at the vertices $0$, $1$, $\infty$ to the complement of its complex  conjugate.  

There are two methods  to prove this.  The first one is based on the Argument Principle and some easy-to-establish mapping properties of the functions on the sides of some relevant circular arc triangles.
While this leads to  elementary proofs with a moderate amount of  computations, the second method is more involved, but offers additional insights.  Namely, we will consider each relevant polymorphic function $f$ as a function of an  appropriately chosen modular function such as the basic modular function $J$. To establish the desired
mapping properties of $f$, one can then just  compute  Schwarzian derivatives 
(such as $\{f,J\}$). As the computation of these Schwarzians turns out to be rather involved if attempted  directly, we will use a classical detour by establishing a  Fuchsian equation for some auxiliary functions whose quotient is equal to $f$.  
We can then derive the desired mapping properties of $f$ from some classical results,
ultimately  going  back  to Schwarz \cite{Schw}.  The basic philosophy of this approach is explained
in my expository paper \cite{BZeta}, where the quasi-periods of the Weierstrass $\zeta$-function were studied from a similar perspective.

In this paper, $\tau$   will usually denote a complex variable in the open upper half-plane $\H\coloneq\{ \tau\in  \C: \im(\tau)>0\}$. 
We set 
 \begin{equation} q=q(\tau)\coloneq e^{2\pi i \tau}. \end{equation}
 Note that $|q|<1$ for $\tau\in \H$. 
We define
\begin{equation}
\sigma_k(n)\coloneq \sum_{m\in \N,\, m|n}m^k
\end{equation}
 for $k,n\in \N$.
So $\sigma_k(n)$  is the sum of  $k$th-powers $m^k$ of all natural numbers $m$ that divide $n$.

The classical (normalized) Eisenstein series are then given by
  \begin{align}\label{E2} E_2(\tau)&\coloneq 1-24\sum_{n=1}^\infty  \frac{nq^n}{1-q^n}=1-24\sum_{n=1}^\infty \sigma_1(n)q^n,\\
  \label{E4}
E_4(\tau)&\coloneq  1+240 \sum_{n=1}^\infty  \frac{n^3q^n}{1-q^n}
 = 1+240 \sum_{n=1}^\infty \sigma_3(n)q^n, \\ \label{E6}
  E_6(\tau)&\coloneq  1-504 \sum_{n=1}^\infty  \frac{n^5q^n}{1-q^n}=
  1-504 \sum_{n=1}^\infty \sigma_5(n)q^n. 
   \end{align}

It is immediate that these series (as functions of $\tau\in \H)$ converge locally uniformly on $\H$. Hence $E_2, E_4, E_6$ are holomorphic   functions on $\H$. To ease notation, we often drop the argument $\tau$ of these functions if it is understood. 

 In order to locate the critical points 
of $E_2, E_4, E_6$, we associate with them  auxiliary functions defined as  
\begin{align} 
 s^\pm_2(\tau)
&\coloneq \tau-\frac{6i}{\pi(E_2\pm E_4^{1/2})}, \label{eq:s2def}\\
s_4(\tau)&\coloneq  \tau-\frac{6iE_4}{\pi(E_2E_4-E_6)},
 \label{eq:s4def}\\
s_6(\tau)&\coloneq  \tau-\frac{6i E_6}{\pi(E_2E_6-E_4^2)},
 \label{eq:s6def}
\end{align}
respectively.

Then  $s_4$ and $s_6$ are meromorphic functions on 
$\H$. For $E_2$ we define two functions  $s^+_2$ and $s^-_2$
according to  the sign in \eqref{eq:s2def}. Here   we have to be careful about the choice of the root $E_4^{1/2}$. Formula~\eqref{E4} shows that $E_4$ takes positive values on the positive imaginary axis. We choose the sign ambiguity in the definition of  $E_4^{1/2}$ so that it also takes 
positive values on the positive imaginary axis. We can analytically continue $s^+_2$ and $s^-_2$ as  germs of  meromorphic functions along  any 
path in $\H$ that avoids the zeros of $E_4$.  If we run around a zero of $E_4$, then $s^+_2$ transforms into $s^-_2$ and vice versa.

The properties of  the functions $s^+_2$ and $s^-_2$ 
 are  best understood 
 if one considers them as branches of a  multivalued function $s_2$ defined as 
 \begin{equation}\label{eq:defs2} 
\H\ni \tau\mapsto s_2(\tau)\coloneq \{ s_2^+(\tau),  s_2^-(\tau)\}\sub \wC\coloneq \C\cup \{\infty\}
 \end{equation}
(see Section~\ref{s:equi} for more discussion). So for each $\tau \in \H$, the value  
$s_2(\tau)$ is subset of the Riemann sphere $\wC$ consisting of at most two elements and  reducing to a singleton set if and only if  $E_4(\tau)=0$.

The functions $s_2$, $s_4$, $s_6$ are {\em polymorphic} in the sense that they undergo simple transformations if their argument is changed 
by an element of the {\em extended modular group} $\overline \Ga$. This group is generated by  the reflection in the imaginary axis
    and the elements of   the {\em modular group}  $\Ga\coloneq \PSL_2(\Z)$ consisting of all  M\"obius transformations of the form 
 \begin{equation}\label{Smodgr}
\tau\mapsto  \varphi(\tau) = \frac{a\tau+b}{c\tau+d},
\end{equation}
 where $a,b,c,d\in \Z$ with $ad-bc=1$ 
 (see  Section~\ref{s:equi} for more discussion).  
 Actually, we have 
 \begin{equation}\label{eq:poly46}
s_k \circ \varphi = \varphi\circ s_k
\end{equation}
for all $k\in \{2,4,6\}$ and $\varphi\in \overline \Ga$, where for 
$k=2$ this has to interpreted properly (see Propositions~\ref{prop:poly46} and ~\ref{prop:poly2}). 
 
 The location of the critical points of our Eisenstein series is the same as the location of the poles of the associated polymorphic function. Indeed, for $\tau\in \H$
  we have (see Lemmas~\ref{lem:critpol46} and ~\ref{lem:critpol2})
 \begin{equation}\label{eq:critpoles46}
E_k'(\tau)=0 \text{ if and only if } s_k(\tau)=\infty
\end{equation}
for  $k=4,6$  and   
 \begin{equation}\label{eq:critpoles2}
E_2'(\tau)=0 \text{ if and only if } \infty\in s_2(\tau).
\end{equation}

 Now the mapping behavior of the functions $s_2$, $s_4$, $s_6$ can be described explicitly. This  is of independent interest and gives a good understanding of the location 
 of the poles of these functions and hence the location of the critical points 
 of the Eisenstein series $E_2$, $E_4$, $E_6$.  In order to formulate our results,
 we first fix  some more notation and terminology.
 
 \begin{figure}[t]
 \begin{overpic}[ scale=0.8
    ]{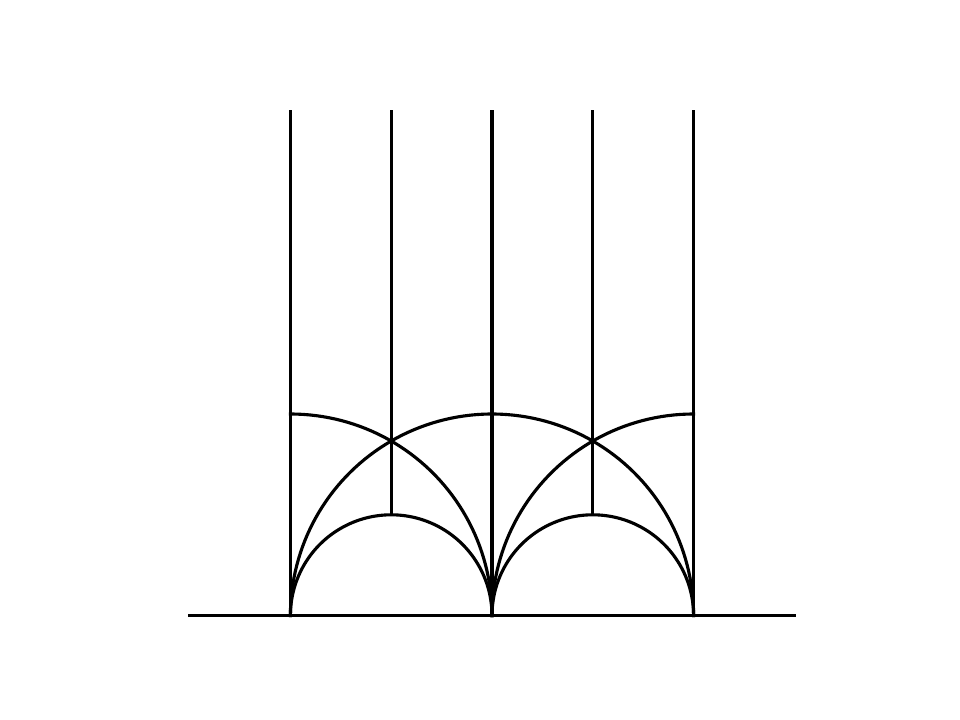}
    \put(51,33){ $i$}
    \put(61.5,31.5){ $\rho$}
     \put(72,12){ $1$}
       \put(52,12){ $0$}
        \put(52,45){ $T_0$}
        \put(52,27){ $T_1$}
         \put(57,23){ $T_2$}
           \put(62,23){ $T_3$}
               \put(67,27){ $T_4$}
         \put(67,45){ $T_5$}
            \put(75,55){ $V_0$}
           \end{overpic}
  \vspace{-1cm}
\caption{Some circular arc triangles in the tessellation 
$\mathcal{T}$.}
\label{fig:T}
\end{figure}
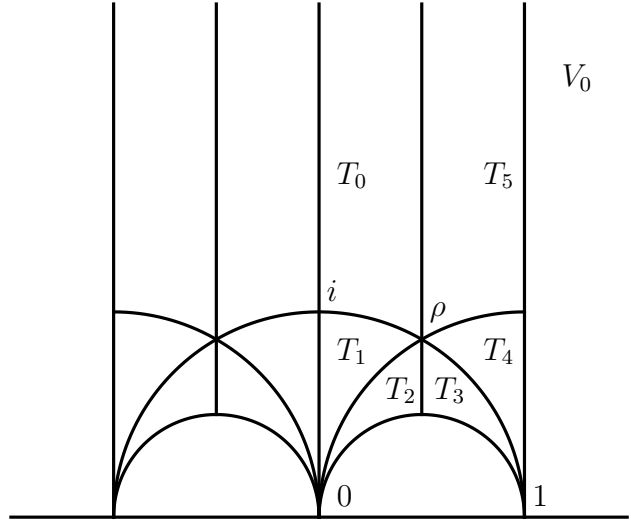

A {\em closed Jordan region} $X$ in the Riemann sphere $\wC$ is a compact set homeo\-morphic to a closed disk. Then its boundary $\partial X$ is a Jordan curve 
and the set of interior points $\inte(X)\coloneq X\setminus \partial X$ a simply connected open region. 

A {\em circular arc triangle} $T$ is a closed Jordan region in $\bC$ whose boundary is decomposed into  three non-overlapping circular arcs considered as the {\em sides} of $T$. The three endpoints of these arcs are the {\em vertices} of $T$.   At each vertex $v$  two sides of $T$ meet at an angle $\alpha \pi$, where $\alpha \in [0,2]$, as seen from inside $T$. In this case,  we say that  $T$ has {\em angle
$\alpha \pi$ at $v$}.

We say that $f$ is a {\em conformal map} between two circular arc triangles $X$ and $Y$ in  $\wC$ if $f$ is a homeomorphism between $X$ and $Y$ that sends the vertices of $X$ to the vertices of $Y$ and is a biholomorphism between $\inte(X)$ and  $\inte(Y)$. 
We say that $f$ is an {\em anti-conformal map} between 
$X$ and $Y$ if $z\in X\mapsto \overline{f(z)}\in \overline Y$ is a conformal map between $X$ and the  complex conjugate  $\overline Y\coloneq\{\overline w: w\in Y\}$ of $Y$ (where we set $\overline {\infty}=\infty$).

  We now consider the circular arc triangle 
\begin{equation}\label{T0}
T_0\coloneq \{ \tau\in \H:  0\le \re(\tau) \le 1/2, \, 
  |\tau|\ge 1\} \cup\{\infty\}
\end{equation}
in  $\wC$. It has the vertices $\infty$, $i$, $\rho$ with angles 
$0$, $\pi/2$, $\pi/3$, respectively, where 
\begin{equation}\label{rho}
\rho\coloneq \tfrac{1}{2} (1+i\sqrt 3).
\end{equation}

The reflections in the three sides of $T_0$ generate the extended modular group $\overline {\Ga}$ mentioned earlier.  The set 
\begin{equation}
\mathcal{T}\coloneq \{\varphi(T_0): \varphi \in 
\overline {\Ga}\}
\end{equation}
 of translates of  $T_0$ under $\overline {\Ga}$ form  a  tessellation of the
the {\em extended upper half-plane} 
\[
\H^*\coloneq\H\cup \Q\cup\{\infty\}
\]
 by circular arc triangles which also have  angles $0$, $\pi/2$, $\pi/3$ 
 at their vertices (see Figure~\ref{fig:T} for an illustration and \cite[Section 2.2]{Sch} for more discussion of these standard facts). 

We can now formulate our first main result.
\begin{thm}   \label{thm:s4} 
The map $s_4$ is a conformal map of the circular arc triangle $T_0$ onto the circular arc triangle $X_0$ bounded by 
the vertical line  from $-i\infty$ to $-i$, the circular arc
on the unit circle from $-i$ to $\rho$ in clockwise orientation, 
and the vertical line from to $\rho$ to $+i\infty$. 
Here $X_0$ lies on the left of its boundary with the orientation indicated, and the correspondence of vertices is 
$$ s_4(\infty)=\infty,  \quad  s_4(i)=-i, \quad  s_4(\rho)=\rho.$$    
 \end{thm}
  For an illustration of this conformal map see Figure~\ref{fig:s4}. In the formulation 
  and in the figure it is convenient to use $\pm i \infty$ to indicate how a part of a line 
 parallel to the imaginary axis is traversed and  oriented: for example, when we say that the vertical line is oriented from 
 $-i\infty$ to $-i$, then we mean that we traverse the imaginary axis from $\infty$ through points $z$ 
 near infinity with large  negative values for $\im(z)$ towards the point $-i$.  

 The equation $s_4(\infty)=\infty$ and similar equations below  have to be interpreted as a suitable  limit.  Here we have $s_4(\tau)\to \infty$ as $T_0\cap\H \ni \tau\to \infty $.   
 
 \begin{figure}[t]
 \begin{overpic}[ scale=0.7
    ]{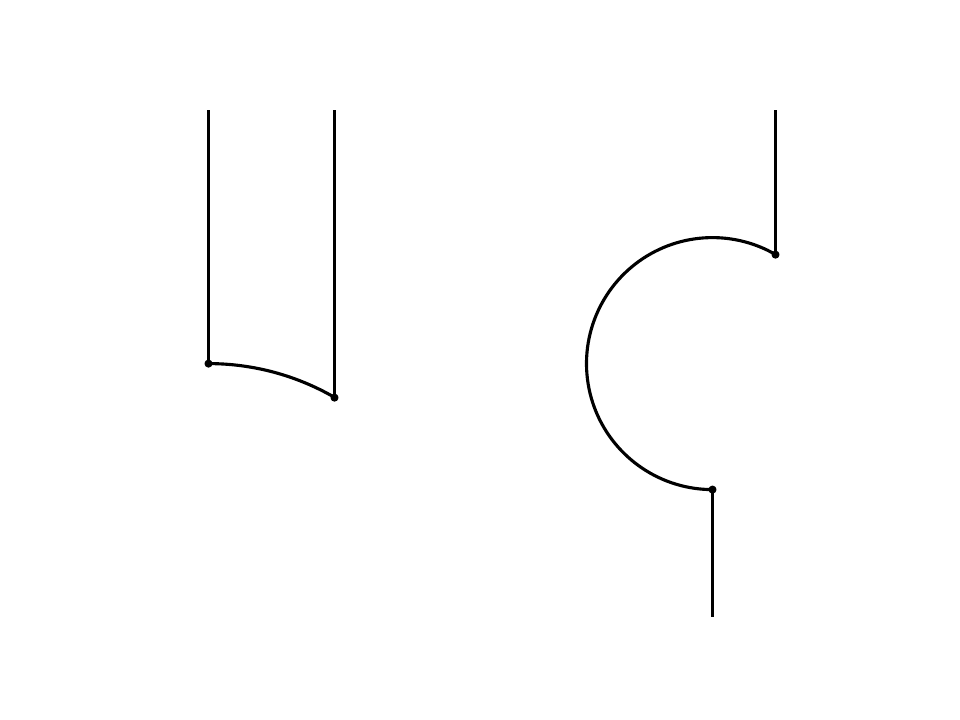}
        \put(23,65){ $+i \infty$}
    \put(17.5,36){ $i$}
    \put(35,32){ $\rho$}
      \put(25,48){ $T_0$}
          \put(21, 55){ $\downarrow$}
          \put(25,37){ $\rightarrow$} 
      \put(31, 55){ $\uparrow$} 
     \put(80.5,47.5){ $\rho$}
        
       \put(70,25.5){ $-i$}
 \put(65,58){ $X_0$}
         \put(49,43){$s_4$}
          \put(48,40){$\longrightarrow$}
       \put(70, 16){ $\uparrow$}
               \put(57, 37){ $\uparrow$}
             \put(77, 57){ $\uparrow$}
            \put(70.25, 7){ $-i \infty$}
                \put(77,65){ $+i \infty$}
                           \end{overpic}
  \vspace{-1cm}
\caption{Mapping property of $s_4$.}
\label{fig:s4}
\end{figure}
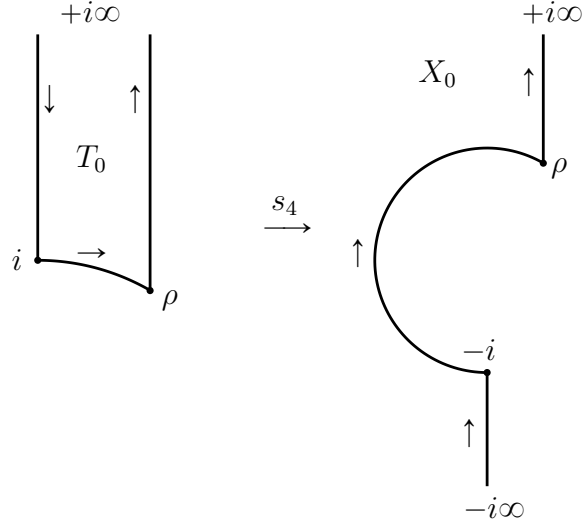

One obtains the mapping behavior of $s_4$ on other parts of $\H^*$ by applying Schwarz reflection. In order to provide a concise statement, we consider    
the circular arc triangle 
\begin{equation}\label{eq:V0}
V_0\coloneq \{\tau \in \H: 0\le \re(\tau)\le 1,\, |\tau-1/2|\ge 1/2\}\cup\{\infty\}
\end{equation} 
which has all its angles equal to $0$ at its  vertices $0$, $1$, $\infty$.
 It consists of a union of six triangles from the tessellation $ \mathcal{T}$ (see Figure~\ref{fig:T}).
We obtain another tessellation  $\mathcal{V}$ of $\H^*$ by successive reflections 
of $V_0$ in its sides.  More precisely, let  $\overline {\Ga}(2)$ be the group generated by the three reflections in the sides of $V_0$. Then 
$\mathcal{V}\coloneq\{\varphi(V_0): \varphi\in \overline {\Ga}(2)\}$. 

To give a simple  statement that describes the mapping behavior of $s_4$, it is convenient to consider the complex conjugate $\overline {s_4}$.

\begin{cor}\label{cor:s4} Let $V\in  
\mathcal{V}$. Then $\overline {s_4}$ is an anti-conformal map of the circular arc triangle  $V$  onto the complementary circular arc triangle $\wC\setminus \inte(V)$. 
Here the vertices of $V$ are fixed and each side of $V$ is mapped homeomorphically onto itself.   
\end{cor}

For $\tau\in \H$ we have $E_4'(\tau)=0$  iff $s_4(\tau)=\infty$ iff 
$\overline {s_4(\tau)}=\infty$ (see \eqref{eq:critpoles46}). We obtain the following consequence. 

\begin{cor}\label{cor:crit4} Let $V\in \mathcal{V}$. If  $\infty\in V$, then $V$  contains no critical point of $E_4$. If $\infty\not \in V$, then 
$V$ contains precisely one critical point of $E_4$. It is simple and lies in the interior of $V$.  
\end{cor}
Since  the triangles  $V\in  \mathcal{V}$ tessellate $\H^*$, this accounts for all 
critical points of $E_4$.

   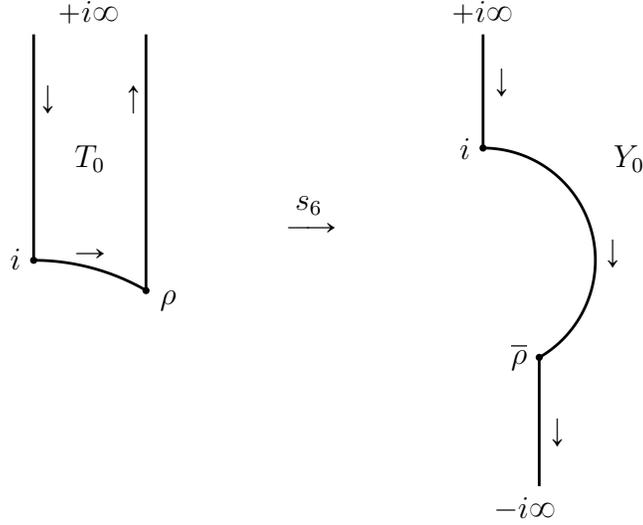
\begin{figure}[t]
 \begin{overpic}[ scale=0.7
    ]{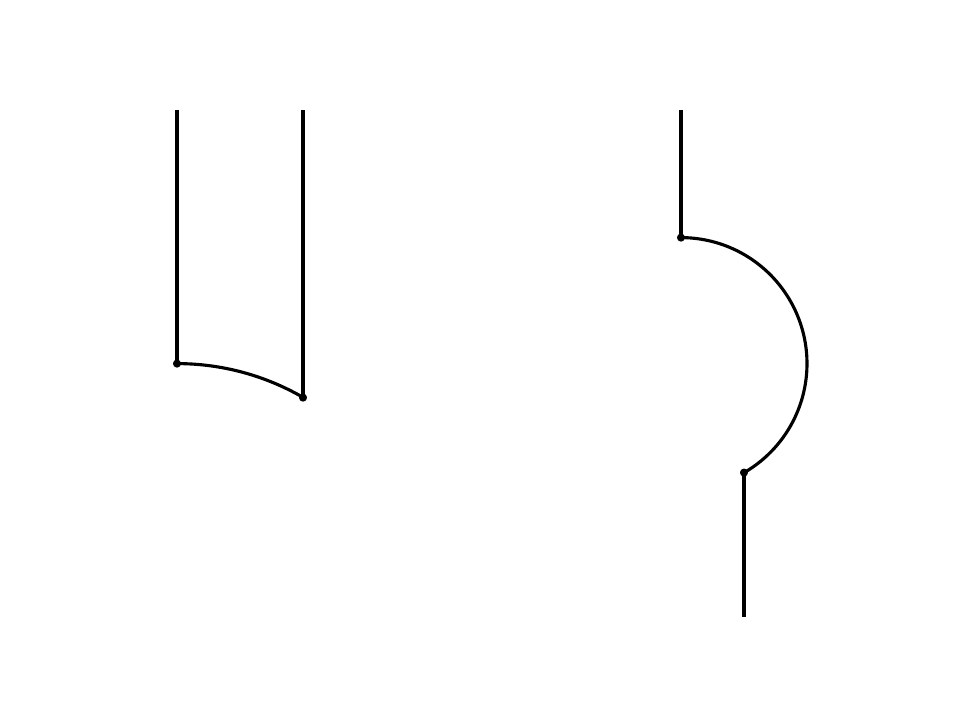}
        \put(20,65){ $+i \infty$}
    \put(14.5,36){ $i$}
    \put(32,32){ $\rho$}
      \put(22,48){ $T_0$}
          \put(18, 55){ $\downarrow$}
          \put(22,37){ $\rightarrow$} 
      \put(28, 55){ $\uparrow$} 
        \put(49,43){$s_6$}
          \put(48,40){$\longrightarrow$}
        \put(66,65){ $+i \infty$}
         \put(71, 57){ $\downarrow$}
          \put(85,48){ $Y_0$}
    \put(67,49){ $i$}
        \put(84, 37){ $\downarrow$}
        \put(73,25){ $\overline \rho$}
     \put(77.5, 16){ $\downarrow$}
            \put(71, 7){ $-i \infty$}
               \end{overpic}
  \vspace{-1cm}
\caption{Mapping property of $s_6$.}
\label{fig:s6}
\end{figure}

Similar results are true for  $s_6$.

\begin{thm} \label{thm:s6} 
The map $s_6$ is a conformal map of the circular arc triangle $T_0$ onto the circular arc triangle $Y_0$ bounded by 
the vertical line  from $+i\infty$ to $i$, the circular arc
on the unit circle from $i$ to $\overline \rho$ in clockwise orientation, 
and the vertical line from  $\overline \rho$ to $-i\infty$. 
Here $Y_0$ lies on the left of its boundary with the orientation indicated, and the correspondence of vertices is 
$$s_6(\infty)=\infty,    \quad s_6(i)=i, \quad  s_6(\rho)=\overline\rho. $$    
 \end{thm}
  
One can also formulate a statement related to the tessellation  $\mathcal{V}$.

\begin{cor}\label{cor:s6} Let $V\in  
\mathcal{V}$. Then on $\inte(V)$ the map $s_6$   attains each value in $\wC\setminus 
V$   twice counting multiplicities.      
\end{cor}

This again leads to a statement about the location of the critical points of $E_6$. 
\begin{cor} \label{cor:crit6} Let $V\in \mathcal{V}$. If  $\infty\in V$, then $V$  contains one critical point of $E_6$. If $\infty\not \in V$, then 
$V$ contains precisely two  critical point of $E_6$. These critical points  are simple and lie in the interior of $V$.  \end{cor}

To describe the  mapping behavior of $s^+_2$, we consider the circular arc triangle 
\begin{equation}\label{eq:U0}
U_0\coloneq \{ \tau\in  \H: 0\le \re(\tau)\le1/2,\, |\tau-1|\ge 1\}\cup\{0,\infty\}
\end{equation}
with vertices $\infty$, $0$, $\rho$ and corresponding angles 
$0$, $0$, $2\pi/3$. It is the union of $T_0$ and the copy of $T_1$ of $T_0$ obtained by its reflection in the unit circle (see Figures~\ref{fig:T} and~\ref{fig:s2}). If we choose the root of $E_4$ in the definition of $s_2^+$ so that $E_4^{1/2}$ takes positive values on the positive imaginary axis (as we discussed), then $s_2^+$ can be considered as a well-defined continuous function  on $U_0$ that is meromorphic in $\inte(U_0)$. With these choices the mapping behavior of 
$s^+_2$ admits the following description. 

\begin{thm} \label{thm:s2} 
The map $s^+_2$ is a conformal map of the circular arc triangle $U_0$ onto the circular arc triangle $Z_0$ bounded by 
the vertical line  from $+i\infty$ to $0$, the  arc
on the circle $\{\tau\in \C:|\tau-1|=1\}$ from $0$ to $\overline \rho$ in 
anti-clockwise orientation, and the vertical 
line from to $\overline \rho$ to $+i\infty$. 
Here $Z_0$ lies on the left of its boundary with the orientation indicated, and the correspondence of vertices is 
$$s^+_2(\infty)=\infty,    \quad s^+_2(0)=0, \quad  s^+_2(\rho)=\overline\rho .$$    
 \end{thm}
See Figure~\ref{fig:s2} for an illustration. 
A description of the mapping behavior of $s_2^-$ on $U_0$ is
 formulated in Corollary~\ref{cor:s2-}. 

 \begin{figure}[t]
 \begin{overpic}[ scale=0.7
    ]{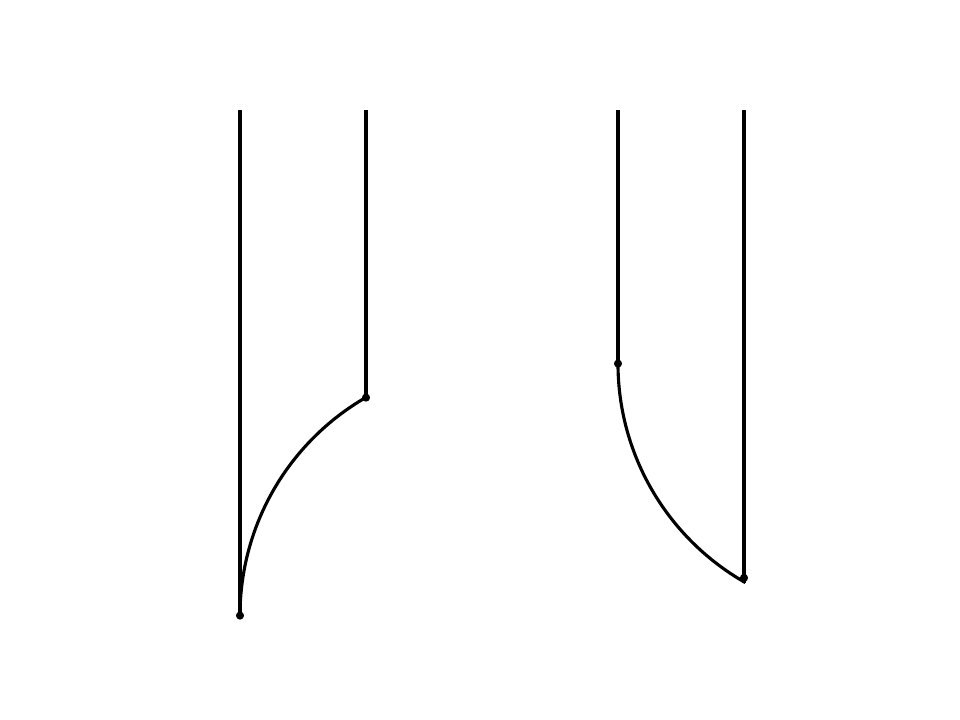}
        \put(26,65){ $+i \infty$}
    \put(20,10){ $0$}
    \put(38.5,33){ $\rho$}
      \put(29,41){ $U_0$}
          \put(25, 50){ $\downarrow$}
          \put(28,29){ $\nearrow$} 
      \put(34, 50){ $\uparrow$} 
        \put(49,43){$s^+_2$}
          \put(48,40){$\longrightarrow$}
        \put(66,65){ $+i \infty$}
         \put(64, 52){ $\downarrow$}
          \put(68,45){ $Z_0$}
    \put(60.5,36){ $0$}
        \put(74, 38){ $\uparrow$}
     \put(68, 22){ $\searrow$}
            \put(75, 11){ $\overline \rho$}
               \end{overpic}
  \vspace{-1cm}
\caption{Mapping property of $s^+_2$.}
\label{fig:s2}
\end{figure}
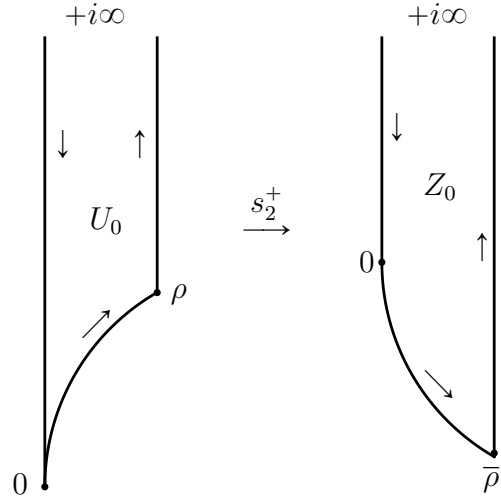

 As before, this can be translated to a statement about the critical points of the corresponding Eisenstein series. 
\begin{cor} \label{cor:crit2} Let $V\in \mathcal{V}$. If  $\infty\in V$, then $V$  contains no critical point of $E_2$. If $\infty\not \in V$, then 
$V$ contains precisely one  critical point of $E_2$. It  is simple and lies in the interior of $V$.  
\end{cor}

The statements about  the critical points of $E_2$ and $E_4$ (Corollaries~\ref{cor:crit2} and~\ref{cor:crit4}) were known before. In equivalent form they were formulated by Z.~Chen and 
C.-S.~Lin   for $E_2$  in \cite[Theorem~1.2]{CL19}  and for $E_4$ in  \cite[Theorem~1.2]{CL24}.  To prove these results, Chen and Lin used a method that is quite  different from our approach. Very recently, Chen \cite{Ch25} also studied the critical points of $E_6$ and obtained Corollary~\ref{cor:crit6}.

It seems that polymorphic functions (such as $s_k$ for $k=2,4,6$)  for the study of critical points of modular forms 
were first considered by H.~Saber and A.~Sebbar in  \cite{SS12}. S.~Gun and J.~Osterl\'e \cite{GO22} used this connection to polymorphic function in combination with the Argument Principle to obtain information about critical points of Eisenstein series. 

Our explicit description of the mapping properties of   
$s_k$ for $k=2,4,6$ seems to be new. From this one can easily gain a good understanding of the real locus of the functions $s_k$  that
is relevant for the precise location of the critical points of $E_k$ for $k=2,4,6$ and was studied in  \cite{CL19, CL24, Ch25, GO22}. We will remark more on this in 
Section~\ref{sec:concl}. Chen \cite{Ch25} raised the question whether there is a unified approach
to study the critical points of $E_k$ for $k=2,4,6$. Our results essentially provide a positive answer.  

The paper is organized as follows. In Section~\ref{sec:basic} we record some relevant facts about the Eisenstein series $E_2, E_4, E_6$.   In Section~\ref{s:equi} we establish the equivariance property \eqref{eq:poly46}
of our functions $s_k$, $k=2,4,6$. This goes back to  \cite{SS12}, but we include complete proofs for the sake of completeness. In Section~\ref{sec:critpoles} 
we study the relation of the critical points of our Eisenstein series  
and the poles of the associated polymorphic function (see Lemmas~\ref{lem:critpol46} and~\ref{lem:critpol2}). We also show that  
for $k=2,4, 6$ each  critical point of $E_k$ and 
each pole of $s_k$   is simple (Proposition~\ref{prop:Esimcrit} and Lemma~\ref{prop:simpol}). We also record explicit formulas for the derivatives $s'_k$ (Lemma~\ref{lem:derivsk}).    

In Section~\ref{sec:auxconf} we collect some auxiliary results about conformal mappings. Of crucial importance for the proof of our main results is the extension of the Argument Principle formulated in Proposition~\ref{prop:arg}.  
The proofs of Theorems~\ref{thm:s4},  \ref{thm:s6}, \ref{thm:s2} are based on this fact and are presented in Sections~\ref{sec:maps4}, \ref{sec:maps6}, \ref{sec:maps2}, respectively. In each of these sections we also describe the global  mapping properties of the corresponding function $s_k$ and  prove the corollaries 
formulated earlier. 

Sections~\ref{sec:diffeq46} and \ref{sec:diffeq2} are independent of the considerations in Sections~\ref{s:equi}--\ref{sec:maps2} and outline an 
approach to the proof of our main results based on associated second order differential equations whose fundamental solutions allow a representation of $s_k$ as a quotient (see \eqref{eq:s4arat} for $k=4$, \eqref{eq:s6arat} for 
$k=6$, and  \eqref{eq:s2arat}  for $k=2$). 
 For $s_4$ and $s_6$ we are led to hypergeometric equations 
 with the modular function $J$ (see \eqref{defJ}) as a variable (see Propositions~\ref{prop:aJ} and~\ref{prop:bJ}), while for $s_2$ we obtain  a Fuchsian equation 
with an algebraic function $u$ of $J$ as the variable (see \eqref{eq:zdef} 
and Proposition~\ref{prop:cz}). As a consequence we arrive at explicit expressions for the Schwarzian derivatives $\{s_k, \tau\}$ for $k=2, 4,6$ (see 
\eqref{eq:Schws4tau}, \eqref{eq:Schws6tau}, \eqref{eq:Schws2tau}). 
  
  The last section, Section~\ref{sec:concl}, is devoted to some additional remarks related to the themes in this paper.
 
 \smallskip \noindent
  {\bf Acknowledgments.} I'd like to thank Alex Eremenko, Misha Hlushchanka, and Daniel Meyer for some useful comments, and Zhijie Chen for  bringing the paper \cite{Ch25} to my attention.

\section{Basic facts about  Eisenstein series}\label{sec:basic}
 
In this section we record some basic  facts about the Eisenstein series $E_2$, $E_4$, $E_6$ defined in \eqref{E2}--\eqref{E6}.   All of this is well-known and can be found in \cite {Sch} or \cite{Zag}, for example.

Let  $\tau\in \H$. Then for $q=q(\tau)=\exp(2\pi i \tau)$ we have 
\begin{align*} 
q(\tau+1)&=\exp(2\pi i (\tau+1))=q(\tau) \text { and }\\ 
q(-\overline {\tau})&=\exp(-2\pi i \overline {\tau})=\overline
{\exp(2\pi i \tau)}=\overline{q(\tau)}. 
\end{align*}

This and the definitions show that  $E_k$ for $k=2,4,6$  is  $1$-periodic, that is, 
 \begin{equation}\label{eq:perE}
  E_k(\tau+1)=E_k(\tau)
 \end{equation} 
 and we also have 
  \begin{equation}\label{eq:symE}
  E_k(-\overline{\tau})=\overline {E_k(\tau)} 
 \end{equation} 
for $\tau\in \H$. 
 We also have  the identities 
 \begin{align} E_2(-1/\tau)&=\tau^2E_2(\tau)-\tfrac{6 i}{\pi}  \tau,\label{transE2'}\\
 E_4(-1/\tau)&=\tau^4E_4(\tau), \label{transE4}\\
 E_6(-1/\tau)&=\tau^6E_6(\tau). \label{transE6}
 \end{align}
 that are not so obvious. They are special cases of the general transformation formula 
 under elements of the modular group $\Ga= \PSL_2(\Z)$ (see \cite[Section~2]{Zag}). 
 \begin{prop} Let $\tau\in \H$ and $a,b,c,d\in \Z$ with $ad-bc=1$. Then we have 
 \begin{align} E_2\bigg(\frac{a\tau+b}{c\tau+d}\biggl)&=(c\tau+d)^2E_2(\tau)-\tfrac{6 i}{\pi}c (c\tau+d) ,\label{gtransE2}\\
 E_4\bigg(\frac{a\tau+b}{c\tau+d}\biggl)&=(c\tau+d)^4E_4(\tau), \label{gtransE4}\\
 E_6\bigg(\frac{a\tau+b}{c\tau+d}\biggl)&=(c\tau+d)^6E_6(\tau). \label{gtransE6}
 \end{align}
 \end{prop}
Because of these transformation formulas, $E_4$ and $E_6$ are called {\em modular forms
of weight} $4$ and $6$, respectively. The function $E_2$ is not a modular form, but since
 its transformation behavior is similar to how a modular form of weight $2$ should transform, it is called a {\em quasi-modular form of weight $2$}.    

We need information about the location of the zeros of $E_4$ and $E_6$. Indeed,  for $\tau\in \H$  we have
\begin{align}\label{E4zeros}
E_4(\tau)&=0  \text { if and only if } \tau= \varphi(\rho) \text{ for some } \varphi\in \Ga,\\
\label{E6zeros}
E_6(\tau)&=0 \text { if and only if } \tau= \varphi(i) \text{ for some }  \varphi\in \Ga.
\end{align}
If we denote by $\Ga(\tau)\coloneq \{ \varphi(\tau): \varphi \in \Ga\}$ the orbit of a point $\tau\in \H$ under the modular group  $\Ga$, then \eqref{E4zeros} and \eqref{E6zeros} say that the set of zeros of $E_4$ and $E_6$ are given by $\Ga\rho$ and $\Ga i$, respectively. 

We need two special values of $E_2$: 
\begin{equation}\label{eq:E2spec}
E_2(i)=\frac{3}{\pi}\quad \text{and} \quad E_2(\rho)=\frac{2\sqrt 3}{\pi}.  \end{equation}
The first equation follows from \eqref{transE2'} when applied to $\tau=i$ and the second one from \eqref{transE2'} when applied to $\tau=\rho$ together with   \eqref{eq:perE}.

We will record the well-known expressions for the derivates  $E'_k=dE_k/d\tau$.  To absorb a factor $2\pi i$ in these formulas and keep the notation simple, we introduce the abbreviation
 \begin{equation}\label{eq:defD}
DF\coloneq\frac{1}{2\pi i} \frac{dF(\tau)}{d\tau}
 \end{equation}
 for a meromorphic function $F$ on $\H$. 
 \begin{prop}\label{Ederi}
 The following identities  are valid:
\begin{align}  DE_2&=\tfrac{1}{12}(E_2^2-E_4),\label{E2der}\\
 DE_4&=\tfrac{1}{3}(E_2E_4-E_6),\label{E4der}\\
 DE_6&=\tfrac{1}{2}(E_2E_6-E_4^2)\label{E6der}. 
 \end{align}
 \end{prop}

 We need an auxiliary function on $\H$, the {\em discriminant},
    defined as 
\begin{equation}\label{eq:discrdef}
\Delta\coloneq\frac{1}{1728}(E_4^3-E_6^2). 
\end{equation}
This is a modular form of weight $12$  and it is a standard fact that $\Delta$, as a function of $\tau\in \H$, allows  the product expansion 
\begin{equation}\label{Dprod}
\Delta(\tau) =q\prod_{n=1}^\infty (1-q^n)^{24}, 
\end{equation} 
where  $q=\exp(2\pi i \tau)$ as before (see \cite{Sch, Cha}, but note that the notation in these references  is different from ours).
This product expansion implies that 
\begin{equation}\label{eq:Delnoz}
\Delta(\tau)\ne 0 \text{ for $\tau\in \H$} 
\end{equation}
and so $\Delta$ has no zeros on $\H$. In particular,  $E_4^3-E_6^2\ne 0$ on $\H$.

 \section{Equivariance properties}\label{s:equi}
 In this section,  we establish the relevant  equivariance properties  of $s_2$, $s_4$, $s_6$. 
We first introduce some auxiliary conformal and anti-conformal maps on $\wC$. We will use the notation established here throughout the paper. 

For $\tau\in \wC$ we define 
\begin{equation}\label{eq:abcdef}
\al(\tau)\coloneq -\overline{\tau},\quad \be(\tau)\coloneq 1/\overline {\tau}, \quad 
\ga(\tau)\coloneq1-\overline {\tau},
\end{equation}
where we set $\overline {\infty}=\infty$ as before. Then $\al$, $\be$, $\ga$ represent the reflections in the imaginary axis
\begin{equation}\label{eq:imax}
L_{0}\coloneq\{\tau\in \C:\re(\tau)=0\}\cup\{\infty\},
\end{equation}
the unit circle
\begin{equation}\label{eq:ucirc}
\partial \D\coloneq \{\tau\in \C:|\tau|=1\},
\end{equation}
 and the line 
\begin{equation}\label{eq:12line}
L_{1/2}\coloneq\{\tau\in \C:\re(\tau)=1/2\}\cup\{\infty\},
\end{equation} respectively.

The maps $\al$, $\be$, $\ga$ generate the extended modular group 
 $\overline {\Ga}$ discussed earlier.  It contains the  modular group
$\Ga= \PSL_2(\Z)$ as a subgroup of index $2$. The elements 
in $\Ga$ are precisely those M\"obius transformations that can be represented  as a composition of an even number of the reflections
$\al$, $\be$, $\ga$. 

The modular group $\Ga$ is generated by the M\"obius transformations 
\begin{equation}
\sigma(\tau) \coloneq \tau+1, \quad \chi(\tau)\coloneq -1/\tau. 
\end{equation}
Note that 
\begin{equation}
\sigma=\ga\circ \al, \quad \chi=\al\circ \be.
\end{equation}
These relations imply that the maps $\al$, $\sigma$, $\chi$ also generate 
$\overline {\Ga}$. 

We denote by 
\begin{equation}\label{eq:delta}
\de(\tau)\coloneq \frac {\overline \tau}{\overline \tau-1}
\end{equation}
the reflection in the circle $\{\tau\in \C: |\tau-1|=1\}$. In particular, 
 the tri\-angles $T_0, \dots, T_5$ in Figure~\ref{fig:T} are given as 
\begin{equation}
T_1=\be(T_0),\  T_2=\de(T_1),\  T_3=\ga(T_2),\  T_4= \be(T_3),\ 
 T_5=\de(T_4). 
\end{equation}
 
We will now establish the crucial equivariance properties of $s_4$ and $s_6$. 

\begin{prop}\label{prop:poly46} Let $k\in \{4,6\}$. Then 
\begin{equation}\label{eq:equi46}
s_k\circ \varphi =\varphi \circ s_k \quad \text{on $\H$}  
\end{equation}
for all $\varphi \in \overline{\Ga}$. \end{prop}

\begin{proof} In order to show \eqref{eq:equi46}, it is enough to verify 
it for generators $\varphi$  of $\overline \Ga$, namely for the maps 
$\al(\tau)=-\overline {\tau}$, $\sig(\tau)=\tau+1$, and  $\chi(\tau)=-1/\tau$.  

Now it follows from \eqref{eq:perE} and \eqref{eq:symE} that for $k=4,6$ we have the identities 
\[
s_k(\tau+1)=s_k(\tau) +1 \text{ and } s_k(-\overline {\tau})=- \overline{s_k(\tau)},  
\]
or equivalently, $s_k\circ\sig =\sig \circ s_k$ and 
$s_k\circ \al=\al\circ s_k$.  

For the transformation behavior under $\tau\mapsto -1/\tau$, we note 
 that $s_4$ and $s_6$ can be written in the form 
 \[
 f(\tau)=\tau-\frac{6i}{\pi(E_2(\tau)-g(\tau))}=
 \frac{\pi \tau(E_2(\tau)-g(\tau))- 6i}{\pi(E_2(\tau)-g(\tau))} 
 \]
 where $g=E_6/E_4$ for $s_4$ and $g=E_4^2/E_6$ for $s_6$.
 In both cases, $g$ is a meromorphic function on $\H$ satisfying the identity
 \[ 
 g(-1/\tau)=\tau^2 g(\tau).
 \] 
Based on this and the transformation behavior \eqref{transE2'} of $E_2$, for $\tau \in \H$  one can now compute
\begin{align*}
f(-1/\tau)&=-1/\tau-\frac{6i}{\pi(E_2(-1/\tau)-g(-1/\tau))}\\
&= -1/\tau-\frac{6i}{\pi(\tau^2 E_2(\tau)-(6i/\pi)\tau-\tau^2 g(\tau))}\\
&=-\frac{\pi   (E_2(\tau)-g(\tau))}{\pi \tau (E_2(\tau)-g(\tau))-6i}=-1/f(\tau).
\end{align*}
It follows that for $k=4,6$ we have the identity
\begin{equation}
s_k(-1/\tau )=-1/s_k(\tau),
\end{equation}
or equivalently, $s_k\circ \chi=\chi\circ s_k$.  
The  statement follows. 
\end{proof}

\begin{prop}\label{prop:poly2} We have 
\begin{equation}\label{eq:equi2}
s_2(\varphi(\tau)) =\varphi (s_2(\tau))
\end{equation}
for all $\tau \in \H$ and $\varphi \in \overline{\Ga}$.
 \end{prop}
Note that this is an equality of two sets, namely 
the set $s_2(\varphi(\tau))$ obtained from $s_2$ at the 
point $\varphi(\tau)$ and the set obtained as the image set
of $s_2(\tau)$ under $\varphi$. If we denote the map
$M\sub \wC\mapsto \varphi (M)$ also by $\varphi$, 
abusing notation, then we can simply write 
\eqref{eq:equi2} in the form $s_2\circ \varphi= \varphi\circ s_2$
as in   \eqref{eq:poly46} for $k=2$. 

\begin{proof} It is again enough  to show that \eqref{eq:equi2} is true for generators of 
$\overline{\Ga}$, namely  for the maps 
$\sig(\tau)=\tau+1$,  $\al(\tau)=-\overline {\tau}$, and  $\chi(\tau)=-1/\tau$.

In the following, if $M\sub \wC$ is a set and $a,b\in \C$, $a\ne 0$, we define 
\[
aM+b\coloneq \{am+b: m\in M\}
\] with the convention that $a\cdot \infty+b=\infty$. In other words, $aM+b$ is the image set of $M$ under the M\"obius transformation 
$z\mapsto az+b$ on $\wC$. Recall also that 
$\overline {M}=\{\overline{m}:m\in M\}$. In some of the following computations, the denominator of a fraction could be equal to $0$. In all these cases, the numerator 
of the fraction is non-zero and we interpret the fraction as equal to $\infty\in \wC$.

We now  consider the multi-valued function $r$ that assigns 
to $\tau\in\H$ the set $r(\tau)\coloneq \{E_4(\tau)^{1/2},-E_4(\tau)^{1/2}\}$ 
consisting of the (at most) two roots of $E_4(\tau)$. Note that if $z_0\in r(\tau)$, then $-z_0$ is the other element of $r(\tau)$ (of course, $z_0=-z_0=0$ is possible). 

By \eqref{eq:perE} for $k=4$ we have 
 $r(\tau+1)=r(\tau)$, and so the two elements $\pm z_1$  of $r(\tau+1)$
  are also the elements of  $r(\tau)$.
  Then the (at most) two elements $w^+_1$ and $w^-_1$  of $s_2(\tau+1)$ can be written in the form 
\begin{align*}
w^\pm _1&=1+\tau -\frac{6i}{\pi(E_2(\tau+1)\pm z_1)}\\
&= 1+ \tau -\frac{6i}{\pi(E_2(\tau)\pm  {z_1})}. 
\end{align*}
Since $z_1\in r(\tau)$, the expression 
\[ 
\tau -\frac{6i}{\pi(E_2(\tau)\pm  z_1)}
\] 
represents the two elements of $s_2(\tau)$ and so we see that 
 $s_2(\tau+1)=s_2(\tau)+1$.  It follows that 
$s_2(\sig(\tau))=\sig(s_2(\tau))$, as desired

The transformations under  $\al$ and $\chi$ follow from a similar reasoning.
 For $\tau\in \H$ 
we have 
 \[
 r(-\overline \tau)=\overline {r(\tau)}
 \]
 by \eqref{eq:symE}.
  So the two elements $\pm z_1$ of $r(-\overline \tau)$ can be represented in the form $\pm z_1=\pm \overline {z_0}$ with proper choice of $z_0\in r(\tau)$. 
 Therefore, the   two elements $w^+_1$ and $w^-_1$  of $s_2(-\overline \tau)$ can be written as
\begin{align*}
w^\pm _1&=-\overline \tau-\frac{6i}{\pi(E_2(-\overline \tau)\pm z_1)}\\
&= -\overline \tau -\frac{6i}{\pi( \overline {E_2(\tau)}\pm \overline {z_0})}\\
&=-\biggl(\overline {\tau-\frac{6i}{\pi(E_2(\tau)\pm z_0)}}\biggr). 
\end{align*}
Since $z_0\in r(\tau)$, this formula shows that $s_2(-\overline \tau)=-\overline{s_2(\tau)}$ and so 
$s_2(\al(\tau))=\al(s_2(\tau))$, as desired.

It follows from  \eqref{transE4} that we  have 
\[ 
r(-1/\tau)=\tau^2 r(\tau).
\]
 
So the two elements $\pm z_1$ of $r(-1/\tau)$ can be represented in the 
form $\pm z_1=\pm \tau^2 z_0$ with proper choice of $z_0\in r(\tau)$. 
 Then  the  two elements $w^+_1$ and $w^-_1$  of $s_2(-1/\tau)$ can be written as 
\begin{align*}
w^\pm _1&=-1/ {\tau}-\frac{6i}{\pi(E_2(-1/\tau)\pm z_1)}\\
&= -1/\tau-\frac{6i}{\pi(\tau^2 E_2(\tau)-(6i/\pi)\tau\pm \tau^2 z_0)}\\
&=-\frac{\pi   (E_2(\tau)\pm z_0)}{\pi \tau (E_2(\tau)\pm z_0)-6i}\\
& - \frac{1}{\biggl(\tau-\displaystyle\frac{6i}{\pi(E_2(\tau)\pm z_0)}\biggr)}.
\end{align*}
Now the denominator of the outer fraction of the last expression represents the two elements of $s_2(\tau)$. It follows that $s_2(\chi(\tau))=\chi(s_2(\tau))$, as desired. 
 \end{proof}

We will  need more  information about how the branches $s^+_2$ and $s^-_2$ 
of $s_2$ transform under some specific elements in  $\overline {\Ga}$.
For this  we introduce the auxiliary set 
\begin{equation}\label{eq:Omdef}
W\coloneq\{ \tau\in \H: -1/2\le \re(\tau)\le 1/2,\ |\tau-1|\ge 1,\ 
|\tau+1|\ge1\}.
\end{equation}
Note that $W$ is the relative closure in $\H$ of the simply connected region $\inte(W)\sub \H$ 
containing  the positive imaginary axis and no root 
of $E_4$. 

If, as before, the root  $ E_4^{1/2}$ is chosen so that it takes positive values on the positive imaginary axis, then we can  consistently extend $ E_4^{1/2}$ to a holomorphic function throughout $\inte(W)$ with a unique continuous extension 
to $W$. 
With this choice of the branch of $ E_4^{1/2}$,   the functions $s^+_2$ and $s^-_2$ given by \eqref{eq:s2def} are well-defined   $\wC$-valued continuous functions on $W$ that are meromorphic in $\inte(W)$.  

 Note that $\al(W)=W$ and $\be(W)=W$ for the maps in \eqref{eq:abcdef}.
 Since $\chi=\al\circ \beta$, the set $W$ is also invariant under the transformation $\tau\mapsto \chi(\tau)=-1/\tau$.

\begin{lem}\label{lem:transs2} On  $W$ the following identities are valid:
\begin{align}
s^+_2\circ \al&=\al\circ s^+_2,\quad{s}^-_2\circ \al=\al\circ s^-_2,\label{eq:sal}\\
s^+_2\circ \beta&= \beta \circ s^-_2, \quad s^-_2\circ \beta=\beta \circ s^+_2. 
\label{eq:sbet}
\end{align}
\end{lem}

\begin{proof}  
We first consider the transformation behavior of 
$E_4^{1/2}$ on $W$ normalized so that it takes positive real values on the 
positive imaginary axis.

 By \eqref{transE4} we have  the identity
$E_4^{1/2}(-1/\tau)=\pm \tau^2  E_4^{1/2}(\tau)$ for $\tau \in W$ with one of the signs. Since $\tau\mapsto -1/\tau$ maps the positive imaginary axis onto itself and $E_4^{1/2}$ takes positive values there,   
we must have 
 \begin{equation} \label{eq:rootE4trans}
E_4^{1/2}(-1/\tau)=-\tau^2  E_4^{1/2}(\tau)
\end{equation}
for $\tau \in W$. 

It now follows that for $\tau\in W$ we have
\begin{align*}
s^+_2(-1/\tau)&=-1/\tau-\frac{6i}{\pi(E_2(-1/\tau)+E_4^{1/2}(-1/\tau))}\\
&= -1/\tau-\frac{6i}{\pi(\tau^2 E_2(\tau)-(6i/\pi)\tau-\tau^2 E_4^{1/2}(\tau))}\\
&=-\frac{\pi   (E_2(\tau)-E^{1/2}(\tau))}{\pi \tau (E_2(\tau)-E_4^{1/2}(\tau))-6i}=-1/ s^-_2(\tau).
\end{align*}
 In other words, 
 \begin{equation}\label{eq:s2chitrans}
s^+_2\circ \chi=\chi\circ  s^-_2 \text{ on $W$}.
\end{equation}

Similarly, from \eqref{eq:symE} for $k=4$  it follows  that
 \begin{equation} \label{eq:rootE4refl}
E_4^{1/2}(-\overline{\tau})=\overline{E_4^{1/2}(\tau)}
\end{equation}
for $\tau \in W$, and so 
\begin{align*}
s^+_2(-\overline{\tau})&=-\overline{\tau}-\frac{6i}{\pi(E_2(-\overline{\tau})+E_4^{1/2}(-\overline{\tau}))}\\
&=  -\overline{\tau}-\frac{6i}{\pi\biggl(\overline{E_2(\tau)}+\overline{E_4^{1/2}(\tau)}\biggr)}\\
&=-\overline{s_2^+(\tau)}, 
\end{align*}
or equivalently, 
\[ 
s^+_2\circ \al=\al\circ  s^+_2 \text{ on $W$}.
\]
In the same way, one verifies that 
\[s^-_2\circ \al=\al\circ  s^-_2 \text{ on $W$}.
\]

This shows that the identities \eqref{eq:sal} are true. The identities \eqref{eq:sbet}
easily follows from \eqref{eq:sal}, \eqref{eq:s2chitrans}, and the fact that $\beta=
\chi\circ \al$. 
\end{proof}

Note that the relation \eqref{eq:s2chitrans} established in the previous proof can be  written more explicitly as
\begin{equation}\label{eq:s2trans}
s_2^+(-1/\tau)=-1/s_2^-(\tau)
\end{equation}
for $\tau\in W$.

\section{Critical points and  poles}\label{sec:critpoles}
In this section,  we study the relation of the critical points of our Eisenstein series  
and the poles of the associated polymorphic function and establish  some related statements that will be useful later. Throughout we use the differential operator \eqref{eq:defD}.

 We start by proving the following auxiliary facts.

 \begin{lem} \label{lem:czeros}The following pairs of functions have no common zero in $\H$:
 
  \begin{enumerate}[label=\normalfont{(\roman*)}]

 \item\label{i:E46}   $E_4$ and $E_6$. 
 
 \smallskip
 \item\label{i:E2}   $E_2$ and $E_4$.

     \smallskip 
    \item\label{i:DE2'}
     $E_2\pm E^{1/2}_4$ and $D(E_2\pm E^{1/2}_4)$.

     \smallskip 
    \item\label{i:DE2}  $E_2^2-E_4$ and 
    $D(E_2^2-E_4)$.  
  
   \smallskip 
    \item\label{i:DE4} $E_2E_4-E_6$ and 
    $D(E_2E_4-E_6)$.  
  
   \smallskip 
    \item\label{i:DE6} $E_2E_6-E^2_4$ and 
    $D(E_2E_6-E^2_4)$.  
   \end{enumerate}
    \end{lem}

   The statement  \ref{i:DE2'} requires proper interpretation that we will discuss 
   in the proof.  
    
  \begin{proof}  \ref{i:E46} This immediately follows from  
 \eqref {E4zeros} and \eqref{E6zeros}.
  
  \smallskip
   \ref{i:E2} If $\tau_0\in \H$ is a zero of $E_4$,
  then  $\tau_0$ is equivalent to $\rho$ under the modular group $\Ga$. It follows that there exist 
$a,b,c,d\in \Z$ with $ad-bc=1$ such that
\[
\tau_0=\frac{a\rho+b}{c\rho+d}.
\]
By the transformation formula \eqref{gtransE2}  and by  \eqref{eq:E2spec},  we then have 
\begin{align*}
E_2( \tau_0)&=(c\rho+d)^2 E_2(\rho)-\tfrac{6i}{\pi}c(c\rho+d)=(c\rho+d) (\tfrac{2\sqrt 3}{ \pi }(c\rho+d) -\tfrac{6i}{\pi}c)\\
&= \tfrac{2\sqrt 3}{\pi} (c\rho+d)(c\rho+d -ic\sqrt{3})\\&=
 \tfrac{\sqrt 3}{2\pi}(ic\sqrt 3+ (c+2d))(-ic\sqrt 3+ (c+2d))\\
 &= \tfrac{\sqrt 3}{2\pi}(3c^2+(c+2d)^2).
\end{align*}
Since the integers  $c$ and $d$ cannot both be $0$, we see that $E_2(\tau_0)\ne 0$. 
 The statement follows. 

\smallskip
 \ref{i:DE2'}  Let us assume that  $E_2+ E_4^{1/2}=0$ at a point $\tau_0\in\H$ with some choice of the root $ E_4(\tau_0)^{1/2}$. Then $E_4(\tau_0)\ne 0$
 as follows from \ref{i:E2}; so  in a neighborhood of $\tau_0\in\H$ (say, a small disk centered at $\tau_0$) we can make a consistent holomorphic 
choice of $E_4^{1/2}$ which also leads to consistent definitions of all powers 
$E_4^{k/2}\coloneq (E_4^{1/2})^k$, $k\in \Z$.   Then $D(E_4^{1/2})$ is well-defined locally near $\tau_0$. The claim is then that  
      $D(E_2+ E^{1/2}_4)\ne 0$ at $\tau_0$.  

Now with these choices understood, we have  
\begin{align}\label{eq:Ds2dem}
D(E_2+E_4^{1/2})&=D(E_2)+\tfrac12 E_4^{-1/2}D(E_4)\\ &=
\tfrac1{12}(E_2^2- E_4+2E_2E_4^{1/2}- 2E_6E_4^{-1/2}),\notag
\end{align}
locally near $\tau_0$. At $\tau_0$ we have 
\[
E_2^2- E_4=(E_2- E_4^{1/2})(E_2+ E_4^{1/2})=0,
\]
and so indeed
  \[
 D(E_2+E_4^{1/2})=\tfrac16 E_4^{-1/2}(E_4^{3/2}-E_6)\ne 0.
 \]
Here we used that $(E_4^{3/2}-E_6)(E_4^{3/2}+E_6)=E_4^3-E_6^2=1728 \Delta$ does nor vanish anywhere on $\H$. 

Since the root of $E_4^{1/2}$ in unique up to sign, this  argument also  implies   that $E_2- E^{1/2}_4$ and $D(E_2-E^{1/2}_4)$ cannot both be equal to $0$ on $\H$.

\smallskip
 \ref{i:DE2}  One can deduce this from \ref{i:DE2'}, but we can also argue directly as follows. To reach a contradiction,  suppose there is a point 
 $\tau_0\in \H$  where $E_2^2-E_4$ and $D(E_2^2-E_4)$ both vanish.  Note that 
 \begin{equation}\label{eq:DDE2}
D(E_2^2-E_4)
=\tfrac{1}{6}(E_2^3-3E_2E_4+2E_6).
\end{equation}
It follows that then we 
 have (with the argument $\tau_0$ understood)
 \[ E_2^2=E_4 \text{ and }  E_2^3-3E_2E_4+2E_6=0.\]
 Using the first equation in the second, we are   led to 
 \[ 
 E_2E_4=E_6.
 \]
 Squaring and using  $E_2^2=E_4$ again, we see that 
 $E_4^3=E_6^2$ at $\tau_0$. But we know that $E_4^3-E_6^2=1728 \Delta\ne 0$ on $\H$, and obtain a contradiction.

 \smallskip
 \ref{i:DE4} To reach a contradiction,  suppose there is a point 
 $\tau_0\in \H$  where $E_2E_4-E_6$ and $D(E_2E_4-E_6)$ both vanish. Now 
\begin{align} \label{eq:Ds4dem}
D(E_2E_4-E_6)&=D(E_2)E_4+E_2D(E_4)-D(E_6)\\
=\tfrac1{12}(E_2^2E_4-E_4^2)&+\tfrac1{3}(E_2^2E_4-E_2E_6)
-\tfrac12(E_2E_6-E_4^2) \notag\\
&= \tfrac{5}{12} (E_2^2E_4+E_4^2-2E_2E_6),\notag
\end{align}  
and so  at the point  $\tau_0$ we have  
 \[ E_2E_4=E_6\text{ and }  E_2^2E_4+E_4^2-2E_2E_6=0.\]
Using the first equation in the second, we arrive at 
 \[ 
 E_4^2=E_2E_6.
 \]
If we multiply this by $E_4$ and use $E_2E_4=E_6$ again, we deduce $E_4^3=E_6^2$ at $\tau_0$, which again gives a contradiction.

 \smallskip
 \ref{i:DE6} To reach a contradiction,  suppose there is a point 
 $\tau_0\in \H$  where $E_2E_6-E_4^2$ and $D(E_2E_6-E_4^2)$ both vanish. We have 
\begin{align} \label{eq:Ds6dem} D(E_2E_6-E_4^2)&=D(E_2)E_6+E_2D(E_6)-2D(E_4)E_4\\
=\tfrac1{12}(E_2^2-E_4)E_6
&+\tfrac12E_2(E_2E_6-E_4^2)-\tfrac2{3}(E_2E_4-E_6)E_4\notag \\
&= \tfrac{7}{12} (E_2^2E_6-2E_2E_4^2+E_4E_6).\notag 
\end{align}  
 Then we 
 have (with the argument $\tau_0$ understood)
 \begin{equation}\label{eq:dubrootE6}
 E_2E_6=E_4^2\text{ and }  E_2^2E_6-2E_2E_4^2+E_4E_6=0.
 \end{equation}
Using the first equation in the second, we are   led to 
 \[ 
 E_4E_6=E_2E_4^2.
 \]
If we multiply this by $E_6$ and use $E_2E_6=E_4^2$, we arrive at 
$E_4(E_4^3-E_6^2)=0$. Since $E_4^3-E_6^2\ne 0$ on $\H$, we then must have $E_4=0$. At such a point we have  $E_6\ne 0$ and so $E_2=0$ by \eqref{eq:dubrootE6}.  So $E_2$ and $E_4$ would have a common zero, but we know by \ref{i:E2}  that such a common zero does not exists. Again we reach a contradiction. 
  \end{proof}

 \begin{prop} \label{prop:Esimcrit}
 For $k=2,4,6$ each  critical point of $E_k$ in $\H$  is simple. 
 \end{prop}

 \begin{proof} We have to show that  $DE_k$ and $D(DE_k)$  have no common zero in $\H$. This immediately follows from the formulas in Proposition~\ref{Ederi} in combination with 
 Lemma~\ref{lem:czeros}~\ref{i:DE2}--\ref{i:DE6}. 
 \end{proof}

\begin{lem} \label{lem:critpol46} For $k=4,6$ and  $\tau\in \H$ we have the following equivalence
\[ E'_k(\tau)=0 \text{ if and only if } s_k(\tau)=\infty.
\]
\end{lem}
 In other words, the location of the critical points of $E_4$ and $E_6$ is the same as the location of the poles of $s_4$ and $s_6$, respectively.   

\begin{proof}  From the definition of $s_4$ it is clear that  that if $s_4(\tau)=\infty$ for some $\tau\in \H$, then $(E_2E_4-E_6)(\tau)=0$ and so $E_4'(\tau)=0$ by \eqref{E4der}. 

Conversely, if $E_4'(\tau)=0$, then $(E_2E_4-E_6)(\tau)=0$. So the only way that $s_4$ could not have a pole at $\tau$ is if $E_4(\tau)=0$.
Then also $E_6(\tau)=0$, and so $E_4$ and $E_6$ would have a common zero; but such a common zero of $E_4$ and $E_6$ does not exist and we must have $s_4(\tau)=\infty$.  

The argument for $E_6$ and $s_6$ is completely analogous based on 
\eqref{E6der}. 
\end{proof}

\begin{lem} \label{lem:critpol2} For  $\tau\in \H$ we have 
\[ E'_2(\tau)=0 \text{ if and only if } \infty\in s_2(\tau).
\]
\end{lem}
\begin{proof}  If $\infty\in s_2(\tau)$, then one of the expressions $E_2+E_4^{1/2}$ or 
$E_2-E_4^{1/2}$ must vanish at $\tau$. But then $E_2^2-E_4=(E_2+E_4^{1/2})(E_2-E_4^{1/2})=0$ at $\tau$, and so $E'_2(\tau)=0$ by \eqref{Ederi}. 

Conversely, if $E_2'(\tau)=0$, then 
$E_2^2-E_4=(E_2+E_4^{1/2})(E_2-E_4^{1/2})=0$ at $\tau$ and so $E_2+E_4^{1/2}$ or
$E_2-E_4^{1/2}$ must vanish. Then $\infty\in s_2(\tau)$. 
\end{proof} 
 
For $\tau\in \H$ we have   $\infty \in s_2(\tau)$ if and only if 
$E_2+E_4^{1/2}=0$ or
$E_2-E_4^{1/2}=0$ at $\tau$. In this case, 
$E_4(\tau)\ne 0$ and so $\tau\in \H\setminus \Ga(\rho)$ by \eqref{E4zeros}, because otherwise $E_4(\tau)=0$ and  $E_2(\tau)=0$. But this contradicts  Lemma~\ref{lem:czeros}~\ref{i:E2}. 

If  $E_4(\tau)\ne 0$, both branches $s_2^+$ and $s_2^-$ are well-defined meromorphic function near $\tau$. We conclude that for 
$\tau\in \H$ we have  $\infty\in s_2(\tau)$ if and only if $\tau\in \H\setminus \Ga(\rho)$ and 
$\tau$ is a pole of one of the branches $s_2^+$ or  $s_2^-$ which are meromorphic near $\tau$.

  \begin{prop} \label{prop:simpol} The poles of the functions $s^\pm_2$, $s_4$, $s_6$ lie in the set $\H\setminus \Ga(\rho)$. Moreover, each pole of any  of these functions is simple.   
 \end{prop}
 \begin{proof} If  $\tau_0\in \H$ is a pole of $s^+_2$ or $s^-_2$, then at this point 
 we have $E_2+E_4^{1/2}=0$ or $E_2-E_4^{1/2}=0$. 
We have just seen that this forces  $\tau_0\in \H\setminus \Ga(\rho)$ and that then the branches $s^+_2$ and $s^-_2$ are meromorphic functions near $\tau_0$. 
By definition of $s^\pm_2$ the point $\tau_0$ can only  be  a double pole of any of these branches, if $E_2\pm E_4^{1/2}$ and $D(E_2\pm E_4^{1/2})$ had the  common  zero $\tau_0$. But by Lemma~\ref{lem:czeros} such a common zero does not exist. 

If $\tau_0\in \H$ is a pole of $s_4$, then the definition of $s_4$ shows that  at this point 
 we have $E_2E_4-E_6=0$. Here $\tau_0\not\in \Ga(\rho)$, because otherwise $E_4(\tau_0)=0$ and $E_6(\tau_0)=0$ which is impossible by 
 Lemma~\ref{lem:czeros}~\ref{i:E46}. The pole cannot be   a double pole, because otherwise  
$E_2E_4-E_6$ and $D(E_2E_4-E_6)$ had the  common  zero $\tau_0$. But by Lemma~\ref{lem:czeros}~\ref{i:DE4} such a common zero does not exist. 

If $\tau_0\in \H$ is a pole of $s_6$, then at this point 
 we have $E_2E_6-E_4^2=0$. Then  $\tau_0\not\in \Ga(\rho)$, because otherwise $E_4(\tau_0)=0$,  and $E_6(\tau_0)=0$ or $E_2(\tau_0)=0$, which are both impossible by Lemma~\ref{lem:czeros}~\ref{i:E2} and \ref{i:DE6}. Again   $\tau_0$ could only be a double pole of $s_6$  
if  
$E_2E_6-E_4^2$ and $D(E_2E_6-E_4^2)$ had the  common  zero $\tau_0$. But by Lemma~\ref{lem:czeros}~\ref{i:DE6} such a common zero does not exist. 
\end{proof}

%
%
%

 
 \begin{lem} \label{lem:derivsk} On $ \H$ we have the following   identities:
 \begin{align}
 (s^\pm_2)'&= \frac{2E_4^{-1/2}(E_4^{3/2}\pm E_6)}
   {\big(E_2\pm E_4^{1/2}\big)^2}, \label{eq:ders2}\\
 s_4'&=-\frac{5(E_4^3-E_6^2)}{(E_2E_4-E_6)^2}, 
 \label{eq:ders4}\\
 s_6'&=\frac{7 E_4(E_4^3-E_6^2)}{(E_2E_6-E_4^2)^2}.
 \label{eq:ders6}
\end{align}
  \end{lem}
The derivatives here are with respect to the variable $\tau\in \H$.   
  \begin{proof}
These formulas can be obtained  by straightforward computations based on  Pro\-po\-sition~\ref{Ederi}.  
Indeed, by \eqref{eq:Ds2dem} we have
\begin{align*}
  Ds^+_2&=\frac{i}{2\pi}\cdot \frac{-(E_2+E_4^{1/2})^2+E_2^2- E_4+2 E_2E_4^{1/2}- 2E_6E_4^{-1/2}}{\big(E_2+E_4^{1/2}\big)^2}\\
  &=-\frac{i}{\pi} \cdot \frac{E_4^{-1/2}(E_6+E_4^{3/2})}
   {\big(E_2+E_4^{1/2}\big)^2}.
  \end{align*}
This gives the identity for $(s_2^+)'$. We obtain the one  for $(s_2^-)'$ by changing the sign of the root $E_4^{1/2}$. 

By  \eqref{E4der} and  \eqref{eq:Ds4dem} we have
\begin{align*}
Ds_4&=\frac{i}{2\pi} \cdot \frac{-5(E_2E_4-E_6)^2+5 E_4
(E_2^2E_4+E_4^2-2E_2E_6)} {(E_2E_4-E_6)^2}\\
&=\frac{5i}{2\pi} \cdot \frac{E_4^3-E_6^2}{(E_2E_4-E_6)^2}. 
  \end{align*}
The identity for $s_4'$ follows.    
  
By \eqref{E6der} and    \eqref{eq:Ds6dem} we have
  \begin{align*}
  Ds_6
  &=\frac{i}{2\pi}\cdot \frac{-7(E_2E_6-E_4^2)^2+
   7(E_2^2E^2_6-2E_2E_4^2E_6+E_4E^2_6)}{(E_4^2-E_2E_6)^2}\\
  &=-\frac{7i}{2\pi}  \cdot \frac{E_4(E_4^3-E_6^2)}
   {(E_4^2-E_2E_6)^2}.  \end{align*}
 The identity for $s_6'$ follows.  
  \end{proof}
  
  If $f$ is a function that is meromorphic near a point $z_0\in \C$, then we call
  $z_0$ a {\em critical point of $f$} if $f(z_0)\in \C$ (and so $f$ is actually a holomorphic function near  $z_0$) and $f'(z_0)=0$. 
  
  \begin{prop} \label{prop:critpts}  \begin{enumerate}[label=\normalfont{(\roman*)}]

 \item\label{i:pols2}   The functions $s_2^\pm$  have  
   no critical points in $\H\setminus \Ga(\rho)$.
   
   \smallskip 
    \item\label{i:pols4} The function $s_4$ has no critical points in $\H$.     
     \smallskip 
    \item\label{i:pols6}
    The set of  critical  points of  $s_6$ in $\H$  is equal to $\Ga(\rho)$. Each of these critical points is simple.  
   \end{enumerate}
   \end{prop}
 
 \begin{proof} \ref{i:pols2}  Since 
 \[(E_4^{3/2}+E_6)(E_4^{3/2}-E_6)=E_4^3-E_6^2=1728\Delta,
 \] the numerator in the expression 
 \eqref{eq:ders2} never vanishes at any point $\tau\in \H\setminus\Ga(\rho)$. This implies that neither $s_2^+$ and $s_2^-$ can have a critical point in 
$ \H\setminus\Ga(\rho)$.
 
 \smallskip
  \ref{i:pols4} In expression \eqref{eq:ders4} the numerator never vanishes and so $s_4$ has no critical points. 
 
\smallskip
  \ref{i:pols6} In  \eqref{eq:ders6} the numerator vanishes precisely at the points in  the set
 $\Gamma(\rho)$. Note that at these points
$E_2E_6-E_4^2=E_2E_6\ne 0$ by Lemma~\ref{lem:czeros}~\ref{i:E46} and
 \ref{i:E2}, and so
the denominator  in  \eqref{eq:ders6} is non-zero. 
We also have $DE_4=\tfrac13 (E_2E_4-E_6)=-\tfrac13 E_6\ne 0$ at these points. This shows that  the set of critical points of  $s_6$ is given by 
$\Ga(\rho)$ and that each of these critical points is simple. 
 \end{proof}
 
 \begin{cor}\label{cor:locinj}
 \begin{enumerate}[label=\normalfont{(\roman*)}]

 \item\label{i:loc2}   The functions  $s_2^\pm$  are locally 
   injective near every point in  $\H\setminus \Ga(\rho)$.
   
   \smallskip 
    \item\label{i:loc4} The function $s_4$ is locally injective near every point in $\H$.     
     \smallskip 
    \item\label{i:loc6} The function $s_6$ is locally injective near every point in   $\H\setminus \Ga(\rho)$   
 \end{enumerate}
 \end{cor}
 
 \begin{proof} These statements immediately follow from Propositions~\ref{prop:simpol} and~\ref{prop:critpts}.
 \end{proof}
   
 \section{Auxiliary results about conformal maps}\label{sec:auxconf}
 In this section, we collect various results about conformal maps.  
 We will need the following version of the Argument Principle.

\begin{prop}\label{prop:arg}
Suppose $U\sub \bC$ is a closed Jordan region and $f\:\inte(U)\ra \C$ is a holomorphic function  that has a continuous extension (as a map into $\bC$) to the boundary $\partial U\sub \bC$. Suppose this extension maps $\partial U$ homeomorphically to the boundary $\partial V\sub \bC$ of a closed Jordan region $V\sub \bC$ with $ \inte(V)\sub \C$. 

If $\infty\in \partial V$ we make the following additional assumptions:
\begin{itemize} 

 \item[\textnormal{(i)}] if $z_0\in \partial U$ is the unique point with $f(z_0)=\infty$, then $f$ is locally injective on $U$ near $z_0$.
 
 \smallskip 
\item[\textnormal{(ii)}]  if the Jordan curve $\partial U$ is oriented so that $U$ lies on the left of $\partial U$ and if $\partial V$ carries the orientation induced by $f$, then $V$ lies on the left of $\partial V$.

\end{itemize}
Under these hypotheses, $f$ is a homeomorphism of   $U$ onto $V$ that is a biholomorphism  between $\inte(U)$ and $\inte(V)$.
\end{prop}
In our application of this statement, $\infty$ will be on the boundary of $U$ and $V$. This proposition goes back to  \cite[Prop.\ 4.1]{BZeta}, but the statement there was  incorrect as condition (i) was missing in the formulation.

\begin{proof}  Suppose first that $\infty\not \in \partial V$ (and hence $\infty\not \in V$). Then  the Argument Principle implies that on $\inte(U)$ the function $f$ attains each value in $\inte(V)$ once and no other values (see \cite[pp.~310--311]{Bur}  and 
 \cite[Exercise 9.17 (i)] {Bur} for this type of argument). The statement easily follows in this  case.   
 
 If $\infty \in  \partial V$, then the additional assumptions allow us to reduce to the previous case. For this we consider  auxiliary 
 Jordan regions $U'\sub U$ whose boundaries $\partial U'$ are obtained by replacing a small arc $\alpha\sub \partial U$ that contains $z_0$ in its interior with a small arc $\alpha'$ that has  the same endpoints, but whose interior is contained in $\inte(U)$ and so avoids $z_0$.  With suitable choices, condition (i) implies that $f$ is a homeomorphism of $\partial U'$ onto its image and $\infty\not \in f(\partial U')$. By the first part of the proof, there exists a Jordan region $V'\sub \C$ with $\partial V'=f(\partial U')$ such that $f$ is a ho\-meomorphism from $U'$ onto $V'$.  By choosing the arcs $\alpha$ and $\alpha'$ smaller and smaller, we  conclude that $f$ is a homeomorphism of $U$ onto its image. Condition (ii) then implies 
 $f(U)=V$ and we see that the statement is also true when $\infty\in \partial V$.
 \end{proof}

 \begin{lem} \label{lem:injinfty} Suppose a holomorphic function $f$ has a  representation as convergent series 
  \begin{equation}\label{eq:finj}
 f(\tau)=\tau+ \sum_{n=-1}^\infty a_n q^n
 \end{equation}
 for $\tau \in \H$ with  $\im(\tau)>c_0\ge 0$, where $a_n\in \C$ for $n\in \{-1\}\cup\N_0$. 
  
Then the following statements are true:
 
 \begin{enumerate}[label=\normalfont{(\roman*)}]

\smallskip
 \item\label{item:inj1} If  $a_{-1}=0$, then there exists 
 $c_1>c_0$ such that $f$ injective on 
 the half-plane  $H_1\coloneq \{ \tau\in \H: \im(\tau)\ge c_1\}$.

\smallskip 
    \item\label{item:inj2} Suppose   $a_{-1}\ne 0$. Let $\alpha, \beta\in \R$ with $\alpha<\beta$ and $\beta-\alpha<1$ be arbitrary.  Then there exists $c_2>c_0$ such that $f$ is injective on the closed 
   half-strip
  \[ H_2\coloneq \{ \tau\in \C: \im(\tau)\ge c_2, \ \alpha\le \re(\tau)\le \beta\}.
  \]
\end{enumerate}
 \end{lem}
 
 \begin{proof} \ref{item:inj1}  We have 
 \[
 f'(\tau)=1+O(q)
 \]
 for $\tau\in \H$ as $\im(\tau)\to+ \infty$. In particular, there exists
 $c_1>c_0$ such that $\re(f'(\tau))>0$ for $\im(\tau)\ge c_1$. Then if $\tau_1, \tau_2\in H_1$, where $H_1$ is  defined as in the statement, and $\tau_1\ne \tau_2$, we have
 \[
 \re \biggl(\frac{ f(\tau_2)-f( \tau_1)}{\tau_2-\tau_1}\biggr)
 =\int_0^1 \re(f'((1-t)\tau_1+t\tau_2))\, dt>0
 \]
 and so $f(\tau_2)\ne f( \tau_1)$.

 \smallskip
\ref{item:inj2} The function 
\[
g(\tau)\coloneq \tau+ \sum_{n=0}^\infty a_n q^n
\]
defined for  $\tau \in \H$ with $\im(\tau)>c_0$  is holomorphic and 
\[ 
g'(\tau)=1+O(q)
\] 
as $\im(\tau)\to + \infty$. This implies that there exists $c_2'>c_0$ such that 
\[
|g(\tau_2)-g(\tau_1)|\le 2|\tau_2-\tau_1| 
\]
whenever $\tau_1, \tau_2\in \H$ with $\im(\tau_k)\ge c_2'$ for $k=1,2$. 

On the other hand, let $\ga\coloneq \beta-\alpha\in (0,1)$. Then the holomorphic function $z\mapsto k(z)\coloneq (1-\exp(2\pi i z))/z$  has a removable singularity at $z=0$ and has its only zeros in the set $\Z\setminus\{0\}$.
In particular, $k$ has no zeros in the set 
\[ U\coloneq  \{z\in \C: \im(z)\ge 0 \text{ and } -\ga\le \re(z)\le \ga\}.
\]
It easily follows that there exists a constant $\delta>0$ such that 
\begin{equation}\label{eq:expaux}
\frac {|1-\exp(2\pi i z)|}{|z|}\ge \frac {\delta}{1+\im(z)}
\end{equation}
for $z\in U$. 

Now let $\tau_1, \tau_2\in \H$ with $\alpha \le \re(\tau_1), \re(\tau_2)\le \beta$ and $\im(\tau_1),\,  \im(\tau_2)>c_0$ be arbitrary.  Without loss of generality, we may assume that 
$\im(\tau_2)\ge \im(\tau_1)$. Then $z_0\coloneq \tau_2-\tau_1\in U$ and so 
\begin{align*}
|\exp(-2\pi i\tau_2)- \exp(-2\pi i\tau_1)|&= \exp(2\pi \im(\tau_2))
|1- \exp(2\pi iz_0)|\\
\ge  \exp(2\pi \im(\tau_2))\frac {\delta |z_0|}{1+\im(z_0)}\\
\ge 
\delta \frac{\exp(2\pi \im(\tau_2))}{1+\im(\tau_2)}|\tau_2-\tau_1|.
\end{align*}
 This inequality shows that if  in addition $$\im(\tau_1),\,  \im(\tau_2)>c_2$$
 with $c_2\ge c_2'$ sufficiently large (independent of $\tau_1$ and $\tau_2$), then 
 for the function $\tau\mapsto h(\tau)\coloneq a_{-1}\exp(-2\pi i \tau)$ we have 
 \[
 |h(\tau_2)-h(\tau_1)|\ge 3|\tau_2-\tau_1|.
 \]
 Note that $f=g+h$. 
It follows that if  $\tau_1, \tau_2\in H_2$, where $H_2$ is  defined as in the statement, and $\tau_1\ne \tau_2$, then we have
\begin{align*}
|f(\tau_2)-f(\tau_1)|&\ge |h(\tau_2)-h(\tau_1)|-|g(\tau_2)-g(\tau_1)|\\
&\ge 3|\tau_2-\tau_1|-2 |\tau_2-\tau_1|=|\tau_2-\tau_1|>0,
\end{align*}
and so $f(\tau_2)\ne f(\tau_1)$. 
 \end{proof}
 
 \begin{lem} \label{lem:orient} Suppose $f$ is a holomorphic function near $0$
 with $f(0)=0$ and $f'(0)> 0$ and such that  $f(z)\in \R$ for $z\in \R$ near $0$. Then  $\im(f(z))>0$ for all $z\in \C$ near $0$ with $\im(z) >0$. 
 \end{lem}
 
 \begin{proof}  Note that for $n\in \N$ we have $|\sin n t|\le n|\sin t|$ for all 
 $t\in\R$.
 It follows that for  $z=re^{it}$ with $r\ge 0$ and $t\in \R$ we have 
  \[
 |\im (z^n)|=r^n |\sin(nt)|\le nr^n |\sin t| =n|z|^{n-1} |\im(z)|.
 \]

Now our assumptions imply that near $0$ the function $f$ has a power series representation of the form 
\[
f(z)=\sum_{n=1}^\infty a_nz^n
\]
with $a_1=f'(0)>0$ and $a_n\in \R$ for $n\in \N$, $n\ge 2$. 
It follows that for $z\in \C$ sufficiently close to $0$ with $\im(z)>0$ we have 
\begin{align*} \im(f(z))&=\sum_{n=1}^\infty a_n\im(z^n)\\
&\ge \im(z) \biggl( a_1-\sum_{n=2}^\infty n |a_n| |z|^{n-1}\biggr)\\
&\ge  \im(z) (a_1-c|z|), 
\end{align*}
 where $c>0$ is a suitable constant. 
 This shows that if $|z|$ is sufficiently small and $\im(z)>0$, then 
 \[
 \im(f(z))\ge \im(z) (a_1-c|z|)\ge  \im(z)(a_1/2)>0,
 \]
 as desired.
 \end{proof}
Another way to think about the statement in the previous lemma is as follows.  
Near $0$ the map $f$ is a conformal map. If we traverse a small 
 interval $I$ in $\R$ containing $0$ in its interior from left to right, then our assumptions imply that under the map $f$ the image point also traverses a small interval $J=f(I)$ in $\R$
 containing $0$ in its interior from left to  right. Now 
points
$z\in\C$ close to $0$ with $\im(z)>0$ lie on the left of the oriented interval $I$. Hence the image point must also lie on the left of the oriented interval $J$, that is, we must have $\im(f(z))>0$. 

This interpretation immediately leads to some generalizations of the previous
lemma that we will freely use in the following: for example, suppose 
$f$ is holomorphic near $i$ and $f'(i)>0$. Suppose $f$ 
sends the imaginary axis near $i$ into itself.   Then for points $\tau$ near $i$ with $\re(\tau)>0$ we have $\re(f(\tau))>0$.  Analytically, this and similar  
statements  can easily be reduced to Lemma~\ref{lem:orient} by using auxiliary maps in source and target space.

  \begin{lem} \label{lem:minprincip} Suppose $V\sub \C$ is an open  region 
  and  $h\: V \ra \R$ is a non-constant harmonic function 
  such that 
  \begin{equation}
\liminf_{V \ni z \to z_0}h(z)\ge 0
\end{equation}
for all $z_0\in \partial V$. Then $h(z)>0$ for all $z\in V$.  
\end{lem}
Here the boundary  $\partial V$  includes $\infty\in \wC$ if $V$ is  unbounded (as in our applications of the lemma).

\begin{proof} This is a standard fact, but we include a proof for the sake of completeness. 
Let $m_0\coloneq \inf_{z\in V} h(z)\in \{-\infty\}\cup\R$. Then we can 
find a sequence $\{z_n\}$ in $V$ such that $h(z_n)\to m_0$ as $n\to \infty$. Passing to a subsequence if necessary, we may assume that $z_n\to z_0$ as  $n\to \infty$,  where $z_0\in V\cup \partial V \sub \wC$.  

Here $z_0\in V$ is impossible, because then $h(z_0)=m_0$ by continuity of $h$. In particular, $h$ would attain its  minimum at   
$z_0$ which is not possible for a non-constant harmonic function. 
We are left with $z_0\in \partial V$. But then our assumptions imply that 
\[m_0=\lim_{n\to \infty} h(z_n)\ge \liminf_{V \ni z \to z_0}h(z)\ge 0.
\] This shows that $h(z)\ge m_0\ge 0$ for $z\in V$. Here we must actually have $h(z)>0$ for $z\in V$, because otherwise we could find a point $u_0\in V$ with $h(u_0)=0$. Then $h$ attains its minimum $0$ at $u_0$ and we obtain a contradiction as before. 
\end{proof}

Recall that an {\em arc}  $\alpha$ in $\wC$ is a set homeomorphic to the unit interval $[0,1]$. We denote by $\inte(\al)$ the {\em set of interior points  of 
$\alpha$}, that is, the points corresponding to points in  $(0,1)$ under a homeomorphism between $[0,1]$ and $\alpha$. 

 \begin{lem} \label{lem:circarc} Suppose 
  $\alpha$ is an arc on a circle $K\sub \wC$ with endpoints $a,b\in K$, and $f$ is a
   $\wC$-valued continuous map on $\alpha$  with the following 
  properties: 
 \begin{enumerate}[label=\normalfont{(\alph*)}]
\item\label{i:cir1} $f(\alpha)\sub K$, 

\smallskip
\item\label{i:cir2} $f$ has no fixed points in $\inte(\alpha)$, 

\smallskip
\item\label{i:cir3}  near each point 
$z_0\in \inte(\alpha)$ the function $f$ has a locally injective meromorphic extension. 
\end{enumerate}
Then the following statements are true:

\begin{enumerate}[label=\normalfont{(\roman*)}]
\item\label{i:ciri} 
If $f(a)\not\in \alpha$, then $f$ is a homeomorphism of $\alpha$ 
onto its image $f(\alpha)$. 

\smallskip
\item\label{i:cirii} Suppose $f(a)=a$, $f(b)=b$, and that for all 
  points 
$p\in \inte(\alpha)$ sufficiently close  to $a$ the point $f(p)$ lies in the interior of  the subarc of $\alpha$ with endpoints  $a$ and $p$. Then $f$ is a homeomorphism of $\alpha$ onto itself.  Moreover,  in this case 
$f(p)$ always lies in the interior of  the subarc of $\alpha$ with endpoints  $a$ and $p$
for $p\in \inte(\alpha)$. 
\end{enumerate}
\end{lem}

Here a circle in $\wC$ is either a Euclidean  circle in $\C$ or a line considered as a circle through the point $\infty$.
 
The local injectivity condition in \ref{i:cir3}  is equivalent to the 
requirement that for each point $z_0\in \inte(\alpha)$  we have $f'(z_0)\ne 0$ if $f(z_0)\in \C$ or that $z_0$ is only a simple pole if $f(z_0)=\infty$. Of course,  \ref{i:cir3} could be replaced with a weaker condition not referring to meromorphic  extensions, but the formulation in this form  will be convenient for us later on.

\begin{proof} It is easiest to think of these statements in terms of moving points on $K$. The ensuing arguments can easily be made rigorous by reducing to  the circle $K=\R \cup\{\infty\}$ by  conjugating $f$ with a suitable M\"obius transformation  and then  invoking the Intermediate Value Theorem.

We consider a point $p$ traveling from $a$ to $b$ along the arc  $\alpha$. Then our assumptions \ref{i:cir1} and  \ref{i:cir3} imply  the image point $f(p)$ starts from $f(a)$ and moves on $K$ in a monotonic way, either in the same or in the opposite direction as $p$. These moving points  can only  collide at the beginning or at the end of the movement, because $f$ has no fixed point  in   $\inte(\alpha)$. Now if $f(a)\not \in \alpha$, then  is clear that in both cases $f(p)$ can never cross over $b$ as this would lead to a forbidden collision.
 
 In particular,    $f(p)$ cannot wrap  around $K$ and will sweep out an arc on $K$ with endpoints $f(a)$ and $f(b)$. The first statement follows.
 
 For the second statement note that $f(p)$ has to move in the same direction  as  $p$ starting at $f(a)=a$, because otherwise   $f(p)$ would leave 
 the arc $\alpha$,  contradicting our assumptions on the image of points close to $a$. Moreover, 
   while $f(p)$ moves in the same direction as $p$ on $\alpha$, at least initially  $f(p)$  is trailing behind $p$. But $f(p)$ can never overtake $p$ as this would force a forbidden collision. 
 Since $f(b)=b$, it follows that $f(p)$ will sweep out the arc $\alpha$
 as $p$ travels from $a$ to $b$ along $\alpha$. Moreover,  
  $f(p)$ always lies in the interior of  the subarc of $\alpha$ with endpoints  $a$ and $p$
for all $p\in \inte(\alpha)$. 
\end{proof} 

  \section{Mapping behavior of  $s_4$}\label{sec:maps4}
  
 We now analyze the mapping behavior of the function of $s_4$ defined in \eqref{eq:s4def} in more detail. This  will lead to a proof of Theorem~\ref{thm:s4} and its corollaries.
 For the proof of Theorem~\ref{thm:s4} we want to apply 
Proposition~\ref{prop:arg} for $f=s_4$ and $U=T_0$ as defined in \eqref{T0}.
This will require some preparation. 
  
   Recall that $E_4(\rho)=0=E_6(i)$. Using \eqref{eq:E2spec} we see that 
\[ s_4(\rho)=\rho,\quad s_4(i)=-i.
\]

The $q$-expansions \eqref{E2}--\eqref{E6} imply that for large $\im(\tau)$ the function $s_4$ can be represented in the form  
\begin{align}\label{eq:s4qser}
s_4(\tau) &=\tau -\frac{i}{120\pi}q^{-1} +i\sum_{n=0}^\infty a_nq^n\\
&=\tau -\frac{i}{120\pi}e^{-2\pi i \tau} +O(1) \text{ as $\im(\tau)\to +\infty$.} \label{eq:s4q}
 \end{align}  
 Here $a_n\in \R$ for $n\in \N_0$ and the $q$-series in 
 \eqref{eq:s4qser} converges when $|q|$ is sufficiently small, or equivalently, if  $\im(\tau)$ is sufficiently large. 
 
From \eqref{eq:s4q} we see  that $s_4(\tau)\to \infty$ as $\tau\in \H$ and 
$\im(\tau)\to +\infty$. This shows that we can continuously extend $s_4$ as a continuous map from the triangle  $T_0$  to $\wC$ by setting 
$s_4(\infty)=\infty$.

Moreover, for  $t>0$ large we have 
\begin{align}
s_4(it)&=it -\frac{i}{120\pi}e^{2\pi t} +O(1), \label{eq:s4imasym}  \\
s_4(\tfrac12+it)&=\tfrac12+it +\frac{i}{120\pi}e^{2\pi t} +O(1).  
\label{eq:s4linsym}
\end{align}

In the following, we  denote by 
\begin{align}\label{eq:sidesT0}
A&\coloneq \{it: t\ge 1\}\cup\{\infty\},\\
B&\coloneq \{e^{it}: \pi/3\le t\le  \pi/2 \},\notag\\
C&\coloneq \{\tfrac12+it: t\ge \sqrt3/2 \} \cup\{\infty\}\notag
\end{align}
the three sides of the circular arc triangle $T_0$ and by  
\begin{align} \label{eq:sidesX0}
A'&\coloneq \{it: t\le -1\}\cup\{\infty\}, \\
B'&\coloneq \{e^{it}: \pi/3\le t\le  3\pi/2 \},\notag \\
C'&\coloneq \{\tfrac12+it: t\ge \sqrt3/2 \} \cup\{\infty\} \notag
\end{align}
the three sides of $X_0$ as described in Theorem~\ref{thm:s4}.

\begin{lem}\label{lem:s4injinfty}
The map $s_4$ is injective on $T_0$ near $\infty$. 
\end{lem}
\begin{proof} This follows from the expansion \eqref{eq:s4qser}
and  Lemma~\ref{lem:injinfty}~\ref{item:inj2} applied for $\alpha=0$ and $\beta=1/2$.  
\end{proof}

\begin{lem}\label{lem:s4fixed}
A point $\tau_0\in \H$ is a fixed point of $s_4$ if and only if $\tau_0\in \Ga(\rho)$. 
\end{lem}
\begin{proof} The definition of $s_4$ shows that $\tau_0\in \H$ can only be a fixed point of $s_4$ if $E_4(\tau_0)=0$. Then 
$\tau_0\in \Ga(\rho)$ by \eqref{E4zeros}. Conversely, if $\tau_0\in \Ga(\rho)$, then $E_4(\tau_0)=0$ and $(E_2E_4-E_6)(\tau_0)=-
E_6(\tau_0)\ne 0$, because $E_4$ and $E_6$ have no common zeros. It follows that then $\tau_0$ is a fixed point of $s_4$. 
\end{proof}

An immediate consequence of the previous lemma is the fact that 
the only fixed points of $s_4$ in $T_0$ are the points $\rho$ and 
$\infty$. In particular, no interior point of any of the sides $A$, $B$, $C$ of $T_0$ is a fixed point of $s_4$.

\begin{lem}\label{lem:s4homeo} The map $s_4$ is a homeomorphism of $\partial T_0$ onto $\partial X_0$. Moreover, if we orient   $\partial T_0$ 
so that $T_0$ lies on the left, then  $X_0$ lies on the left of $\partial X_0$ if we equip $\partial X_0$ with the orientation induced  by the map $s_4$.  
\end{lem}
\begin{proof} We know that no interior point of any of the sides of $T_0$ is a fixed point of  the meromorphic function $s_4$  (Lemma~\ref{lem:s4fixed}) and that $s_4$  is locally injective near each point in $\H$ 
(Corollary~\ref{cor:locinj}~\ref{i:loc4}). So condition~\ref{i:cir2} and  condition~\ref{i:cir3} in Lemma~\ref{lem:circarc} are true for $s_4$ and each side of $T_0$.

It  follows directly from the definitions that $E_2(\tau)$, $E_4(\tau)$,  $E_6(\tau)$ take real values when $\tau\in \H$ and 
$\re(\tau)=0$ or $\re(\tau)=1/2$, because then $q=e^{2\pi i\tau}\in \R$. 
This implies that  $s_4(\tau)\in L_{0}$ (see \eqref{eq:imax}) when   $\re(\tau)=0$ and 
 $s_4(\tau)\in L_{1/2}$ (see \eqref{eq:12line}) when $\re(\tau)=1/2 $. 
 As a consequence,  $s_4$ sends the arc $A\sub L_{0}$ into $ L_{0}$ and the arc $C \sub L_{1/2}$  into 
   $L_{1/2}$.

  Since $s_4(A)\sub L_{0}$ and $s_4(i)=-i\not\in A$, we can apply Lemma~\ref{lem:circarc}~\ref{i:ciri} (with $L_{0}$ taking the role of $K$) and conclude that $s_4$ sends $A$ homeomorphically 
onto its image $s_4(A)$. This image must be one of the two subarcs
of  $L_{0}$ joining $s_4(i)=-i$ and $s_4(\infty)=\infty$. 
Now \eqref{eq:s4imasym} shows that $\im(s_4(it))$ takes  large negative values if $t>0$ is large. Therefore, we must have $s_4(A)=A'$ and so 
$s_4$ is a homeomorphism of $A$ onto $A'$. In particular,  the point $s_4(\tau)$  moves monotonically
from $\infty$ to $-i$ along $A'$ if $\tau$ moves from $\infty$ 
to $i$ along $A$. So if we equip $A'$ with this orientation, then  
$X_0$ lies on the left of $A'$.

We have $s_4(C)\sub L_{1/2}$,  $s_4(\infty)=\infty$, and 
 $s_4(\rho)=\rho$. Moreover, \eqref{eq:s4linsym} shows that for $t>0$  large enough we have $\im(s_4(\tfrac12+it))>t$. For such $t$
 the image $s_4(\tfrac12+it)$ of $p=\tfrac12+it$ lies in the interior 
 of the subarc of 
 $C$ with  endpoints $\infty$ and $p$. Lemma~\ref{lem:circarc}~\ref{i:cirii} (with $L_{1/2}$ taking the role of $K$) shows that 
 $s_4$ is a homeomorphism  of $C$ onto itself. Therefore,  the point $s_4(\tau)$  moves monotonically
from $\rho$ to $\infty$ along $C'=C$ if $\tau$ moves from $\rho$ 
to $\infty$ along $C$. So if we equip $C'$ with this orientation, then  
$X_0$ lies on the left of $C'$.  
 
 We know by Proposition~\ref{prop:poly46} that 
\[s_4(1/\overline {\tau})=1/\overline {s_4(\tau)}\] 
for  $\tau\in \H$.
If $\tau\in B$, then $\tau$ lies on the unit circle $\partial \D$ and so $\tau=1/\overline \tau$. 
This implies that
\[ {s_4(\tau)}= s_4(1/\overline \tau)=1/\overline {s_4(\tau)}\] 
  and so $|s_4(\tau)|=1$ for $\tau\in B$. We see that $s_4$ sends $B\sub \partial \D$ into $\partial \D$. Since $s_4(i)=-i\not\in B$, Lemma~\ref{lem:circarc}~\ref{i:ciri} implies that  $s_4$ is a  homeomorphism of $B$ onto its image $s_4(B)$. 
  
  This image is one of the subarc of $\partial \D$ with endpoints $s_4(i)=-i$ and $s_4(\rho)=\rho$.   Here we actually  must have $s_4(B)=B'$ as follows by considering orientations. Indeed, since $s_4$ is locally injective near $i$,
  this map  is  conformal and hence orientation-preserving   near $i$. We know that $s_4$ sends $A$ oriented  from $\infty$ to $i$ onto the arc $A'$ oriented from $\infty$ to $-i$. Since $B$ lies on  the left of the oriented arc $A$ locally near $i$, the image $s_4(B)$ must lie on the left of $A'$ 
near $s_4(i)=-i$. This only leaves the possibility that $s_4$ sends $B$ homeomorphically onto $B'$. Moreover, 
the point $s_4(\tau)$  moves monotonically
from $-i$ to $\rho$ along $B'$ if $\tau$ moves from $i$ 
to $\rho$ along $B$. So if we equip $B'$ with this orientation, then  
$X_0$ lies on the left of $B'$.

We conclude that $s_4$ sends the three sides $A$, $B$, $C$ of $T_0$ homeomorphically onto the three sides $A'$, $B'$, $C'$ of $X_0$, respectively. 
It follows that $s_4$ is a homeomorphism of $\partial T_0$ onto $\partial X_0$. Moreover, our previous considerations show that if we orient $\partial T_0$ so that $T_0$ lies on the left of $\partial T_0$, then $X_0$ lies on the left of $\partial X_0$ equipped with 
induced orientation under $s_4$. 
\end{proof}

\begin{lem} \label{lem:incronB} There exists a strictly increasing continuous function 
\[\varphi\: 
[\pi/3, \pi/2]\ra [\pi/3, 3\pi/2]\] such that 
$s_4(e^{it})=e^{i\varphi(t)}$ for  $t\in [\pi/3, \pi/2]$ 
and $\varphi(t)>t$ for $t\in  (\pi/3, \pi/2]$.  
\end{lem}

\begin{proof} We have seen in the proof of Lemma~\ref{lem:s4homeo} that 
$s_4$ maps the arc $B\sub \partial \D$ homeomorphically onto the arc $B'\sub \partial \D$. It  follows that there exists a strictly increasing function 
\[\varphi\: 
[\pi/3, \pi/2]\ra [\pi/3, 3\pi/2]\] such that 
$s_4(e^{it})=e^{i\varphi(t)}$ for  $t\in [\pi/3, \pi/2]$. 
Since $s_4(\rho)=\rho$ and $s_4(i)=-i$, we have $\varphi(\pi/3)=\pi/3$ and $\varphi(\pi/2)=3\pi/2>\pi/2$. The last inequality implies that on 
$\varphi(t)>t$ for all $t\in (\pi/3, \pi/2]$, because otherwise 
there would be a point $t_0\in  (\pi/3, \pi/2]$ with $\varphi(t_0)=t_0$
by the intermediate value theorem. But then $\tau_0=e^{it_0}$ would be a fixed point of $s_4$ in $B\setminus\{\rho\}\sub \H\setminus \Ga(\rho)$, but we know that $s_4$ has no  fixed points in $\H \setminus \Ga(\rho)$ by Lemma~\ref{lem:s4fixed}.   
\end{proof}

In order to apply Proposition~\ref{prop:arg} for $f=s_4$ and $U=T_0$, we want to show that $s_4$ has no poles in $\inte(T_0)$ and is hence holomorphic there. One way to do this is to use numerical estimates that rule out that the function 
$E_2E_4-E_6$  giving rise to poles of $s_4$ (see  \eqref{eq:s4def}) has any zeros 
in $\inte(T_0)$. We prefer a different approach that avoids 
numerical estimates. 

For this we consider the auxiliary function
\begin{equation}\label{eq:gdef}
g(\tau)\coloneq \frac{1}{\tau-s_4(\tau)}=-\frac{i\pi}{6}\cdot\frac {E_2E_4-E_6}{E_4}.
\end{equation}
This is a meromorphic function on $\H$ with poles exactly 
at the points where $E_4$ vanishes, that is, the points in $\Ga(\rho)$. In particular, $g$ is holomorphic in $\inte(T_0)$. The idea is now to show that $\re(g(\tau))>0$ for $\tau \in \inte(T_0)$. 
Then necessarily $E_2(\tau)E_4(\tau)-E_6(\tau)\ne 0$ for  
$\tau \in \inte(T_0)$. This in turn implies that $s_4$ has no poles in 
$\inte(T_0)$. In order to show that $g$ has positive real part on $\inte(T_0)$,
we analyze the behavior of $g$ near points 
$\tau\in \partial T_0$. 

\begin{lem} \label{lemgs4sides} We have $\re(g(\tau))=0$ for $\tau \in \inte(A)\cup \inte 
(C)$ and $\re(g(\tau))>0$ for $\tau \in \inte(B)$. 
\end{lem}

\begin{proof}  The first statement follows immediately from 
the definition \eqref{eq:gdef} of $g$, because on the set  $\inte(A)\cup \inte 
(C)$ the functions $E_2$, $E_4$, $E_6$ only take real values. 

 If $w\in \C$ and $\re(w)>0$, then also $\re(1/w)>0$.
Therefore, it is enough to establish the inequality 
$\re(\tau-s_4(\tau))>0$ for 
$\tau\in \inte(B)$. To see this, let $\varphi$ be the function from Lemma~\ref{lem:incronB}. Then it suffices to show that $\re(e^{it}-e^{i\varphi(t)})>0$ for 
$t\in (\pi/3,\pi/2)$. Essentially, this follows from the fact 
that 
\[\pi/3< t<\varphi(t)\le 3\pi/2\]  for $t\in (\pi/3,\pi/2)$, and so 
the point $s_4(e^{it})=e^{i\varphi(t)}$ lies in the half-plane 
$$\{ z\in \C: \re(z)<\re(e^{it})=\cos t\}. $$ More formally,  note that for 
$t\in (\pi/3,\pi/2)$ we have 
 $t+\varphi(t)\in (0,2\pi)$ and 
\[ 0>t-\varphi(t) >\pi/3-3\pi/2=-7\pi/6>-2\pi.  
\]
Hence $(t+\varphi(t))/2\in (0,\pi)$ and $(t-\varphi(t))/2\in (-\pi,0)$, which implies that 
\[
\sin((t+\varphi(t))/2)>0 \text{ and } \sin((t-\varphi(t))/2)<0
\]
and so
\begin{align*}
\re(e^{it}-e^{i\varphi(t)})&=\cos t-\cos\varphi(t)\\
&=-2\sin((t+\varphi(t))/2)
\sin((t-\varphi(t))/2)>0. \qedhere
\end{align*}
\end{proof}

\begin{lem} \label{lem:gs4vertex} We have 
\[
\liminf_{\inte(T_0)\ni \tau \to \tau_0} \re(g(\tau))\ge 0
\]
for $\tau_0\in \{i, \rho, \infty\}$
\end{lem}

\begin{proof} Let 
\[ f(\tau)\coloneq \tau-s_4(\tau)=\frac{6iE_4}{\pi(E_2E_4-E_6)}
\] for $\tau\in \H$.  Note that $s_4'(i)=-5E_4(i)/E_2(i)^2$ by \eqref{eq:ders4}. Since $E_2(i)>0$ by 
\eqref{eq:E2spec} and $E_4(i)>0$ as easily follows from \eqref{E4}, we see that $f'(i)=1-s'_4(i)>0$. Moreover, $f$ sends the imaginary axis near $i$ into itself.     Then Lemma~\ref{lem:orient} implies that
$\re(f(\tau))>0$ and so $\re(g(\tau))=\re(1/f(\tau))>0$ 
 for $\tau \in \inte(T_0)$ near $i$. The statement for $\tau_0=i$ follows.

 By \eqref{eq:ders4} we have $s_4'(\rho)=5$ and so  
$f'(\rho)=-4<0$. The map $f$ sends the line $L_{1/2}$ near $\rho$ into the imaginary axis. It  easily follows from Lemma~\ref{lem:orient} that 
$\re(f(\tau))>0$ and so $\re(g(\tau))=\re(1/f(\tau))>0$ for $\tau \in \inte(T_0)$ near $\rho$.  The statement for $\tau_0=\rho$ follows.    
  
Finally, \eqref{eq:s4q} and \eqref{eq:gdef} show  that for $\tau\in \H$ with $\im(\tau)$ large we have 
\[ 
g(\tau)=-120\pi iq (1+O(q)),
\]
and so 
\[
\liminf_{\inte(T_0)\ni \tau \to \infty} \re(g(\tau))=0.
\]
The proof is complete.
\end{proof}

\begin{lem}  \label{lem:gs4pos}
We have $\re(g(\tau))>0$ for $\tau\in \inte(T_0)$. 
\end{lem}

\begin{proof} By Lemma~\ref{lem:minprincip} it is enough to show that
\[
\liminf_{\inte(T_0)\ni \tau \to \tau_0} \re(g(\tau))\ge 0
\]
for all $\tau_0\in \partial T_0$. This follows from Lemmas \ref{lemgs4sides} and~\ref{lem:gs4vertex}. 
\end{proof}

\begin{lem} \label{lem:s4nopoles} The function $s_4$ has no poles in $\inte(T_0)$. 
\end{lem}
\begin{proof} By Lemma~\ref{lem:gs4pos} the function  $g$ defined in \eqref{eq:gdef} has no zeros in $\inte(T_0)$. This implies that $s_4$ cannot have poles in $\inte(T_0)$. \end{proof}

 \begin{proof}[Proof of Theorem~\ref{thm:s4}]
 We can apply Proposition~\ref{prop:arg} to $f=s_4$, $U=T_0$, and $V=X_0$ 
   by Lemmas~\ref{lem:s4homeo},  \ref{lem:s4injinfty} and \ref{lem:s4nopoles}. The statement follows.
  \end{proof}

 \begin{figure}
 \begin{overpic}[ scale=0.7
    ]{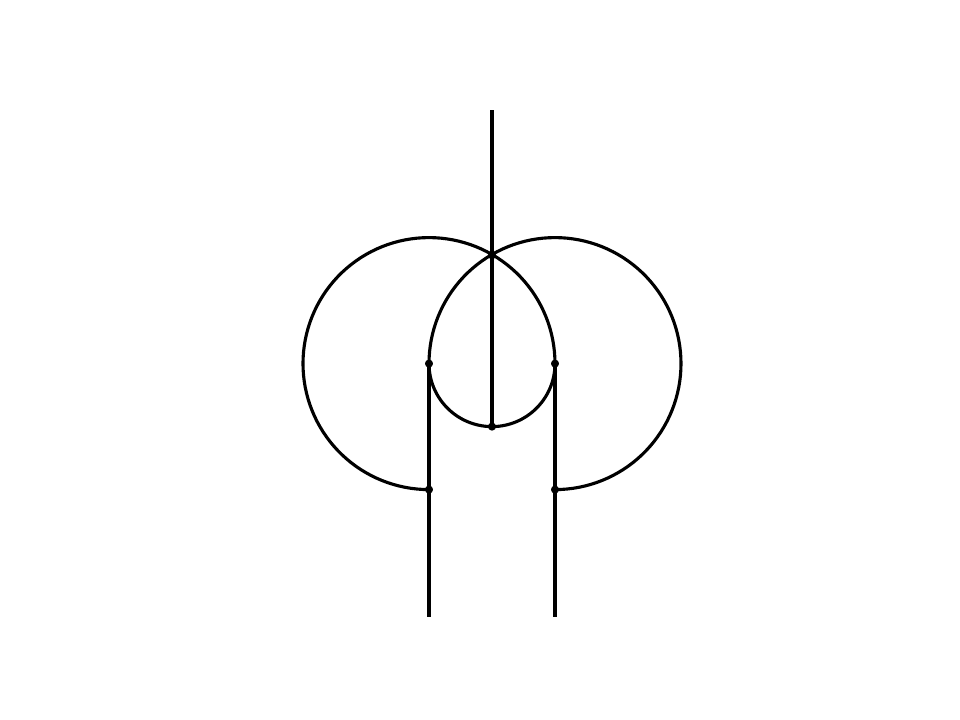}

      \put(35,55){ $X_0$}
       \put(40,45){ $X_1$}
         \put(45,38){ $X_2$}
           \put(51,38){ $X_3$}
              \put(57,45){ $X_4$}
               \put(62,55){ $X_5$}
              \put(40,36){ $0$}
               \put(51,51){ $\rho$}
              \put(59,36){ $1$}
                   \put(47,26){ $\tfrac{1-i}2$}
        \put(38,21){ $-i$}
                   \put(58, 21){ $1-i$}
               \end{overpic}
  \vspace{-1cm}
\caption{Images of the six subtriangles of $V_0$ under $s_4$.}
\label{fig:covs4}
\end{figure}
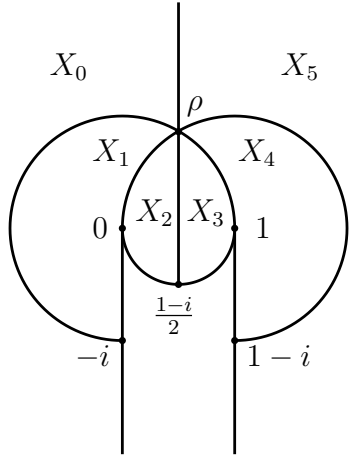

 \begin{proof}[Proof of Corollary~\ref{cor:s4}] We will just give the basic idea omitting some  details. 
 
 By Theorem~\ref{thm:s4} the map 
 $s_4$ is a conformal map of $T_0$ onto $X_0$. One can now apply successive Schwarz reflections. If we do this with the triangle $T_0$ 
 in Figure~\ref{fig:T} and reflect in appropriate sides, then one obtains the triangles $T_0, \dots, T_5$ as indicated in the figure. Note that reflection in any of the sides of these triangles and, more generally, in any side of a triangle in the tessellation tessellation $\mathcal{T}$, is given by elements in the extended modular group $\overline{\Ga}$.

By Proposition~\ref{prop:poly46} the images under these reflections  undergo the same operation as the original triangle in the tessellation $\mathcal{T}$. For example, to obtain $T_1$, we reflect $T_0$ in the unit circle. This means that to obtain the image $X_1=s_4(T_1)$, we also have to reflect $X_0$ also in the unit circle. 
 In follows that the images $X_k\coloneq s_4(T_k)$, $k=0, \dots, 5$ can be represented as in Figure~\ref{fig:covs4}. Each of them is a circular arc triangle with angles  $\pi, \pi/2, \pi/3$. Note that 
 the union of these images is equal to $\wC\setminus \overline{\inte(V_0)}$, where $V_0$ is the circular arc triangle given in \eqref{eq:V0}. 
 
These considerations  imply $\overline{s_4}$ is an anti-conformal map that sends the circular arc triangle $V_0$ onto  the complementary circular arc triangle $\wC\setminus \inte(V_0)$. 
Here the vertices of $V_0$ are fixed and each side of $V_0$ is mapped homeomorphically onto itself. 

 If $V$ is an arbitrary circular arc triangle from the tessellation $\mathcal{V}$, then there exists a unique  $\varphi\in \overline {\Ga}(2)\sub \overline{\Ga}$
such that $V=\varphi(V_0)$. Since each element in $\overline{\Ga}$ preserves the real axis, we have $\varphi(\overline{\tau})=\overline{\varphi(\tau)}$ for $\tau\in \wC$.  Then it follows from  \eqref{eq:poly46} that
\[
\varphi \circ \overline{s_4}=\overline{\varphi \circ s_4}=\overline{s_4\circ \varphi}
=\overline{s_4}\circ \varphi. 
\] 
This identity implies  that $\overline s_4$ is an anti-conformal map of $V=\varphi(V_0)$ onto its complementary triangle $\varphi(\wC\setminus \inte(V_0))=\wC\setminus \inte(V)$ as described in the statement.  
 \end{proof}
 
  \begin{proof}[Proof of Corollary~\ref{cor:crit4}]
   We know that $\tau\in \H$ is a critical point of $E_4$ if and only of $s_4(\tau)=\infty$ (see Lemma~\ref{lem:critpol46}). Now the considerations 
  in the proof of Corollary~\ref{cor:s4} show that for $\tau\in V_0$ we have 
  $s_4(\tau)=\infty$ if and only if $\tau=\infty$ (essentially, this follows from Figure~\ref{fig:covs4} by inspection). We conclude that 
  $E_4$ has  no critical points in $V_0\cap \H$. 
  
  If $V\in \mathcal{V}$ and $\infty\in V$, then $V=V_0+n$ for some $n\in \Z$. 
  Since we have the identity $E_4(\tau+n)=E_4(\tau)$ for $\tau\in \H$, it follows that $E_6$ has no critical points in $V\cap \H$.
  
  Now let $V\in \mathcal{V}$ with $\infty\not \in V$ be arbitrary. Then it follows from 
   Corollary~\ref{cor:s4} that  the map $s_4$ attains the value $\infty$ exactly once in $V$ and this unique preimage of $\infty$ in $V$ must  lie in $\inte(V)$.   
  Since the critical points of $E_4$ are simple by 
   Proposition~\ref{prop:Esimcrit}, the statement follows. 
  \end{proof}

  \section{Mapping behavior of  $s_6$} \label{sec:maps6}

  Analyzing the mapping behavior of $s_6$ runs along very similar lines as for $s_4$ with small modifications. 
  
   Using $E_4(\rho)=0=E_6(i)$ and \eqref{eq:E2spec},  we see that 
$$ s_6(\rho)=\overline {\rho},\quad s_6(i)=i.$$
For large $\im(\tau)$ the function $s_6$ can be represented in the form  
\begin{align}\label{eq:s6qser}
s_6(\tau) &=\tau +\frac{i}{168\pi}q^{-1} +i \sum_{n=0}^\infty a_nq^n\\
&=\tau +\frac{i}{168\pi}e^{-2\pi i \tau} +O(1) \text{ as $\im(\tau)\to +\infty$.}  \label{eq:s6q}
 \end{align}  
 Here $a_n\in \R$ for $n\in \N_0$ and the $q$-series in 
 \eqref{eq:s6qser} converges when $|q|$ is sufficiently small, or equivalently, if  $\im(\tau)$ is sufficiently large. 
 
From \eqref{eq:s6q} we see  that $s_6(\tau)\to \infty$ as $\tau\in \H$ and 
$\im(\tau)\to +\infty$. This shows that we can continuously extend $s_6$ as a continuous map from $T_0$ to $\wC$ by setting 
$s_6(\infty)=\infty$. 

For  $t>0$ large we have 
\begin{align}
s_6(it)&=it +\frac{i}{168\pi}e^{2\pi t} +O(1), \label{eq:s6imasym}  \\
s_6(\tfrac12+it)&=\tfrac12+it -\frac{i}{168\pi}e^{2\pi t} +O(1).  
\label{eq:s6linsym}
\end{align}

In the following, we  again denote by $A$, $B$, $C$ the sides of $T_0$ as in \eqref{eq:sidesT0}. Let   
\begin{align*}
A'&\coloneq \{it: t\ge 1\}\cup\{\infty\},\\
B'&\coloneq \{e^{it}: -\pi/3\le t\le  \pi/2 \},\\
C'&\coloneq \{\tfrac12+it: t\le- \sqrt3/2 \} \cup\{\infty\}
\end{align*}
be the three sides of $Y_0$ as described in Theorem~\ref{thm:s6}.

\begin{lem}\label{lem:s6injinfty}
The map $s_6$ is injective on $T_0$ near $\infty$. 
\end{lem}
\begin{proof} This follows from the expansion \eqref{eq:s6qser}
and Lemma~\ref{lem:injinfty}~\ref{item:inj2} applied for $\alpha=0$ and $\beta=1/2$.  
\end{proof}

\begin{lem}\label{lem:s6fixed}
A point $\tau_0\in \H$ is a fixed point of $s_6$ if and only if $\tau_0\in \Ga(i)$. 
\end{lem}
\begin{proof} The definition of $s_6$ shows that $\tau_0\in \H$ can only be a fixed point of $s_6$ if $E_6(\tau_0)=0$. Then 
$\tau_0\in \Ga(i)$ by \eqref{E6zeros}. Conversely, if $\tau_0\in \Ga(i)$, then $E_6(\tau_0)=0$ and $(E_2E_6-E_4^2)(\tau_0)=-
E_4^2(\tau_0)\ne 0$, because $E_4$ and $E_6$ have no common zeros. It follows that then $\tau_0$ is a fixed point of $s_6$. 
\end{proof}

An immediate consequence of the previous lemma is the fact that 
the only fixed points of $s_6$ in $T_0$ are the points $i$ and 
$\infty$. In particular, no interior point of any of the sides $A$, $B$, $C$ of $T_0$ is a fixed point of $s_6$.

\begin{lem}\label{lem:s6homeo} The map $s_6$ is a homeomorphism of $\partial T_0$ onto $\partial Y_0$. Moreover, if we orient   $\partial T_0$ 
so that $T_0$ lies on the left, then  $Y_0$ lies on the left of $\partial Y_0$ if we equip under $\partial Y_0$ with the orientation induced  by $s_6$.  
\end{lem}
\begin{proof} We know that no interior point of any of the sides of $T_0$ is a fixed point of  the meromorphic function $s_6$  (Lemma~\ref{lem:s6fixed}). Moreover, by Corollary~\ref{cor:locinj}~\ref{i:loc6} the function $s_6$  is locally injective near each point in $\H\setminus \Ga(\rho)$ and hence near 
each  interior point of any of the sides of $T_0$. 
 So condition~\ref{i:cir2} and  condition~\ref{i:cir3} in Lemma~\ref{lem:circarc} are true for $s_6$ and each side of $T_0$.

Since  $E_2(\tau), E_4(\tau), E_6(\tau)\in \R$  when $\tau\in \H$ and 
$\re(\tau)=0$ or $\re(\tau)=1/2$, the map $s_6$ sends the arc $A\sub L_{0}$ into $ L_{0}$ and the arc $C \sub L_{1/2}$  into 
  $ L_{1/2}$.
  
We have  $s_6(i)=i$ and $s_6(\infty)=\infty$. Moreover, \eqref{eq:s6imasym} shows that for $t>0$  large enough we have $\im(s_6(it))>t$. For such $t$
 the image $s_6(it)$ of $p=it$ lies in the interior 
 of the subarc of 
 $A$ with  endpoints $\infty$ and $p$.  Lemma~\ref{lem:circarc}~\ref{i:cirii} (with $L_{0}$ taking the role of $\partial K$) shows that 
 $s_6$ is a homeomorphism  of $A$ onto itself. As a consequence, the point $s_6(\tau)$  moves monotonically
from $\infty$ to $i$ along $A'=A$ if $\tau$ moves from $\infty$ 
to $i$ along $A$. So if we equip $A'$ with this orientation, then  
$Y_0$ lies on the left of $A'$.

We have $s_6(C)\sub L_{1/2}$  and 
 $s_6(\rho)=\overline \rho \not \in C$. Lemma~\ref{lem:circarc}~\ref{i:ciri} (with $L_{1/2}$ taking the role of $K$)  shows that $s_6$ is a homeomorphism of $C$ onto its image $s_6(C)$. 
 This image is a one of the two subarcs of    $L_{1/2}$
with endpoints $s_6(\rho)=\overline \rho$ and $s_6(\infty)=\infty$.  
Now  \eqref{eq:s6linsym} shows that for $t>0$  large  $\im(s_6(\tfrac12+it))$ takes a large negative value. 
Hence $s_6(C)=C'$.  

It follows that  the  point $s_6(\tau)$  moves monotonically
from $\overline \rho$ to $\infty$ along $C'$ if $\tau$ moves from $\rho$ 
to $\infty$ along $C$. So if we equip $C'$ with this orientation, then  
$Y_0$ lies on the left of $C'$.   
 
 We know by Proposition~\ref{prop:poly46} that 
\[s_6(1/\overline {\tau})=1/\overline {s_6(\tau)}\] 
for  $\tau\in \H$.
If $\tau\in B$, then $\tau$ lies on the unit circle and so $\tau=1/\overline \tau$. 
This implies that
\[ {s_6(\tau)}= s_6(1/\overline \tau)=1/\overline {s_6(\tau)}\] 
  and so $|s_6(\tau)|=1$ for $\tau\in B$. 

We see that $s_6$ sends $B$ into the unit circle $\partial \D$. Since $s_6(\rho)=\overline \rho\not \in B$, Lemma~\ref{lem:circarc}~\ref{i:ciri} implies that $s_6$ sends $B$ homeomorphically onto its image $s_6(B)$. 

This image is one of the subarc of $\partial \D$ with endpoints $s_6(i)=i$ and $s_6(\rho)=\overline \rho$.   Here we actually  must have $s_6(B)=B'$ as follows by considering orientations. Indeed,
since $s_6$ is locally injective near $i$,
  this map  is  conformal and hence orientation-preserving 
map  near $i$.  We know that $s_6$ sends $A$ oriented from $\infty$ to $i$ into itself  oriented in the same way.  Since $B$ lies on  the left of the oriented arc $A$ locally near $i$, the image $s_6(B)$ must lie on the left of $A'=A$ 
near $s_6(i)=i$. This only leaves the possibility that $s_6$ sends $B$ homeomorphically onto $B'$. Moreover, 
the point $s_6(\tau)$  moves monotonically
from $i$ to $\overline \rho$ along $B'$ if $\tau$ moves from $i$ 
to $\rho$ along $B$. So if we equip $B'$ with this orientation, then  
$Y_0$ lies on the left of $B'$.

We conclude that $s_6$ sends the three sides $A$, $B$, $C$ of $T_0$ homeomorphically onto the three sides $A'$, $B'$, $C'$ of $Y_0$, respectively. 
It follows that $s_6$ is a homeomorphism of $\partial T_0$ onto $\partial Y_0$. Moreover, our  considerations shows that if we orient $\partial T_0$ so that $T_0$ lies on the left of $\partial T_0$, then $Y_0$ lies on the left of $\partial Y_0$ equipped with 
induced orientation under $s_6$. 
\end{proof}

\begin{lem} \label{lem:incronB6} There exists a strictly increasing continuous function 
\[\varphi\: 
[\pi/3, \pi/2]\ra [-\pi/3, \pi/2]\] such that 
$s_6(e^{it})=e^{i\varphi(t)}$ for  $t\in [\pi/3, \pi/2]$ 
and $\varphi(t)<t$ for $t\in  [-\pi/3, \pi/2)$.  
\end{lem}

\begin{proof} We have seen in the proof of Lemma~\ref{lem:s4homeo} that 
$s_6$ maps the arc $B\sub \partial \D$ homeomorphically onto the arc $B'\sub \partial \D$. It  follows that there exists a strictly increasing function 
\[\varphi\: 
[\pi/3, \pi/2]\ra [-\pi/3, \pi/2]\] such that 
$s_6(e^{it})=e^{i\varphi(t)}$ for  $t\in [\pi/3, \pi/2]$. 
Since $s_6(i)=i$ and $s_6(\rho)=\overline \rho$, we have $\varphi(\pi/2)=\pi/2$ 
and $\varphi(\pi/3)=-\pi/3<\pi/3$. The last inequality implies that  
$\varphi(t)<t$ for all $t\in [\pi/3, \pi/2)$, because otherwise 
there would be a point $t_0\in  [\pi/3, \pi/2)$ with $\varphi(t_0)=t_0$
by the intermediate value theorem. But then $\tau_0=e^{it_0}$ would be a fixed point of $s_6$ in $B\setminus\{i\}\sub \H\setminus \Ga(i)$, but we know that $s_6$ has no  fixed points in $\H \setminus \Ga(i)$ by Lemma~\ref{lem:s6fixed}.  
\end{proof}

In order to apply Proposition~\ref{prop:arg}  for $f=s_6$ and $U=T_0$, we have to show that $s_6$ has no poles in $\inte(T_0)$ and is hence holomorphic there. 

For this we consider the auxiliary function
\begin{equation}\label{eq:gdef6}
g(\tau)= \frac{1}{s_6(\tau)-\tau}=\frac{i\pi}{6}\cdot \frac {E_2E_6-E_4^2}{E_6}.
\end{equation}
This is a meromorphic function on $\H$ with poles exactly 
at the points where $E_6$ vanishes, that is, the points in $\Ga(i)$. In particular, $g$ is holomorphic in $\inte(T_0)$. 

\begin{lem} \label{lemgs6sides} We have $\re(g(\tau))=0$ for $\tau \in \inte(A)\cup \inte 
(C)$ and $\re(g(\tau))>0$ for $\tau \in \inte(B)$. 
\end{lem}

\begin{proof}  The first statement follows immediately from 
the definition \eqref{eq:gdef6} of $g$, because on the set  $\inte(A)\cup \inte 
(C)$ the functions $E_2$, $E_4$, $E_6$ only take real values. 

 If $w\in \C$ and $\re(w)>0$, then also $\re(1/w)>0$.
Therefore, it is enough to establish the inequality 
$\re(s_6(\tau)-\tau)>0$ for 
$\tau\in \inte(B)$. To see this, let $\varphi$ be the function from Lemma~\ref{lem:incronB6}. Then it suffices to show that $\re(e^{i\varphi(t)}-e^{it})>0$ for 
$t\in (\pi/3,\pi/2)$. Note that for 
$t\in (\pi/3,\pi/2)$ we have 
 $\varphi(t)+t\in (0,\pi)$ and 
\[ 0>\varphi(t)-t >-\pi/3-\pi/2=-5\pi/6>-\pi.  
\]
Hence $(\varphi(t)+t)/2\in (0,\pi/2)$ and $(\varphi(t)-t)/2\in (0,-\pi/2)$, which implies that 
\[
\sin((\varphi(t)+t)/2)>0 \text{ and } \sin((\varphi(t)-t)/2)<0
\]
and so
\begin{align*}
\re(e^{i\varphi(t)}-e^{it})&=\cos\varphi(t)- \cos t\\
&=-2\sin((\varphi(t)+t)/2)
\sin((\varphi(t)-t)/2)>0.\qedhere
\end{align*}\end{proof}

\begin{lem} \label{lem:gs6vertex} We have 
\[
\liminf_{\inte(T_0)\ni \tau \to \tau_0} \re(g(\tau))\ge 0
\]
for $\tau_0\in \{i, \rho, \infty\}$
\end{lem}

\begin{proof} Let 
\[ f(\tau)\coloneq s_6(\tau)-\tau=-\frac{6iE_6}{\pi(E_2E_6-E_4^2)}
\] for $\tau\in \H$.  Note that $s_6'(i)=7$ by  \eqref{eq:ders6} and so $f'(i)=6>0$. Moreover, $f$ sends the imaginary axis near $i$ into itself.     Then Lemma~\ref{lem:orient} implies that
$\re(f(\tau))>0$ and so $\re(g(\tau))=\re(1/f(\tau))>0$ 
 for $\tau \in \inte(T_0)$ near $i$. The statement for $\tau_0=i$ follows.

 By \eqref{eq:ders6} we have $s_6'(\rho)=0$ and so  
$f'(\rho)=-1<0$. The map $f$ sends the line $L_{1/2}$ near $\rho$ into the imaginary axis. It  easily follows from Lemma~\ref{lem:orient} that 
$\re(f(\tau))>0$ and so $\re(g(\tau))=\re(1/f(\tau))>0$ for $\tau \in \inte(T_0)$ near $\rho$.  The statement for $\tau_0=\rho$ follows.    
  
Finally, \eqref{eq:s6q} and \eqref{eq:gdef6} show  that for $\tau\in \H$ with $\im(\tau)$ large we have 
\[ 
g(\tau)=-168\pi iq (1+O(q)),
\]
and so 
\[
\liminf_{\inte(T_0)\ni \tau \to \infty} \re(g(\tau))=0.
\]
The proof is complete.
\end{proof}

\begin{lem}  \label{lem:s6gpos}
We have $\re(g(\tau))>0$ for $\tau\in \inte(T_0)$. 
\end{lem}

\begin{proof} By Lemma~\ref{lem:minprincip} it is enough to show that
\[
\liminf_{\inte(T_0)\ni \tau \to \tau_0} \re(g(\tau))\ge 0
\]
for all $\tau_0\in \partial T_0$. This follows from Lemmas \ref{lemgs6sides} and~\ref{lem:gs6vertex}. 
\end{proof}

 \begin{figure}
 \begin{overpic}[ scale=0.8
    ]{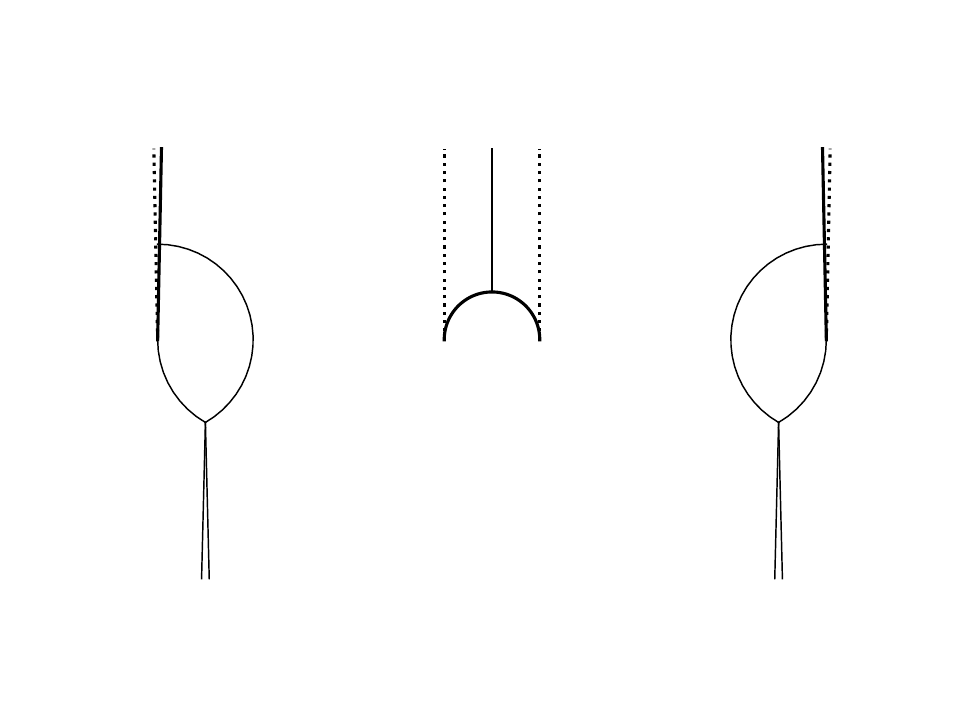}
\put(23,55){ $Y_0$}
\put(73,55){ $Y_5$}
   \put(18,40){ $Y_1$}
    \put(78,40){ $Y_4$}
         \put(9.5,48){ $Y_2$}
           \put(88,48){ $Y_3$}
            \put(46,48){ $Y_2$}
           \put(51,48){ $Y_3$}
           \put(12,55){ $a$}
           \put(46,55){ $a$}
            \put(52.5,55){ $c$}
            \put(86.5,55){ $c$}
            \put(17,20){ $b'$}
            \put(21.5,20){ $b$}
           \put(77,20){ $b$}
            \put(81.5,20){ $b'$}              
          \end{overpic}
  \vspace{-1cm}
\caption{Images of the six subtriangles of $V_0$ under $s_6$.}
\label{fig:covs6}
\end{figure}
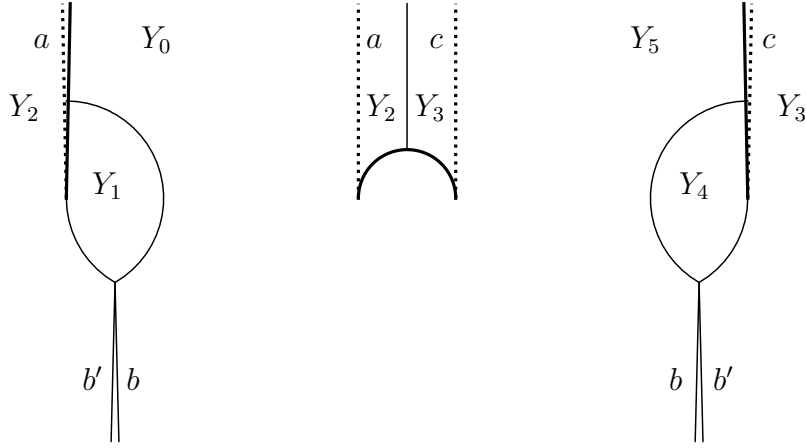

\begin{lem} \label{lem:s6nopoles} The function $s_6$ has no poles in $\inte(T_0)$. 
\end{lem}
\begin{proof} By Lemma~\ref{lem:s6gpos} the function  $g$ has no zeros in $\inte(T_0)$. This implies that $s_6$ cannot have poles in $\inte(T_0)$. \end{proof}

 \begin{proof}[Proof of Theorem~\ref{thm:s6}]
 We can apply Proposition~\ref{prop:arg} to $f=s_6$, $U=T_0$, and $V=Y_0$ by Lemmas~\ref{lem:s6homeo},  \ref{lem:s6injinfty} and \ref{lem:s6nopoles}. The statement follows.
  \end{proof}

   \begin{proof}[Proof of Corollary~\ref{cor:s6}] Again we only outline the argument leaving some details to the reader. 
   
   We know that 
   $s_6$ maps the circular arc triangle $T_0$ conformally onto the circular arc triangle $Y_0$ as defined in the statement of Theorem~\ref{thm:s6} (see Figure~\ref{fig:s6}).  The image $Y_k\coloneq s_6(T_k)$ of the  circular arc triangle $T_k$, $k=0, \dots, 5$, as indicated in Figure~\ref{fig:T} can be obtained by successive Schwarz reflections. These reflections are again given by elements in $\overline{\Ga}$. By  Proposition~\ref{prop:poly46}
 triangles in $\mathcal{T}$ and their images under $s_6$ undergo the same 
 transformations under these reflections.  
  The images $Y_k=s_6(T_k)$ for $k=0,\dots,5$ obtained in this way  are represented in Figure~\ref{fig:covs6}. Note that in the figure, $Y_2$ consists of two pieces (one indicated  on the   left and  one in the middle picture) that have 
  to be glued together by identifying corresponding points  on the dotted arcs $a$.
  A similar remark applies to $Y_3$ (here the identification of the two pieces is along the dotted arcs $c$).  
  
It easily follows from  Theorem~\ref{thm:s6} and this Schwarz reflection process that for $k=0,\dots, 5$ the map $s_6$ sends the side of $T_k$ that does not contain  $\rho$ homeomorphically onto itself with the endpoints of this side fixed.  Since the boundary $\partial V_0$ of 
the circular arc triangle $V_0$ defined in 
\eqref{eq:V0} consists of the non-overlapping union of theses sides of $T_k$, 
$k=0,\dots, 5$, it follows that $s_6$ sends $\partial V_0$ homeomorphically onto itself.

Since  $V_0$  is equal to the union of the triangles  $T_k$, $k=0, \dots, 5$, we also see that $s_6$ maps  $V_0$  biholomorphically onto a Riemann surface $R_0$ with boundary that is spread over the Riemann sphere. As indicated in Figure~\ref{fig:covs6}, $R_0$ is obtained by gluing together
  a copy of the Riemann sphere with slits  (indicated on the very left of the figure) 
  another copy of the Riemann sphere with slits (indicated on the very right of the figure) and a copy of a circular arc triangle (it happens to be identical to $V_0$ and is indicated in the middle of the figure). As we already mentioned, the corresponding  points on the arc
  $a$ that is represented twice in the figure have to be identified. Similar identifications are necessary for the points on $c$, on $b$, and on $b'$.  
    
  The Riemann surface $R_0$ has a (simple)  branch  point lying  over $\overline {\rho}$ 
   where all triangles $Y_k$, $k=0, \dots, 5$, meet. The boundary of $R_0$ is  given by 
 the union of the  three arcs represented by thick black lines  in Figure~\ref{fig:covs6}.
  
  These considerations imply that over each point $w\in \wC\setminus V_0$, $w\ne \overline {\rho}$,
  lie two points of $R_0$. Actually, the same is also true for the vertices $0,1,\infty$ of $V_0$, but over each point in $V_0\setminus\{0,1,\infty\}$ lie three points of $R_0$.  This implies that each for each point $w\in \wC\setminus V_0$, $w\ne \overline {\rho}$, there are two distinct simple preimages  $z_1, z_2\in V_0$ under the map
  $s_6$. 
 Note that here necessarily $z_1, z_2\in\inte(V_0)$ as $s_6(\partial V_0)=\partial V_0$. The point $w= \overline {\rho}$ is a critical value of $s_6$. In
$ \inte(V_0)$ it has the preimage $z=\rho$, where the local degree of $s_6$ is equal to $2$. This shows that each point in $w\in \wC\setminus V_0$
has two preimages in $\inte(V_0)$ counting multiplicities. So the statement of the corollary is true if $V=V_0$.
 
 If $V$ is an arbitrary circular arc triangle from the tessellation $\mathcal{V}$, then there a exists unique  $\varphi\in \overline {\Ga}(2)\sub \overline{\Ga}$
such that $V=\varphi(V_0)$.  Then by \eqref{eq:poly46} we have 
$\varphi \circ s_6=s_6\circ \varphi$. It follows that each point $w\in \wC\setminus V=\varphi(\wC\setminus V_0)$ 
has precisely two preimages in $ \varphi(\inte(V_0))=\inte(V)$ counting multiplicities. 
\end{proof}

  \begin{proof}[Proof of Corollary~\ref{cor:crit6}]  We know that $\tau\in \H$ is a critical point of $E_6$ if and only of $s_6(\tau)=\infty$ (see Lemma~\ref{lem:critpol46}). Now the considerations 
  in the proof of Corollary~\ref{cor:s6} show that for $\tau\in V_0$ we have 
  $s_6(\tau)=\infty$ exactly at two points $\tau\in V_0$. Namely, we have $s_6(\infty)=\infty$. Moreover,  Figure~\ref{fig:covs6} shows that if the point $\tau$ moves along the common boundary arc $\sigma$ of $T_2$ and $T_3$ from $\rho$ to $\tfrac{1}{2}(1+i)$, then $s_6(\tau)$ moves (monotonically) along the line 
  $L_{1/2}$ from $s_6(\rho)=\overline{\rho}$ to $-i\infty$ through $\infty$ and then from $+i\infty$ to  $s_6(\tfrac{1}{2}(1+i))=\tfrac{1}{2}(1+i)$. Hence there is a unique  point $\tau_0\in \inte(\sigma)$ with $s_6(\tau_0)=\infty$. 
  It follows $E_6$ has  precisely one  critical point in $V_0\cap \H$ and this point lies in $\inte(V_0)$. 
  
  If $V\in \mathcal{V}$ and $\infty\in V$, then $V=V_0+n$ for some $n\in \Z$. 
  Since we have the identity $E_6(\tau+n)=E_6(\tau)$ for $\tau\in \H$, it follows that $E_6$ has one critical point in $V$ and again this point lies in $\inte(V)$.
  
  Now let $V\in \mathcal{V}$ with $\infty\not \in V$ be arbitrary. Then by 
   Corollary~\ref{cor:s6} the map $s_6$ attains the value $\infty$ twice in $V$ counting multiplicities and these preimages of $\infty$ lie in $\inte(V)$.  Since the poles of $s_6$ are all simple (see Proposition~\ref{prop:simpol}), we conclude that $s_6$ attains the value $\infty$ on $V$ exactly at two distinct points $z_1, z_2\in \inte(V)$. 
  It follows that $E_6$ has exactly two critical points in $V$, namely   
  $z_1, z_2\in \inte(V)$. Since the critical points of $E_6$ are simple by 
   Proposition~\ref{prop:Esimcrit}, the statement follows. 
  \end{proof}

 \section{Mapping behavior of  $s_2$}\label{sec:maps2} 
 In Section~\ref{s:equi} we have seen that if we choose the root  $E_4^{1/2}$ so that it takes positive real values on the positive imaginary axis, then 
 $s^+_2$ and $s^-_2$ (see \eqref{eq:s2def}) are well-defined   continuous functions on $W$ (see \eqref{eq:Omdef})  that are meromorphic in $\inte(W)$.
 Actually, the $q$-expansion \eqref{E4}  implies that 
 \begin{equation}\label{eq:E4root}
E_4(\tau)^{1/2}=1 +120q +O(q^2)
 \end{equation}
 for $\tau\in W$ with  $\im(\tau)$  large. If we combine this with   \eqref{E2}, we see that  that  the functions $s^\pm_2$  can be represented in the form  
\begin{align}\label{eq:s2+ser}
s_2^+(\tau) &=\tau-\frac{3i}{\pi}+ i\sum_{n=1}^\infty a_nq^n=\tau-\frac{3i}{\pi} +O(q), \\
s_2^-(\tau)&=\tau+\frac{i}{24\pi}q^{-1}+ i\sum_{n=0}^\infty b_nq^n
=\tau+\frac{i}{24\pi}q^{-1}+O(1),
 \label{eq:s2-ser}
 \end{align}  
 where $a_n\in \R$ for $n\in \N$, $b_n\in \R$ for $n\in \N_0$  and the $q$-series converge  if  $\im(\tau)$ is sufficiently large.

 It follows from  \eqref{E4zeros} and \eqref{eq:E2spec} that 
 \begin{equation}\label{eq:s2+}
 s_2^+(\rho)=s_2^-(\rho) =\overline \rho.
\end{equation}

%
%

\begin{lem} \label{lem:s2onim} The maps $s_2^+$ and $s_2^-$ on $W$  can be continuously extended to  $0$ and $\infty$ by setting $s_2^+(0)=0=s_2^-(0)$ and $s_2^+(\infty)=\infty=s_2^-(\infty)$.
Moreover, for small $t>0$ we have 
$s_2^+(it)=ih(t)$, where $0<h(t)<t$.  
\end{lem}

\begin{proof} 
It follows from \eqref{eq:s2+ser} and   \eqref{eq:s2-ser} that $s_2^+(\tau)\to \infty$ and $s_2^-(\tau)\to \infty$ as $\tau \in W$ and $\tau\to \infty$, because then necessarily $\im(\tau)\to +\infty$ and $q\to 0$.

For a point $\tau\in W$ let $\tau'\coloneq -1/\tau\in W$. Then 
as $\tau\to 0$, we have $\tau'\to \infty$ and so $s_2^+(\tau')\to \infty$ and  $s_2^-(\tau')\to \infty$ by what we have just seen. Then  \eqref{eq:s2trans} implies that 
 $s_2^+(\tau)=-1/s_2^-(\tau')\to 0$ and  $s_2^-(\tau)=-1/s_2^+(\tau')\to 0$. We conclude that on $W$ the maps  $s^+_2$   and $s^+_2$  have indeed a continuous extension to $0$ and $\infty$ as stated.

Now suppose that  $\tau=it$ with $t>0$. Then $\tau'=-1/\tau=it'$, where  
$t'\coloneq 1/t$,  and  \eqref{eq:s2-ser}  shows that if $t$ is sufficiently small, then 
$s_2^-(\tau')=i\widetilde h(t')$ with $\widetilde h(t')>t'$. It follows that 
if we set $h(t)\coloneq 1/\widetilde h(t')$ for small $t>0$, then 
$0<h(t)<1/t'=t$ and 
\[ s_2^+(it)=-1/s_2^-(it')=i/\widetilde h(t')=i h(t).
\]
The statement follows. 
\end{proof}

Recall the definition of $U_0$ given in \eqref{eq:U0}. Since $U_0\sub W\cup\{0,\infty\}$, 
by the previous lemma we can consider $s_2^+$ and $s_2^-$ as continuous maps on $U_0$.  

\begin{lem}\label{lem:s2injinfty}
The map $s_2^+$ is injective on $U_0$ near $\infty$. 
\end{lem}
\begin{proof} This follows from  \eqref{eq:s2+ser}
and  Lemma~\ref{lem:injinfty}~\ref{item:inj1}.  
\end{proof}

 \begin{lem}\label{lem:s2fixed} The maps   $s^+_2$ and $s^-_2$ have no 
 fixed point in $W$.
 \end{lem}  
 \begin{proof} Note that $s_2^\pm(\tau)-\tau=1/(E_2\pm E_4^{1/2})$ cannot vanish
 anywhere on $W$ and so the statement follows.
 \end{proof}

 We denote by 
\begin{align*}
A&\coloneq \{it: t\ge 0\}\cup\{\infty\},\\
B&\coloneq \{1+ e^{it}: 2\pi/3\le t\le  \pi \},\\
C&\coloneq \{\tfrac12+it: t\ge \sqrt3/2 \} \cup\{\infty\}
\end{align*}
the three sides of $U_0$ and by  
\begin{align*}
A'&\coloneq \{it: t\ge 0\}\cup\{\infty\},\\
B'&\coloneq \{1+e^{it}: \pi \le t\le  4\pi/3 \},\\
C'&\coloneq \{\tfrac12+it: t\ge -\sqrt3/2 \} \cup\{\infty\}
\end{align*}
the three sides of $Z_0$ as described in Theorem~\ref{thm:s2}.

 \begin{lem}\label{lem:s2homeo} The map $s_2^+$ is a homeomorphism of $\partial U_0$ onto $\partial Z_0$. Moreover, if we orient   $\partial U_0$ 
so that $U_0$ lies on the left, then  $Z_0$ lies on the left of $\partial Z_0$ if we equip  $\partial Z_0$ with the orientation induced  by the map $s^+_2$.  
\end{lem}

\begin{proof} No interior point of any of the sides of $U_0$ is a fixed point of $s_2^+$ or  $s_2^+$ (Lemma~\ref{lem:s2fixed}) and it follows 
from Corollary~\ref{cor:locinj}~\ref{i:loc2} that  both $s_2^+$ and $s_2^-$ are locally injective near each of these points. So condition~\ref{i:cir2} and  condition~\ref{i:cir3} in Lemma~\ref{lem:circarc} are true for $s_2^+$ and 
$s_2^-$ and each side of $U_0$.

We know that  $E_2(\tau), E_4(\tau)\in \R$  when $\tau\in \inte(A)\cup \inte(C)$. Actually, the values of $E_4(\tau)$ are positive when 
$\tau\in \inte(A)$. They are also positive for  $\tau\in \inte(C)$, because they are real, positive for $\tau\in \inte(C)$ near $\infty$ as follows from \eqref{E4}, and 
there is no zero of $E_4$ in $\inte(C)$. Therefore, \eqref{eq:E4root} implies that 
 $E_4^{1/2}(\tau)>0$ and in particular $E_4^{1/2}(\tau) \in \R$ for $\tau\in \inte(A)\cup \inte(C)$. We conclude that 
 $s^+_2(A)\sub  L_{0}$,   
 $s_2^+(C)\sub L_{1/2}$, and  $s_2^-(C)\sub   L_{1/2}$. 
 
Now by  Lemma~\ref{lem:s2onim} we have  $s_2^+(0)=0$, $s_2^+(\infty)=\infty$, and for points 
$\tau$ on the positive imaginary axis close to $0$, the image 
$s_2^+(\tau)$ lies in the interior  of the subarc of $A$ with endpoints $0$ and $\tau$. So by Lemma~\ref{lem:circarc}~\ref{i:cirii} (with $L_{0}$ taking the role of $K$) 
$s_2^+$ maps $A$ homeomorphically onto $A'=A$. Moreover,
it follows that there exists an increasing function  $h\: [0,\infty)\ra [0,\infty)$
such that 
\begin{equation}\label{eq:s2+h}
 0<h(t)<t \text{ and }  s_2^+(it)=ih(t)
\end{equation}
   for all $t>0$. 
We also see that the point $s_2^+(\tau)$  moves monotonically
from $\infty$ to $0$ along $A'=A$ if $\tau$ moves from $\infty$ 
to $0$ along $A$. So if we equip $A'$ with this orientation, then  
$Z_0$ lies on the left of $A'$. 
     
By \eqref{eq:s2+} we have  $s_2^+(\rho)= s_2^-(\rho)=\overline \rho\not\in C$.  Now Lemma~\ref{lem:circarc}~\ref{i:ciri} (with $L_{1/2}$ taking the role of $K$) shows that both maps  $s_2^+$ and $s_2^-$ send $C$ homeomorphically onto their images in  $ L_{1/2}$. These are necessarily arcs in 
$ L_{1/2}$ with endpoints $\overline \rho$ and $\infty$.
Now the expansions \eqref{eq:s2+ser} and \eqref{eq:s2-ser}  show that $\im(s_2^+(\tfrac12+it))$ takes large positive values and $\im(s_2^-(\tfrac12+it))$ takes large negative values when $t>0$ is large. It follows that 
$s_2^+$ sends $C$ homeomorphically onto $C'$, while 
$s_2^-$ sends $C$ homeomorphically onto the other subarc 
\[C''\coloneq \{\tfrac12+it: t\le -\sqrt3/2 \} \cup\{\infty\}
\]
of $ L_{1/2}$ with endpoints $\overline \rho$ and $\infty$.

The former statement implies  that $s_2^+(\tau)$  moves monotonically
from $\overline \rho$ to $\infty$ along $C'$ as $\tau$ moves from $\rho$ 
to $\infty$ along $C$. So if we equip $C'$ with this orientation, then  
$Z_0$ lies on the left of $C'$. We can also conclude that $s_2^-(\tau)$  moves monotonically
from $\overline \rho$ to $\infty$ along $C''$ as $\tau$ moves from $\rho$ 
to $\infty$ along $C$.

 Now let $\beta$ denote the reflection in the unit circle as in \eqref{eq:abcdef}.
Then $\beta(B)=C$ and $\beta(C'')=B'$. Moreover, $s_2^+=\beta \circ s_2^-\circ \beta$ on $B$ as follows from Lemma~\ref{lem:transs2}. Since $s_2^-$ sends $C$ homeomorphically onto $C''$, 
the map 
$s_2^+=\beta\circ s_2^-\circ \beta $ sends $B$ homeomorphically onto $B'$. 
Keeping track of orientations, we see that  
$s_2^+(\tau)$ moves from $0$ to $\overline \rho$ along $B'$ as  
as $\tau$ moves from $0$ 
to $\rho$ along $B$. If we equip $B'$ with this orientation, then $Z_0$ lies on the left of $B'$.  
 
 We conclude that $s_2^+$ sends the three sides $A$, $B$, $C$ of $U_0$ homeomorphically onto the three sides $A'$, $B'$, $C'$ of $Z_0$, respectively. 
It follows that $s_2^+$ is a homeomorphism of $\partial T_0$ onto $\partial Z_0$. Moreover, our  considerations also show
 that  the statement about orientations is true as stated. 
 \end{proof}

In order to show that $s_2^+$ has no poles in $\inte(U_0)$, 
 we consider the auxiliary function
\begin{equation}\label{eq:gdef2}
g(\tau)\coloneq \frac{1}{s_2^+(\tau)-\tau}=\frac{i\pi}{6} (E_2+E_4^{1/2}).
\end{equation}
This is a holomorphic  function on $\inte(U_0)$ with a continuous extension to (at least) $U_0\setminus\{0,\infty\}$.  

\begin{lem} \label{lemgs2sides} We have $\im(g(\tau))>0$ for $\tau \in \partial U_0\setminus\{0,\infty\}$. 
\end{lem}

\begin{proof}  If $\tau\in \inte(A)$, then 
$\tau=it$ with $t>0$ and  by \eqref{eq:s2+h}
we have 
$s_2^+(it)=ih(t)$ with $0<h(t)<t$. It follows that 
\[
\im (g(\tau))= \im \biggl( \frac{1}{s_2^+(it)-it}\biggr) =\frac1{t-h(t)}>0.
\]

 If $\tau\in \inte(B)$, then $s_2^+(\tau)\in \inte(B')$. In particular, $\im(\tau)>0$ and $\im(s_2^+(\tau))<0$. So $\im (\tau-s_2^+(\tau))>0$ and 
 \[
\im (g(\tau))= \im \biggl(\frac{1}{s_2^+(\tau)-\tau}\biggr)>0,
\]
 because $\im(w)>0$ for $w\in \C$ implies that $\im(-1/w)>0$. 
 
 Finally, we have seen in the proof of Lemma~\ref{lem:s2homeo}   that $s_2^+(\tau)$  moves monotonically
from $\overline \rho$ to $\infty$ along $C'$ as $\tau$ moves from $\rho$ 
to $\infty$ along $C$. Since  $s_2^+$ has no fixed points in 
$C\setminus\{\infty\}$ 
(see Lemma~\ref{lem:s2fixed}), it follows that there exists a function $k\: [\sqrt3/2, \infty)\ra [-\sqrt3/2, \infty)$ with $k(t)<t$ and 
\[
s_2^+(\tfrac 12+it)=\tfrac 12 +ik(t)
\] for  $t\in  [\sqrt3/2, \infty)$. 
Now if $\tau\in C\setminus\{\infty\}$, then $\tau=\tfrac12+it$ with 
 $t\in  [\sqrt3/2, \infty)$ and so 
\[
\im (g(\tau))=\im \biggl( \frac{1}{s_2^+(\tau)-\tau}\biggr)=\frac{1}{t-k(t)}>0.
\]
The statement follows.
 \end{proof}

\begin{lem} \label{lem:gs2vertex} We have 
\[
\liminf_{\inte(U_0)\ni \tau \to \tau_0} \im(g(\tau))\ge 0
\]
for $\tau_0\in \{0,  \infty\}$
\end{lem}

\begin{proof} The expansion \eqref{eq:s2+ser} shows that 
\[ s_2^+(\tau)-\tau=-\tfrac{3i}{\pi}+o(1) \text{ and so } 
g(\tau)= \frac{1}{s_2^+(\tau)-\tau}=\tfrac{i\pi}{3}+o(1)
\]
as $U_0\ni \tau \to \infty$. 
Hence 
\[
\liminf_{\inte(U_0)\ni \tau \to \infty} \im(g(\tau))=\pi/3\ge 0. 
\]

Now suppose 
$\tau\in \inte(U_0)$ is close to $0$. Then 
$\tau'\coloneq 1/\overline \tau\in  \inte(U_0)$ is close to $\infty$. In particular, $\im(\tau')$ is large and $\re(\tau')$ is uniformly bounded. Moreover,   \eqref{eq:s2-ser} shows that 
\[
s_2^-(\tau')=\tau'+\frac{i}{24\pi}\widehat q^{-1}+O(1),
\]
where $\widehat q=\exp(2\pi i \tau')$. 
Since $|\widehat q^{-1}|=\exp(2\pi \im(\tau'))$ and $\im(\tau')\to +\infty$ as $ \inte(U_0)\ni \tau\to 0$, we have 
\[
\tau'/s_2^-(\tau')\to 0
\]
as $ \inte(U_0)\ni \tau\to 0$. 

Now by \eqref{eq:sbet} we have 
\[
s_2^+(\tau)=1/\overline{s_2^-(\tau')}
\]
and so 
\begin{align*}
\im(g(\tau))&=\im\biggr(\frac{1}{s_2^+(\tau)-\tau}\biggl)=
\im\biggr(\frac{1}{1/\tau'-1/s_2^-(\tau')}\biggl)\\
&=\im\biggr( \frac{\tau'}{1-\tau'/s_2^-(\tau')}\biggl)=\im(\tau'(1+o(1)))\to +\infty
\end{align*}
as $ \inte(U_0)\ni \tau\to 0$ (for the last conclusion it is important that $\re(\tau')$ is uniformly bounded). 
It follows that 
\[
\liminf_{\inte(T_0)\ni \tau \to 0} \im(g(\tau))=+\infty\ge 0.
\]
The proof is complete.
\end{proof}

\begin{lem}  \label{lem:gs2pos}
We have $\im(g(\tau))>0$ for $\tau\in \inte(U_0)$. 
\end{lem}

\begin{proof} By Lemma~\ref{lem:minprincip} it is enough to show that
\[
\liminf_{\inte(U_0)\ni \tau \to \tau_0} \im(g(\tau))\ge 0
\]
for all $\tau_0\in \partial U_0$. This follows from Lemmas \ref{lemgs2sides} and~\ref{lem:gs2vertex}. 
\end{proof}

\begin{lem} \label{lem:s2nopoles} The function $s_2^+$ has no poles in $\inte(U_0)$. 
\end{lem}
\begin{proof} By Lemma~\ref{lem:gs2pos} the function  $g$ has no zeros in $\inte(U_0)$. This implies that $s_2^+$ cannot have poles in $\inte(U_0)$. \end{proof}

 \begin{proof}[Proof of Theorem~\ref{thm:s2}]
 We can apply Proposition~\ref{prop:arg} to $f=s_2^+$, $U=U_0$, and $V=Z_0$  by Lemmas~\ref{lem:s2homeo},  \ref{lem:s2injinfty} and \ref{lem:s2nopoles}. The statement follows.
  \end{proof}

    \begin{figure}
 \begin{overpic}[ scale=0.8
    ]{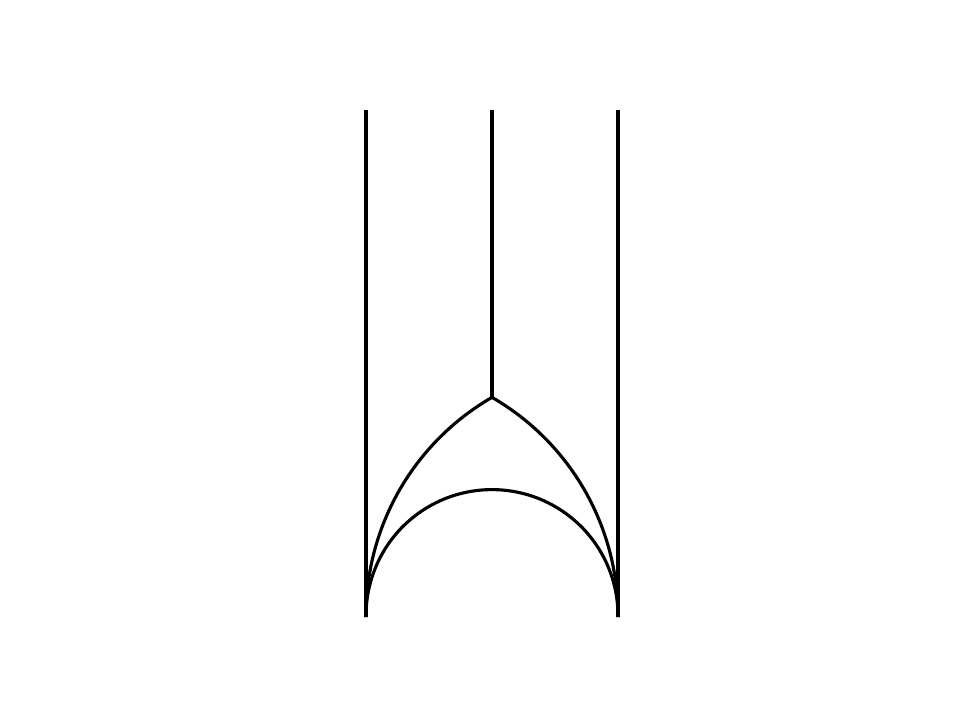}
            \put(42,48){ $U_0$}
           \put(55,48){ $U_2$}
             \put(49,27){ $U_1$} 
             \put(38.5,10.5){ $0$}
             \put(65,10.5){ $1$} 
             \put(52,35){ $\rho$}
              \put(68,55){ $V_0$}           
          \end{overpic}
  \vspace{-1.5cm}
\caption{Decomposition  of $V_0$ into $U_0$, $U_1$, $U_2$.}
\label{fig:U02}
\end{figure}
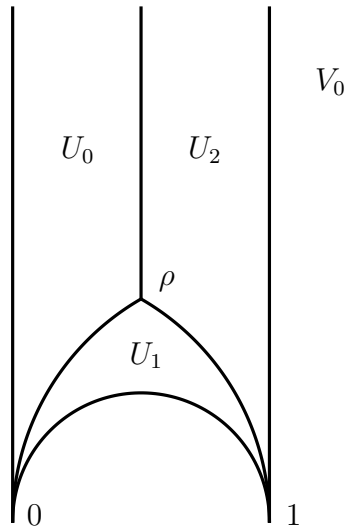

\subsection*{Global mapping behavior of $s_2$} We now want to use Theorem~\ref{thm:s2} to obtain a better understanding of the global mapping behavior of the multi-valued function $s_2$. For this we consider the circular arc triangles
obtained from $U_0$ under the successive reflections $\delta$, $\beta$, $\ga$ (see  \eqref{eq:delta}  and \eqref{eq:abcdef}). If we define  $U_1\coloneq \delta(U_0)$
and $U_2\coloneq \beta(U_1)$, then $U_0=\ga (U_2)$ and  the circular arc triangles $U_0$, $U_1$, $U_2$ tile $V_0\in \mathcal{V}$ as represented in Figure~\ref{fig:U02}. Since the branch $s_2^+$ of $s_2$ restricted to  $U_0$ maps each side of $U_0$ onto a subarc of the same circle that contains the given side, the branch $s_2^+$ undergoes the same transformation as the underlying triangle. So, for example, if we continue $s^+_2$ analytically across 
the side 
\[B= \{1+ e^{it}: 2\pi/3\le t\le  \pi \}\]  of $U_0$ to $U_1=\delta(U_0)$, then 
we obtain a conformal map of $U_1$ onto the circular arc triangle $Z_1\coloneq 
\delta(Z_0)$.  Further analytic continuation across the common side 
$ \{e^{it}: 0\le t\le \pi/3 \}$ of $U_1$ and $U_2=\beta(U_1)$ leads to a conformal map of $U_2$ onto the circular arc triangle $Z_2\coloneq \beta(Z_1)$.
Now reflection of $U_2$ in the common side of $U_2$ and $U_0$ given by $
\ga$  leads back to $U_0$. But through this sequence 
$\delta$, $\beta$, $\ga$  of reflections, we completed a full circuit around 
$\rho$. This  is a zero of $E_4$ and hence a branch point of $s_2$. In particular, by this operation we arrive at the other branch $s_2^-$ of $s_2$ on 
$U_0$. We conclude that  $s_2^-$ is a conformal map of $U_0$ onto the circular 
arc triangle $Z_0'\coloneq \ga(Z_2)$. Further analytic continuation 
of $s_2^-$ to $U_1$ and $U_2$ (as described  for $s_2^+$ above) gives conformal maps 
onto circular arc triangles $Z_1'\coloneq \delta(Z_0')$ and $Z_2'\coloneq \beta(Z_1')$. The six  triangles $Z_0, \dots,  Z_2'$ can be glued together in a natural way to give a Riemann surface $R_1$ with boundary that is spread over $\wC$. This is indicated 
in Figure~\ref{fig:covs2}. 

Here the picture in the middle represents the Riemann sphere with two slits. The pictures on the left and right are copies of $V_0$ and 
one has to identify points on   the dotted arcs indicated by $a$ and $c$ with the corresponding points in the middle picture.  Some of the $Z$-triangles consist of two pieces in the figure that have to be glued together as indicated.  
The boundary of $R_1$ consists of the six arcs represenred by thick black lines in Figure~\ref{fig:covs2}.

  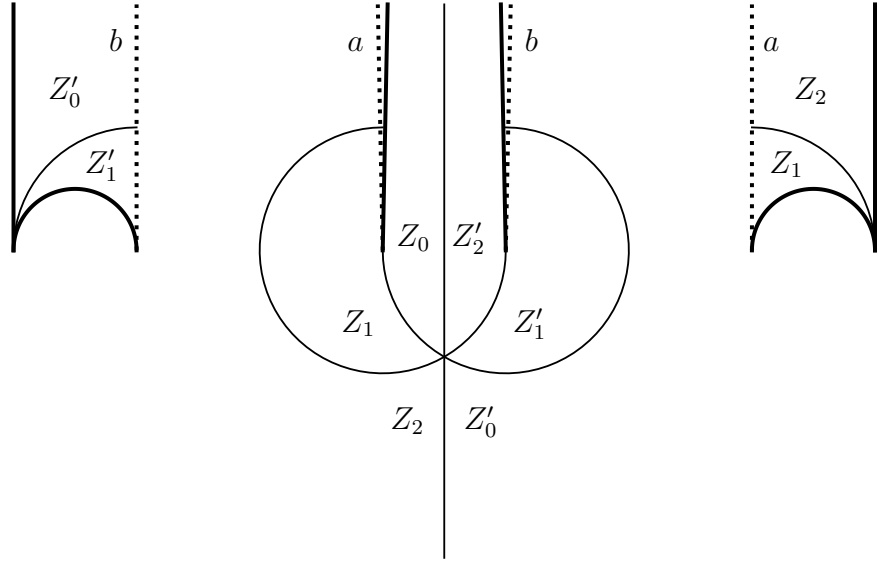
\begin{figure}
 \begin{overpic}[ scale=1.0
    ]{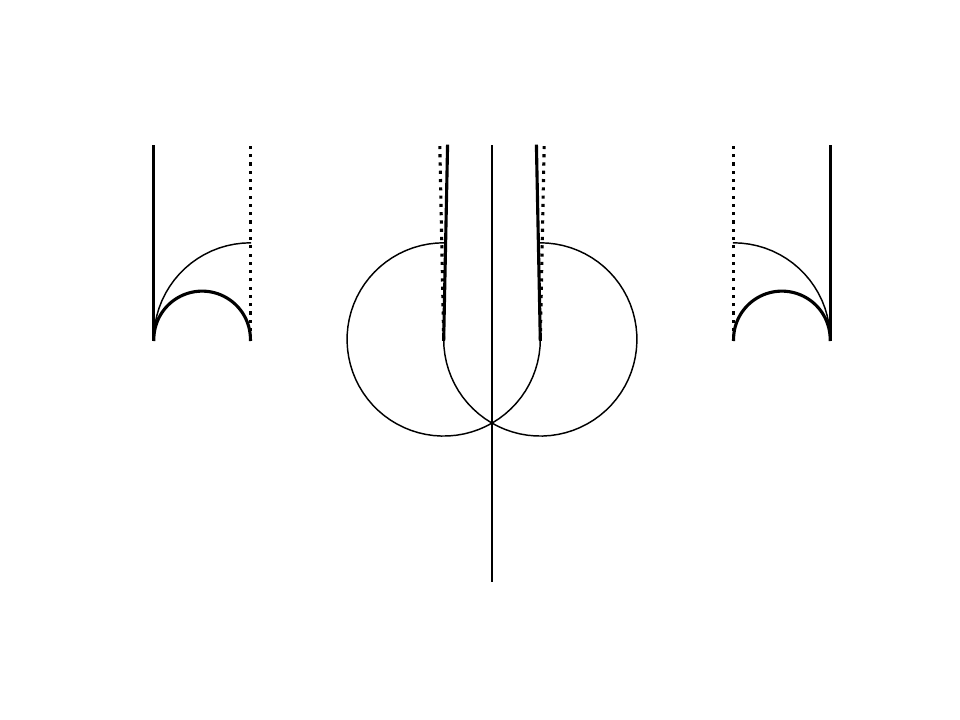}
\put(18,52){ $Z'_0$}
\put(79,52){ $Z_2$}
         \put(21,46){ $Z'_1$}
           \put(77,46){ $Z_1$}

            \put(46.5,40){ $Z_0$}
             \put(51,40){ $Z_2'$}
             \put(42,33){ $Z_1$}
             \put(56,33){ $Z'_1$}
               \put(46,25){ $Z_2$}
               \put(52,25){ $Z'_0$}
          
           \put(23,56){ $b$}
           \put(42.5,56){ $a$}
            \put(57,56){ $b$}
            \put(76.5,56){ $a$}            
          \end{overpic}
  \vspace{-3cm}
\caption{Images of the subtriangles $U_0, U_1, U_2$ of $V_0$ under $s_2$.}
\label{fig:covs2}
\end{figure}

We record some immediate consequences of these considerations.
\begin{cor}\label{cor:s2-} 
The map $s^-_2$ is a conformal map of the circular arc triangle $U_0$ onto the circular arc triangle $Z_0'$ bounded by 
the vertical line  from $+i\infty$ to $0$, the  arc
on the circle $\{\tau\in \C:|\tau-1|=1\}$ from $0$ to $\overline \rho$ in 
anti-clockwise orientation, and the vertical 
line from to $\overline \rho$ to $-i\infty$. 
Here $Z'_0$ lies on the left of its boundary with the orientation indicated, and the correspondence of vertices is 
$$s^-_2(\infty)=\infty,    \quad s^-_2(0)=0, \quad  s^-_2(\rho)=\overline\rho .$$    
\end{cor}
Another way to arrive at this statement is to use Theorem~\ref{thm:s2} in combination with the relation  \eqref{eq:s2trans}.

\begin{cor}\label{cor:s2valatt} Let $V\in \mathcal{V}$. For each $w\in \wC\setminus V$ there exists 
a unique point $\tau\in V$ such that  $w\in s_2(\tau)$. Moreover, $\tau\in \inte(V)$.  
\end{cor}
\begin{proof} First assume that $V=V_0$. Then the previous discussion and Figure~\ref{fig:covs2} show the Riemann surface obtained by gluing together the triangles $Z_0, \dots, Z_2'$ has a unique point over each point $w\in \wC\setminus V_0$. In other words, there exists a unique point $\tau\in V_0$ such that $s_2^+(\tau)=w$ or $s_2^-(\tau)=w$, or equivalently, such that $w\in s_2(\tau)$.

It follows from the  previous discussion that $s_2^+$ and $s_2^-$ map the side of the triangle $U_k$, $k=0,1,2$,  that does not contain $\rho$ homeomorpically onto itself. Since the union of these sides is equal to $\partial V_0$,  we conclude that 
 \[ s_2(\partial V_0)\coloneq \bigcup_{\tau\in \partial V_0}s_2(\tau)=
 \bigcup_{\tau\in \partial V_0}\{s_2^+(\tau), s_2^-(\tau)\}=\partial V_0.\]
 Therefore, if $w\in \wC\setminus V_0$, $\tau\in V_0$, and  $w\in s_2(\tau)$,
 then $\tau\in V_0\setminus \partial V_0=\inte(V_0)$. This shows that the statement is true for if $V=V_0$.  
 
 If $V\in \mathcal{V}$ is arbitrary, then there exists a unique $\varphi\in \overline {\Ga}(2)$ such that $V=\varphi(V_0)$. The statement now immediately follows 
 form the statement for $V_0$ and the relation 
 $\varphi\circ s_2=s_2\circ \varphi$ (see \eqref{eq:equi2}).  
\end{proof}

\begin{proof}[Proof of Corollary~\ref{cor:crit2}] This easily follows from Proposition~\ref{prop:Esimcrit} and 
Corollary~\ref{cor:s2valatt} in combination with  Lemma~\ref{lem:critpol2}. 
As the argument is very similar to the proof of Corollary~\ref{cor:crit4}, we skip the details.   \end{proof}

\section{Differential equations related to $s_4$ and $s_6$}
\label{sec:diffeq46}

There is another approach to proving our main results  
(Theorems~\ref{thm:s4},~\ref{thm:s6}, and~\ref{thm:s2}). Namely, as discussed in the introduction, we may want to represent each of our functions  $s^{\pm}_2$, $s_4$, $s_6$ as a quotient of two  
auxiliary functions satisfying a hypergeometric differential equation or a closely related differential equation of Fuchsian type and then resort to classical results by Schwarz in order to argue that  $s^{\pm}_2$, $s_4$, $s_6$ are conformal maps between suitable circular arc triangles. This  gives an interesting additional perspective to our results.

In this section, we will implement this idea for the maps 
$s_4$ and $s_6$. The maps $s^{\pm}_2$ require  somewhat different considerations  that we will discuss in the next section.  Throughout  we use the differential operator \eqref{eq:defD}.

For our discussion we need two auxiliary functions. The first one is
the discriminant $\Delta$ defined in \eqref{eq:discrdef}. 
This  is a modular form of weight $12$. From Proposition~\ref{Ederi} it follows that 
\begin{align} \label{eq:DDel} D  \Delta&=\tfrac{1}{1728}(D(E_4^3)-D(E_6^2))=\tfrac{1}{1728}(3E_4^2(DE_4)-2E_6(DE_6))\\
&=\tfrac{1}{1728}E_2(E_4^3-E_6^2)=  E_2 \Delta. \notag
\end{align}

The expansion \eqref{Dprod} shows that $\Delta$ has no zeros on $\H$ and that it takes positive real values for $\tau$ on the positive imaginary axis. In the following, we will consider various roots $\Delta^{\al}$ of $\Delta$ for $\al\in \Q$.
Since $\Delta$ has no zeros on $\H$, these powers  are holomorphic functions on the whole upper half-plane $\H$. A priori, $\Delta^{\al}$ is only defined up to multiplication by a root of unity. We fix this ambiguity so that
$\Delta^{\al}$ attains positive real values on the positive imaginary axis. With this convention \eqref{eq:DDel} implies that
\begin{equation}\label{eq:DrootDel}
D(\Delta^\alpha)=\alpha E_2\Delta^\alpha
\end{equation}
for each $\alpha \in \Q$. 

We need another auxiliary function, namely  the $J$-{\em invariant} given as 
\begin{equation}\label{defJ}
J\coloneq E_4^3/(E_4^3-E_6^2)= \tfrac{1}{1728} E_4^3/\Delta. 
\end{equation}
This is a holomorphic function on $\H$ and its definition  implies that 
$$J(\tau)=\tfrac{1}{1728}q^{-1}+O(1) \text{ as } \im(\tau)\to +\infty. $$ 
Hence $J(\infty)=\infty$ (understood in a limiting sense). Moreover, 
 \eqref{E4zeros} and   \eqref{E6zeros}  imply that  
$$ J(\rho)=0  \text{ and } J(i)=1. $$

 It is well known that $J$   maps the circular arc  triangle $T_0$ as in  \eqref{T0} conformally to the closed lower half-plane, considered 
as a circular arc triangle with vertices $0$, $1$, $\infty$ (see 
\cite[Chapter 6; Theorem 5, p.~90]{Cha}). 

It follows from  \eqref{defJ} that we obtain a uniquely determined holomorphic third  root of $J$ by setting $$J^{1/3}\coloneq\tfrac 1{12}E_4/\Delta^{1/3}.$$
Similarly, by \eqref{defJ} we have $J-1=\tfrac{1}{1728}E_6^2/\Delta$, which 
allows us to define the holomorphic function 
$$(J-1)^{1/2}\coloneq\tfrac{1}{24 \sqrt 3}E_6/\Delta^{1/2}. $$
Note that these definitions fix the ambiguity of $J^{1/3}$  and $(J-1)^{1/2}$  in such a way that these functions take positive real values with  $\tau$ on the positive imaginary axis with $\im(\tau)$ large.

The previous formulas can be rewritten as follows: 
\begin{equation}\label{eq:E46J}
E_4=12J^{1/3} \Delta^{1/3}, \quad E_6=24\sqrt3 (J-1)^{1/2}
\Delta^{1/2}. 
\end{equation} 
We also have
\begin{align} DJ&= \tfrac{1}{1728} \biggl( \frac{3 E_4^2 D(E_4)}{ \Delta}- \frac{E_4^3D( \Delta)}{\Delta^2}\biggr) \label{eq:derJ}\\ &=
- \tfrac{1}{1728}E_4^2E_6/ \Delta=- 2\sqrt3 J^{2/3} (J-1)^{1/2} \Delta^{1/6} . \notag
 \end{align}

We now want to represent $s_4$ as a ratio $s_4=a_1/a_2$ of suitably chosen  meromorphic functions  functions $a_1$ and $a_2$ in the upper half-plane.  We hope 
to find a second order ordinary differential  that $a_1$ and $a_2$  satisfy as 
functions of $\tau$ or some other  auxiliary (local) variable. 
In view of the definition of $s_4$, the proportionalities 
\begin{align*}
a_1&\sim \tau (E_2E_4-E_6)-\tfrac{6i}{\pi}E_4,\\
a_2&\sim  E_2E_4-E_6 
\end{align*}
are plausible.  To implement this idea, it is crucial to introduce the correct normalizations. Accordingly, we define
\begin{equation} \label{eq:a2def}
a_2\coloneq \Delta^{-5/12} (E_2E_4-E_6).
 \end{equation} 
The motivation for the factor $\Delta^{-5/12}$   is that in order to obtain a second  order ODE for $a_2$, it should 
be a   (pseudo-)modular forms of weight $1$ (as  explained in
\cite[Section 5.4]{Zag}).  Note that $E_2E_4-E_6$ has weight $6$, $\Delta$ has weight $12$, and so $a_2$ as defined above has weight $1$.  

On a more formal level,  we will see that the definition of $a_2$ leads to the vanishing of an unpleasant $E_2^2$-term under differentiation that one would expect from the formulas \eqref{E2der}--\eqref{E6der}.

Indeed, by \eqref{eq:Ds4dem} and \eqref{eq:DrootDel} we have 
\begin{align*}Da_2&= \Delta^{-5/12} (D(E_2E_4-E_6)-\tfrac5{12} E_2(E_2E_4-E_6))\\
&= \tfrac{5}{12}  \Delta^{-5/12}(E_2^2E_4+E_4^2-2E_2E_6-E_2(E_2E_4-E_6))\\
&=-\tfrac{5}{12}   \Delta^{-5/12}(E_2E_6-E_4^2). 
\end{align*} 
This computation suggests  that we should ``match" $a_2$ with a normalized version of the function $E_2E_6-E_4^2$ that appears in the denominator of the expression defining $s_6$. As  the latter expression corresponds to weight $8$, the right normalizing factor here is $\Delta^{-7/12}$, and so 
we define 
\begin{equation}\label{eq:b2def}
 b_2\coloneq\Delta^{-7/12}(E_2E_6-E_4^2). 
 \end{equation} 
Then by \eqref{eq:Ds6dem} we have 
\begin{align*}Db_2&= \Delta^{-7/12} (D(E_2E_6-E_4^2)-\tfrac7{12} E_2(E_2E_6-E_4^2))\\
&= \tfrac{7}{12}  \Delta^{-7/12}(E_2^2E_6-2E_2E_4^2+E_4E_6-E_2(E_2E_6-E_4^2))\\
&=-\tfrac{7}{12}   \Delta^{-7/12}E_4(E_2E_4-E_6). 
\end{align*} 
In particular, our formulas show  
\begin{equation}\label{eq:absys}
 Da_2=-\tfrac{5}{12}  \Delta^{1/6} b_2,\quad 
Db_2=-\tfrac{7}{12} E_4 \Delta^{-1/6} a_2 .
\end{equation}

We will now convert this into $J$-derivatives. Note that  the functions $a_2$ and $b_2$  depend on $\tau$ and can hence 
 be considered locally as (multi-valued) functions of $J$. By the chain rule we have 
 \[Da_2=\frac{da_2}{dJ}\cdot DJ.\]
 Solving for $\displaystyle \frac{da_2}{dJ}$ and using \eqref{eq:derJ}
 we arrive at 
  \begin{align}\label{eq:Jdervsysab}\frac{da_2}{dJ}=\frac{Da_2}{DJ} 
  &=\tfrac{5}{24\sqrt3}J^{-2/3}(J-1)^{-1/2} b_2,\\
\frac{db_2}{dJ}=\frac{Db_2}{DJ}&= \tfrac{7}{24\sqrt3}  E_4\Delta^{-1/3}J^{-2/3}(J-1)^{-1/2} a_2 \notag \\
&=\tfrac{7}{2\sqrt3}J^{-1/3}(J-1)^{-1/2} a_2.   \notag 
\end{align}

We need corresponding functions $a_1$ and $b_1$ satisfying the same differential system. 
We set 
\begin{equation}\label{eq:a1def}
a_1\coloneq \tau a_2-\tfrac{6i}{\pi}E_4\Delta^{-5/12}.
\end{equation} 
This  choice of $a_1$ is dictated by the requirement 
\begin{equation}\label{eq:s4arat}
s_4=a_1/a_2.
\end{equation}

Then
 \begin{align*}Da_1&=\tfrac{1}{2\pi i} a_2+\tau Da_2 -\tfrac{6i}{\pi}D(E_4)\Delta^{-5/12}+\tfrac{5i}{2\pi}E_4E_2 \Delta^{-5/12}\\
 &=\tau Da_2+  \tfrac{i}{2\pi }  \Delta^{-5/12}   
 (-E_2E_4+E_6-4E_2E_4+4E_6+ 5E_2E_4)\\
 &=\tau Da_2+  \tfrac{5i}{2\pi }  \Delta^{-5/12}   E_6=
- \tfrac{5}{12} \tau b_2 \Delta^{1/6}+  \tfrac{5i}{2\pi }  \Delta^{-5/12}   E_6\\
 &=-\tfrac{5}{12}\Delta^{1/6}(\tau b_2-\tfrac{6i}{\pi}E_6\Delta^{-7/12})=-\tfrac{5}{12}\Delta^{1/6}b_1,
 \end{align*}
 where 
 \begin{equation}\label{eq:b1def}
 b_1\coloneq \tau b_2-\tfrac{6i}{\pi}E_6\Delta^{-7/12}. 
 \end{equation} 
Note that 
\begin{equation}\label{eq:s6arat} s_6=\tau-\frac{6iE_6}{\pi(E_2 E_6-E_4^2)}=b_1/b_2.
\end{equation}
Moreover,  
  \begin{align*}Db_1&=\tfrac{1}{2\pi i} b_2+\tau Db_2 -\tfrac{6i}{\pi}D(E_6)\Delta^{-7/12}+\tfrac{7i}{2\pi}E_6E_2 \Delta^{-7/12}\\
 &=\tau Db_2+  \tfrac{i}{2\pi }  \Delta^{-7/12}   
 (-E_2E_6+E_4^2 -6E_2E_6+6E_4^2+ 7E_2E_6)\\
 &=\tau Db_2 + \tfrac{7i}{2\pi }  \Delta^{-7/12}   E_4^2=
 -\tfrac{7}{12} \tau a_2  E_4\Delta^{-1/6}+  \tfrac{7i}{2\pi }  \Delta^{-7/12}   E_4^2\\
 &=-\tfrac{7}{12}\Delta^{-1/6}E_4(\tau a_2-\tfrac{6i}{\pi}E_4\Delta^{-5/12})=-\tfrac{7}{12}\Delta^{-1/6}E_4a_1.
 \end{align*}
In other words, the pair $a_1,b_1$ satisfies exactly the same system of equations \eqref{eq:absys} as the pair $a_2, b_2$.
It then follows that we  also get  the same system \eqref{eq:Jdervsysab} for $J$-derivatives. We arrive at the following conclusion. 

\begin{prop} \label{prop:aJ} The functions $a_1$ and $a_2$ defined in \eqref{eq:a1def} and \eqref{eq:a2def} considered as functions of $J$ form  
a fundamental system of the hypergeometric equation
\begin{equation}\label{eq:hypa}
w''(z)+\frac{7z-4}{6z(z-1)}w'(z)-\frac{35}{144z(z-1)}w(z)=0. 
\end{equation} 
 
\end{prop} 
\begin{proof} Equations \eqref{eq:Jdervsysab} are true for the pair $a_2, b_2$ as well as for $a_1,b_1$. For simplicity we use the notation $a=a_k$ and $b=b_k$ for $k=1,2$. Then we have
\begin{align*} 
\frac{d^2a}{dJ^2}& =\frac{d}{dJ}\biggl(\frac{5}{24\sqrt3}J^{-2/3}(J-1)^{-1/2}  b\biggr),\\
&=\biggl(-\frac{2}{3J}-\frac{1}{2(J-1)}\biggr) \frac{da}{dJ}
+\frac{5}{24\sqrt3}J^{-2/3}(J-1)^{-1/2}  \frac{db}{dJ}\\
&=-\frac{7J-4}{6J(J-1)} \frac{da}{dJ}+
\frac{35}{144J(J-1)}a. 
\end{align*} 
The functions  $a_1$ and $a_2$ are linearly independent, since 
$a_1/a_2=s_4$ is not a constant.
\end{proof} 

The general form of the hypergeometric differential equation  (see \cite [Part 7, Chapter 2]{Car} or \cite[Chapter 8]{Bie}) is 
\begin{equation}\label{hyperODE}
 w''(z)+\frac{z(1+\alpha+\beta)-\ga}{z(z-1)}w'(z)+\frac{\alpha\beta}{z(z-1)}w(z)=0. 
\end{equation}

So \eqref{eq:hypa} corresponds to the parameters  $\alpha=7/12$, $\beta=-5/12$, $\ga=2/3$,
or   $\alpha=-5/12$, $\beta=7/12$, $\ga=2/3$.

It follows (see  \cite[Part 7, Chapter 2]{Car}  or \cite[pp.~252--254]{Bie} for more explanation)   that $s_4=a_1/a_2$ sends the closed lower  $J$-half-plane to a circular arc triangle
with angles $\la\pi, \mu\pi, \nu\pi$ corresponding to the $J$-values $0,1, \infty$, respectively, where 
$$  \la=1-\ga=1/3, \quad \mu=\ga -\alpha-\beta=1/2,\quad
\nu=\alpha-\beta=\pm 1.
$$ 
Recall that the map $J$ send the triangle $T_0$ (as defined in \eqref{T0}) in  the $\tau$-plane conformally onto the closed lower $J$-half-plane so that $\rho\mapsto 0$, $i\mapsto
1$, and $\infty\mapsto \infty$.

This implies that  $s_4$ sends $T_0$ to a circular arc triangle with vertices 
$s_4(\rho)=\rho$, $s_4(i)=-i$, and $s_4(\infty)=\infty$
so that at $\rho$ we have angle $\pi/3$, at  
$-i$ we have angle $\pi/2$, and at $\infty$ we have angle $\pi$. It is not hard to see that $X_0$ as in Theorem~\ref{thm:s4} is the unique circular arc triangle with these properties. We skip the details  (see \cite[Lemma~3.6]{BZeta} for a related argument) as this would add nothing new beyond what we already know by Theorem~\ref{thm:s4}.

%
%
%
%
%
 
 A similar conclusion can be derived for the functions $b_1$ and $b_2$. 
 \begin{prop} \label{prop:bJ} The functions $b_1$ and $b_2$ defined in \eqref{eq:b1def} and \eqref{eq:b2def} considered as functions of $J$ form  
a fundamental system of the hypergeometric equation
$$w''(z)+\frac{5z-2}{6z(z-1)}w'(z)-\frac{35}{144z(z-1)}w(z)=0. 
$$ 
 \end{prop} 
\begin{proof} Again we use  \eqref{eq:Jdervsysab} and the notation $a=a_k$ and $b=b_k$ for $k=1,2$. Then we have
\begin{align*} 
\frac{d^2b}{dJ^2}& =\frac{d}{dJ}\biggl(\frac{7}{2\sqrt3}J^{-1/3}(J-1)^{-1/2} a\biggr)\\
&=\biggl(-\frac{1}{3J}-\frac{1}{2(J-1)}\biggr) \frac{db}{dJ}
+\frac{7}{2\sqrt3}J^{-1/3}(J-1)^{-1/2}  \frac{da}{dJ}\\
&=-\frac{5J-2}{6J(J-1)} \frac{db}{dJ}+
\frac{35}{144J(J-1)}b.
\end{align*} 
The functions  $b_1$ and $b_2$ are linearly independent, since 
$b_1/b_2=s_6$ is not a constant. \end{proof} 
This corresponds to a hypergeometric equation 
with $\alpha=5/12$, $\beta=-7/12$, $\ga=1/3$ or $\alpha=-7/12$, $\beta=5/12$, $\ga=1/3$.

It follows that $s_6=b_1/b_2$ sends the closed lower  $J$-half-plane to a circular arc triangle
with angles $\la\pi, \mu\pi, \nu\pi$ corresponding to the $J$-values $0,1, \infty$, respectively, where 
$$  \la=1-\ga=2/3, \quad \mu=\ga -\alpha-\beta=1/2,\quad
\nu=\alpha-\beta=\pm 1.
$$ 
This in turn implies that  $s_6$ sends $T_0$ to a circular arc triangle with vertices 
$s_6(\rho)=\overline{\rho}$, $s_6(i)=i$, $s_6(\infty)=\infty$ so that 
at $\overline{\rho}$ we have angle $2\pi/3$, at  
$i$ we have angle $\pi/2$, and at $\infty$  we have angle $\pi$. The triangle $Y_0$ as in Theorem~\ref{thm:s6} is the unique circular arc triangle with these properties. 

The {\em Schwarzian derivative} of a meromorphic function $f(z)$ depending on a complex variable $z$ is defined as 
\begin{align*}
\{f,z\}&\coloneq \frac{d^2}{dz^2}\log f'(z)-\frac 12 \biggl(\frac{d}{dz}\log f'(z)\biggr)^2\\
&=  \frac{2f'(z)f'''(z)-3f''(z)^2}{2f'(z)^2}.
\end{align*}
We have the following chain rule for the Schwarzian
derivative $$\{f\circ g, z\}=\{f,g\}\biggl(\frac{dg}{dz}\biggr)^2+\{g,z\}.$$ 

If the functions $w_1$  and $w_2 $ form the two fundamental solutions of the differential equation 
\[w''(z)+p(z)w'(z)+q(z)=0
\] with meromorphic coefficients $p$ and $q$, then the ratio  $f=w_1/w_2 $  has a Schwarzian derivative given by
\begin{equation}\label{eq:Schwrat}
\{f, z\}=-p'-\tfrac12p^2+2q.
\end{equation}
This implies  (see \cite[p.~253]{Bie}) that the ratio $f=w_1/w_2 $ of two fundamental solutions of the hypergeometric differential equation has the Schwarzian derivative 
\begin{equation}\label{eq:Schwtriag}
 \{f, z\} =\frac{1-\la^2}{2z^2} + \frac{1-\mu^2}{2(1-z)^2}+ 
\frac{1-\lambda^2-\mu^2+\nu^2}{2z(1-z)},
\end{equation}
 where
$$ \la=1-\ga,\quad \mu=\ga-\alpha-\beta,\quad \nu=\alpha-\beta.$$ 

For our function $s_4=a_1/a_2$ we have $\la=1/3$, $\mu=1/2$, $\nu=\pm 1$, and so 
\begin{equation}\label{eq:s4JSch}
\{s_4,J\}=\frac{4}{9J^2} + \frac{3}{8(1-J)^2}+ 
\frac{59}{72J(1-J)}. 
\end{equation}

It is well known that (see \cite[Section 5]{BZeta}, where this is derived) 
\begin{equation}\label{eq:tauJSch}
\{\tau,J\}= \frac{4}{9J^2} + \frac{3}{8(1-J)^2}+ 
\frac{23}{72J(1-J)}.
 \end{equation}
Using  these formulas, \eqref{eq:derJ}, and \eqref{eq:E46J}, we see that 
\begin{align}\label{eq:Schws4tau}
\{s_4,\tau\} &= \{s_4, J\} J'(\tau)^2+\{J,\tau\} = (\{s_4, J\} -\{\tau, J\})J'(\tau)^2 \\
&= 
\frac{J'^2}{2J(1-J)}=24\pi^2\Delta^{1/3}J^{1/3} =2\pi^2E_4.\notag
\end{align} 

Similarly, for our function $s_6=b_1/b_2$ we have $\la=2/3$, $\mu=1/2$, $\nu=\pm 1$, and so 
\begin{equation}\label{eq:s6JSch}
\{s_6,J\}=\frac{5}{18J^2} + \frac{3}{8(1-J)^2}+ 
\frac{47}{72J(1-J)}. 
\end{equation}
 It follows that 
\begin{align}\label{eq:Schws6tau}
\{s_6,\tau\} &= \{s_6, J\} J'(\tau)^2+\{J,\tau\} = (\{s_6, J\} -\{\tau, J\})J'(\tau)^2 \\
&= 
J'^2 \biggl(-\frac{1}{6J^2} + 
\frac{1}{3J(1-J)}\biggr)=\frac{J'^2(1-3J)}{6J^2(J-1)}\notag \\
&=8\pi^2\Delta^{1/3}J^{-2/3}(3J-1)=\tfrac{2}{3}\pi^2
\frac{E_6^2+2E_4^3}{E_4^2}. \notag
\end{align}

 \section{A differential equation related to  $s_2$} \label{sec:diffeq2}
In order to imitate the approach from the previous section for $s_2$,   
we consider the function  
\begin{equation}\label{eq:defc2}
c_2\coloneq \Delta^{-1/12} (E_2+E_4^{1/2}).
 \end{equation}
 Since our goal is to derive some differential equations, it suffices 
 consider $c_2$ and other functions appearing in this section    as locally defined meromorphic functions.   

In particular,  we think of $E_4^{1/2}$ as a holomorphic function of $\tau$ 
near a point $\tau_0 \in  \H\setminus \Ga(\rho)$ (where 
 $E_4(\tau_0)\ne 0$) obtained from a locally  consistent choice of the root. This also determines  all powers $E_4^{n/2}=(E_4^{1/2})^n$,
 $n\in \Z$,  near $\tau_0$ and leads to consistent local definitions of $c_2$ and similar functions where these powers appear.

By  \eqref{eq:Ds2dem} we have 
\begin{align*}
Dc_2&=\Delta^{-1/12} D(E_2+E_4^{1/2})-\tfrac{1}{12}\Delta^{-1/12}
E_2(E_2+E_4^{1/2})\\
&=\tfrac{1}{12}\Delta^{-1/12}(-E_4+E_2E_4^{1/2}-2E_6E_4^{-1/2}).
\end{align*}
 We  introduce the new function
 \begin{equation}
f_2\coloneq 12 E_4^{1/2} \Delta^{-1/3}Dc_2=  \Delta^{-5/12}
  (E_2E_4- E_4^{3/2} -2E_6).
  \end{equation}
 We have 
 \begin{align*}D(E_2E_4- E_4^{3/2} -2E_6)&=D(E_2)E_4+E_2D(E_4)
 -\tfrac{3}{2}E_4^{1/2}D(E_4)-2D(E_6)\\
 =\tfrac{1}{12}(E_2^2E_4-E_4^2)+\tfrac{1}{3}(E^2_2E_4&-E_2E_6)
 -\tfrac12(E_2E_4^{3/2}-E_4^{1/2}E_6)-E_2E_6+E_4^2\\
  =\tfrac{5}{12}E_2^2E_4 +\tfrac{11}{12}E_4^2&-\tfrac{4}{3}E_2E_6
 -\tfrac12E_2E_4^{3/2}+\tfrac12E_4^{1/2}E_6,
    \end{align*}
 and so 
 \begin{align*}Df_2&= \Delta^{-5/12}D(E_2E_4- E_4^{3/2} -2E_6)-
 \tfrac{5}{12} \Delta^{-5/12} E_2 (E_2E_4- E_4^{3/2} -2E_6)\\
 &=\Delta^{-5/12}(\tfrac{11}{12}E_4^2-\tfrac{1}{2}E_2E_6
 -\tfrac{1}{12}E_2E_4^{3/2}+\tfrac12E_4^{1/2}E_6)\\
&=-\tfrac12 E_4^{1/2}f_2+\Delta^{-1/3}(\tfrac5{12} E_4^{3/2}-\tfrac12 E_6) 
c_2.
 \end{align*}
Therefore, the functions $c_2$ and $f_2$ satisfy the system
\begin{align}\label{eq:cdsys}
 Dc_2&=\tfrac{1}{12}  \Delta^{1/3} E_4^{-1/2}   f_2, \\ \notag
 Df_2&=-\tfrac12 E_4^{1/2}f_2+\Delta^{-1/3}(\tfrac5{12} E_4^{3/2}-\tfrac12 E_6) 
c_2.
\end{align}

Let 
\begin{equation}\label{eq:defc1}
c_1\coloneq \tau c_2-\tfrac{6i}{\pi}\Delta^{-1/12}. 
\end{equation} 
Then 
\begin{equation}\label{eq:s2arat}
s^+_2=c_1/c_2
\end{equation}
and
\begin{align*}
Dc_1&=\tfrac{1}{2\pi i} c_2+ \tau Dc_2
+\tfrac{i}{2\pi}\Delta^{-1/12}E_2= \tau Dc_2-\tfrac{i}{2\pi} \Delta^{-1/12} E_4^{1/2}\\
&=\tfrac{1}{12} \Delta^{1/3}  E_4^{-1/2}  (\tau f_2  -\tfrac{6i}{\pi}\Delta^{-5/12} E_4)= \tfrac{1}{12} \Delta^{1/3}  E_4^{-1/2} f_1, 
\end{align*}
where 
\[ 
f_1\coloneq \tau f_2  -\tfrac{6i}{\pi}\Delta^{-5/12} E_4.
\]
Then 
\begin{align*}
Df_1&=\tfrac{1}{2\pi i} f_2+ \tau Df_2
+\tfrac{5i}{2\pi}\Delta^{-5/12}E_2E_4-
\tfrac{2i}{\pi}\Delta^{-5/12}(E_2E_4-E_6) \\
&= \tau Df_2+\tfrac{i}{2\pi}\Delta^{-5/12}(6E_6+E_4^{3/2}) \\
&=-\tfrac12 E_4^{1/2}f_1+\Delta^{-1/3}(\tfrac5{12} E_4^{3/2}-\tfrac12 E_6) 
c_1.
\end{align*}
This shows that the pair $c_1, f_1$ also satisfies the system \eqref{eq:cdsys}.

We now  introduce the (multi-valued) function 
\begin{equation}\label{eq:zdef}
u\coloneq E_6E_4^{-3/2} =
(J-1)^{1/2} J^{-1/2}.
\end{equation}
If we choose branches here so that $E_4^{1/2}$, and hence also 
$E_4^{-3/2}$, attains positive real values on the positive imaginary axis, then 
$u$ maps the triangle $U_0$ defined in \eqref{eq:U0} onto the 
lower half-plane. This motivates introducing $u$ as a new local variable
for the functions $c_k$ and $f_k$, $k=1,2$,   in the same way as $J$ 
was chosen as a convenient variable for the functions $a_k$ and $b_k$ in Section~\ref{sec:diffeq46}.    

We have $1-u^2=1728 \Delta E_4^{-3},$ and so
$$ E_4=12 (1-u^2)^{-1/3} \Delta^{1/3}, \quad 
E_6= 24\sqrt3 u(1-u^2)^{-1/2}  \Delta^{1/2}.$$
Moreover, 
\[ Du=D(E_6)E_4^{-3/2}-\tfrac32 E_6E_4^{-5/2}D(E_4) 
=-864  \Delta E_4^{-5/2}.\]

If follows from the chain rule and the previous identities  that for $k=1,2$ we have   
\[ \frac {dc_k}{du} =\frac{Dc_k}{Du}=
-\tfrac1{6\cdot 1728} \Delta^{-2/3} E_4^2 f_k=-\tfrac1{72}(1-u^2)^{-2/3} f_k,
\] 
\begin{align*} \frac {df_k}{du} &=\frac{Df_k}{Du}=
-\tfrac{1}{1728}\Delta^{-4/3}(\tfrac5{6} E_4^4- E_4^{5/2}E_6) 
c_k
 +\tfrac1{1728} \Delta^{-1} E_4^{3}f_k\\
& = (-10(1-u^2)^{-4/3}+12u (1-u^2)^{-4/3})c_k +(1-u^2)^{-1}f_k\\
  =&(1-u^2)^{-4/3}(12u-10)c_k+(1-u^2)^{-1}  f_k.  
\end{align*}
We now arrive at an analog of Propositions~\ref{prop:aJ} and~\ref{prop:bJ}. 
\begin{prop}\label{prop:cz} The functions $c_1$ and $c_2$ defined in \eqref{eq:defc1} and \eqref{eq:defc2}  and considered as functions of $u$ in \eqref{eq:zdef}
 form  
a fundamental system of the Fuchsian  equation
\begin{equation}\label{eq:Fuchs}
w''(z)-\frac{4z+3}{3(1-z^2)}w'(z)+\frac{6z-5}{36(1-z^2)^2}w(z)=0. 
\end{equation}
\end{prop}

\begin{proof}If we write $c=c_k$ and $f=f_k$ for $k=1,2$, then we have 
\begin{align*}
\frac{d^2c}{du^2}& =\frac{d}{du}\biggl(-\frac{1}{72}(1-u^2)^{-2/3}  f\biggr),\\
&=\frac{4u}{3(1-u^2)} \frac{dc}{du}
-\frac{1}{72} (1-u^2)^{-2/3}  \frac{df}{du}\\
&=\frac{4u}{3(1-u^2)} \frac{dc}{du}-\frac{1}{72} (1-u^2)^{-2/3} 
((1-u^2)^{-4/3}(12u-10)c+(1-u^2)^{-1}f)\\
&=\frac{4u}{3(1-u^2)} \frac{dc}{du}-\frac{1}{72} (1-u^2)^{-5/3}f
- \frac{6u-5}{36(1-u^2)^2} c\\
& =\frac{4u+3}{3(1-u^2)}  \frac{dc}{du}-\frac{6u-5}{36(1-u^2)^2} c.
\end{align*} 
The functions  $c_1$ and $c_2$ are linearly independent, since 
$c_1/c_2=s_2^+$ is not a constant.
\end{proof}

Equation~\eqref{eq:Fuchs} is a  Fuchsian differential equation with  three singular points, namely  $-1$, $1$, and $\infty$. Accordingly, it  can be transformed into the hypergeometric equation, but we will not pursue this here.  
We want to use  \eqref{eq:Fuchs} to obtain expression for the Schwarzian 
of $s_2^+$ with respect to various variables.  

It follows from \eqref{eq:Schwrat} that 
\begin{align*}
\{s_2^+, u\}&=- \frac{d}{du}\biggl(-\frac{4u+3}{3(1-u^2)}\biggr)-\frac12\biggl(-\frac{4u+3}{3(1-u^2)}\biggr)^2+ 2 \cdot\frac{6u-5}{36(1-u^2)^2}\\
&=\frac{4u^2+9u+5}{9(1-u^2)^2}.
\end{align*}
Let $v\coloneq \tfrac12(u+1)$. Then we have
\begin{align*}
\{s_2^+, v\}&=\{s^+_2, u\}\cdot \biggl(\frac {du}{dv}\biggr)^2+\{u,v\}\\
&=4\cdot \frac{4u^2+9u+5}{9(1-u^2)^2}=\frac{8v+1}{18v(v-1)^2}\\
&=\frac{1}{2(1-v)^2}+\frac{1}{18v(1-v)}.
\end{align*}
This implies (see \eqref{eq:Schwtriag}) that $s^+_2$ maps the upper $v$-half-plane onto a circular 
arc triangle with angles $\la\pi$, $\mu\pi$, $\nu\pi$, where 
$\la=1$, $\mu=0$, $\nu=1/3$, and can be used to give an alternative proof of Theorem~\ref{thm:s2}.

Note that $J=1/(1-u^2)$ and so
\[
\{J,u\}=\{u^2,u\}=-\frac{3}{2u^2}.
\]
It follows that 
\begin{align*}
\{\tau, u\}&=\{\tau, J\}\cdot \biggl(\frac {dJ}{du}\biggr)^2+\{J,u\}\\
&=\biggl(\frac{4}{9J^2} + \frac{3}{8(1-J)^2}+ 
\frac{23}{72J(1-J)}\biggr)\cdot \frac {4u^2}{(1-u^2)^4}
-\frac{3}{2u^2}\\
&=\frac{5u^2+31} {18(1-u^2)^2}.\end{align*}
We conclude that 
\begin{align}\label{eq:Schws2tau}
\{s^+_2, \tau\}&=(\{s^+_2, u\}-\{\tau,u\})\biggl(\frac{du}{d\tau}\biggr)^2\\
&=-4\pi^2\frac{(u-1)(u+7)}{6(1-u^2)^2}864^2\Delta^2E_4^{-5}\notag\\&=
-\tfrac{1}{6}\pi^2(u-1)(u+7)E_4\notag \\
&=
\tfrac{1}{6}\pi^2(E_4^{3/2}-E_6)(7E_4^{3/2}+E_6)E_4^{-2}.\notag
\end{align}
Note that $\{s^-_2, \tau\}$ is obtained by changing the sign of the root 
$E_4^{3/2}$ here; that is, with a locally consistent choice of this sign 
we have 
 \begin{equation}\label{eq:Schws2-tau}
 \{s^-_2, \tau\}=\tfrac{1}{6}\pi^2(-E_4^{3/2}-E_6)(-7E_4^{3/2}+E_6)E_4^{-2}.
 \end{equation}

\section{Remarks}\label{sec:concl}

1.~While Corollaries \ref{cor:crit2},~\ref{cor:crit4}, \ref{cor:crit6} give good qualitative information on the location of the critical points of the Eisenstein series $E_2$, $E_4$, $E_6$, one can use the connection to the polymorphic functions $s_2$, $s_4$, $s_6$ to obtain somewhat more precise information
about these locations. Here the {\em real locus} of the functions $s_k$ is important. To explain the main idea, we restrict ourselves to $E_4$ and $s_4$. 

  Let $V\in  \mathcal{V}$ with  $\infty\not \in V$. Then $V$ contains a unique critical point $\tau_V\in V$ of $E_4$ according to Corollary~\ref{cor:crit4} and we have $\tau_V\in \inte(V)$ and  $s_4(\tau_V)=\infty$.  There exists a unique element $\varphi\in \overline {\Ga}(2)$
such that $V=\varphi(V_0)$ and so we can represent $\tau_V$ in the form 
$\tau_V= \varphi(\tau_V')$ with   $\tau_V'\in \inte(V_0)$. 
Since $\varphi \circ s_4=s_4\circ \varphi$ (see Proposition~\ref{prop:poly46}), we then have  
\[ \varphi(s_4(\tau_V'))=s_4(\varphi(\tau_V'))=s_4(\tau_V)= \infty,
\]
and so $s_4(\tau_V')=\varphi^{-1}(\infty)$.  Since $\varphi^{-1}\in  \overline {\Ga}(2)\sub \overline{\Ga}$, we have $\varphi^{-1}(\infty)\in \Q\cup\{\infty\}$, but here the values $0,1,\infty$ (corresponding to the vertices of $V_0$) are ruled out, because $\infty\not \in V=\varphi(V_0)$. We see that 
\[s_4(\tau_V')=\varphi^{-1}(\infty)\in \Q\setminus\{0,1\}.\]
In particular ,  
$\tau_V'\in V_0$ lies in the intersection of one of the three sets 
$s_4^{-1}((0,1))$, $s_4^{-1}((1,\infty))$ 
and $s_4^{-1}((-\infty, 0))$ with $\inte(V_0)$. It easily follows from the discussion
in the proof of Corollary~\ref{cor:s4} and Figure~\ref{fig:covs4}
that these intersections form three disjoint smooth curves in $V_0$ joining the vertices 
of $V_0$ asymptotically in pairs. For example, $s_4^{-1}((-\infty, 0))\cap V_0$ is a curve $\mathcal{C}$ that is  contained in the interior of the union  $T_0\cup T_1= U_0$ and   joins $0$ and $\infty$ (more precisely, $0,\infty\in \partial  \mathcal{C}$).  Similar results can be deduced for $E_2$ and $E_6$ 
(essentially, from Figures~\ref{fig:covs2} and~\ref{fig:covs6} by inspection).

These curves were first studied for $E_2$ and $E_4$ by Chen and Lin 
(see  \cite{CL19, CL24}; a graphical representation of the curves for $E_2$
is given in \cite[Figure~3]{CL19}). Very recently, similar curves for $E_6$ were considered in \cite{Ch25}. 

2.\ Our approach to identifying the mapping behavior of $s_k$ for $k=2,4,6$ is admittedly somewhat {\em ad hoc}. It would be interesting to develop a more 
systematic theory that  allows one to study the conformal maps associated with polymorphic functions that are connected with the problem of finding the critical points of modular forms. In \cite{SS12} it is was observed that if $f$ is a modular 
form of weight $k$, then the associated function
\begin{equation}\label{eq:sf}
s_f(\tau)\coloneq \tau+k\frac{f(\tau)}{f'(\tau)}
\end{equation}
defined for $\tau\in \H$
  is polymorphic in the sense that 
  \begin{equation} \label{eq:polysf}
\varphi\circ s_f=s_f\circ \varphi \text{ for } \varphi\in \Ga.
  \end{equation}
  Here the critical points of $f$, that is, the  zeros of $f'$, correspond to the poles of $s_f$. Note that for $f=E_4$ and $f=E_6$ the associated functions $s_f$ are $s_4$ and $s_6$, respectively. 
  
  In order to determine the mapping properties of $s_f$,
  one can try to compute the Schwarzian derivative  $\{s_f, \tau\}$, but, 
  as we remarked, this leads to cumbersome  computations if attempted directly. 
  
  In some cases one can arrive at explicit formulas by using some general properties. Indeed, if follows from  \eqref{eq:polysf} and the chain rule for Schwarzians   that if $S(\tau)=\{s_f,\tau\}$, then 
\[ S\circ \varphi = S\cdot (\varphi')^{-2} \text{ for } \varphi\in \Ga.
\]
In other words, the Schwarzian $S$  transforms like a modular form of weight $4$. In the special case $f=E_4$ and $s_f=s_4$ one can say more here. Namely,  \eqref{eq:ders4} shows that $s_4$ has no critical points in $\H$ which in turn implies that its Schwarzian $S(\tau)=\{s_4, \tau\}$ has no poles in $\H$. 
Therefore, this Schwarzian $S$ is holomorphic in $\H$ and transforms like a modular form of weight $4$. To be a (holomorphic) modular form of weight $4$, the function  $S$ should also have an expansion as a power series in $q$ near $\infty$.
This follows 
from \eqref{eq:s4qser} by an easy calculation; actually, 
\begin{equation}\label{eq:Schs4asym}
S(\tau)=2\pi^2+O(q) \text{ as } \im(\tau)\to +\infty. 
\end{equation}
Since the space of modular forms of weight $4$ is one-dimensional and spanned by $E_4$, we conclude that $S=cE_4$ with some constant $c\in \C$.
By \eqref{eq:Schs4asym} we have $c=2\pi^2$ and we recover \eqref{eq:Schws4tau}. From this one can in turn derive  formula \eqref{eq:s4JSch} which easily leads to a proof of Theorem~\ref{thm:s4}.

A similar analysis can be given for $s_6$. Formula  
\eqref{eq:ders6} shows that each point in $\Ga(\rho)$ is a simple critical point of $s_6$ and hence a double pole of the Schwarzian $S(\tau)=\{s_6, \tau\}$. Essentially, this accounts for the factor $E_4^2$ in the  denominator of the last expression in \eqref{eq:Schws6tau}. The numerator here has to be a  modular form of weight $4+8=12$ and so is equal to an expression of the form  $aE_6^2+bE_4^3$ with $a,b\in \C$.
The coefficients $a$ and $b$  can then be determined by 
comparing the $q$-expansions of  $S\cdot E_4^2$ and $aE_6^2+bE_4^3$ near $\infty$.  This leads to \eqref{eq:Schws6tau}, to \eqref{eq:s6JSch}, and eventually to an alternative of proof of Theorem~\ref{thm:s6}.

\bigskip
{\bf Ethical statement:}
The author states that there is no conflict of interest.


\begin{thebibliography}{XXX}
\bibliographystyle{alpha}



\bibitem[Bi65]{Bie}  L.~Bieberbach, {\em Theorie der gew\"ohnlichen Differentialgleichungen.} 2nd ed., 
Grund\-lehren 
der Mathematischen  Wissen\-schaften 66, Springer, Berlin-G\"ottingen-Heidel\-berg,  1965. 

\bibitem[Bo22]{BZeta}  M.~Bonk, {\em The quasi-periods of the Weierstrass zeta-function}.  Enseign.\ Math.\ 71 (2025), 335--364.

 \bibitem[Bu79]{Bur} R.B.\ Burckel, {\em An Introduction to Classical Complex Analysis.} Birk\-h\"auser, Basel, Stuttgart, 1979.


\bibitem[Ca50]{Car} C.~Carath\'eodory, {\em Funktionentheorie. Zweiter Band.} Birkh\"auser, Basel, 1950.  


\bibitem[Ch85]{Cha} K.~Chandrasekharan, {\em Elliptic Functions.} Grundlehren 
der Mathematischen  Wissen\-schaften 281, Springer, Berlin-Heidelberg,  1985. 

\bibitem[Ch25] {Ch25} Z.~Chen, {\em On the distribution of critical points of  the Eisenstein series $E_6$ and monodromy interpretation}, Preprint, 2025.

\bibitem[CL19] {CL19} Z.~Chen and C.-S.~Lin, {\em Critical points of the classical Eisenstein series of weight two},
J.\ Differential Geom.\ 113 (2019), 189--226. 

\bibitem[CL24] {CL24} Z.~Chen and C.-S.~Lin,  {\em Critical points of the Eisenstein series $E_4$ and application to the spectrum of the Lam\'e operator},  J.\ Spectr.\ Theory 14 (2024), 
 959--990. 
 
 

\bibitem[GO22]{GO22} S.~Gun, and J.~Oesterl\'e,  {\em Critical points of Eisenstein series}, 
Mathematika 68 (2022), 259--298. 

%
%

%
%
%

\bibitem[SS12]{SS12}  H.~Saber and A.~Sebbar, {\em On the critical points of modular forms}, J.\ Number Theory 132 (2012),  1780--1787. 

\bibitem[Schw]{Schw} H.A.~Schwarz, 
{\em Ueber diejenigen F\"alle, in welchen die Gaussische hypergeometrische Reihe eine algebraische Function ihres vierten Elementes darstellt}, J.\ Reine Angew.\ Math.\ 75 (1873), 292--335. 

 
\bibitem[Sch74]{Sch} B.\ Schoeneberg, {\em Elliptic Modular Functions}.   Springer, Berlin-Heidelberg-New York, 1974.

%
%
%

\bibitem[Za08]{Zag} D.~Zagier, {\em Elliptic Modular Forms and Their Applications}, in: The 1-2-3 of modular forms, 
Universitext, Springer, Berlin, 2008, pp.~1--103.





\end{thebibliography}
\end{document}